\documentclass[final, 3p, times]{elsarticle}

\pdfoutput=1

\usepackage{amssymb}
\usepackage{amsmath,amsfonts,amsthm,amssymb}
\usepackage{natbib}
\usepackage{bm}
\usepackage{xcolor}
\usepackage{overpic}

\makeatletter
\@addtoreset{equation}{section}
\makeatother

\newtheorem{proposition}{Proposition}[section]
\newtheorem{corollary}{Corollary}[section] 

\newtheorem{lemma}[proposition]{Lemma}

\theoremstyle{remark}

\newtheorem{remark}{Remark}

\theoremstyle{remark}

\newcommand{\R}{{\mathbb R}}
 
\renewcommand\d{\mathrm{d}}

\newcommand\nvec{\mathbf{n}}

\newcommand\vvec{\mathbf{v}}
\newcommand\rvec{\mathbf{r}}

\newcommand\xhat{\hat{\mathbf{x}}}
\newcommand\yhat{\hat{\mathbf{y}}}
\newcommand\zhat{\hat{\mathbf{z}}}
\newcommand\Qvec{\mathbf{Q}}
\newcommand\Gvec{\mathbf{H}}
\newcommand\Pvec{\mathbf{P}}
\newcommand\Ivec{\mathbf{I}}

\newcommand\pp{\partial}
\newcommand\tr{\mathrm{tr}}

\newcommand{\abs}[1]{\left|#1\right|}

\begin{document}

\begin{frontmatter}



\title{The Well Order Reconstruction Solution for Three-Dimensional Wells, in the Landau-de Gennes theory.}

\author{Giacomo Canevari}
\ead{gcanevari@bcamath.org}
\address{Basque Center for Applied Mathematics, Alameda de Mazarredo 14, 48009 Bilbao, Spain}

\author{Joseph Harris}
\ead{jh981@bath.ac.uk}
\address{Mathematical Sciences, University of Bath, Claverton Down, Bath, BA2 7A9, United Kingdom}

\author{Apala Majumdar}
\ead{a.majumdar@bath.ac.uk}
\address{Mathematical Sciences, University of Bath, Claverton Down, Bath, BA2 7A9, United Kingdom}

\author{Yiwei Wang}
\ead{ywang487@iit.edu}
\address{Department of Applied Mathematics, Illinois Institute of Technology, Chicago, IL 60616, USA}


\begin{abstract}
We study nematic equilibria on three-dimensional square wells, with emphasis on Well Order Reconstruction Solutions (WORS) as a function of the well size, characterized by $\lambda$, and the well height denoted by $\epsilon$. The WORS are distinctive equilibria reported in \cite{kralj2014order} for square domains, without taking the third dimension into account, which have two mutually perpendicular defect lines running along the square diagonals, intersecting at the square centre. We prove the existence of WORS on three-dimensional wells for arbitrary well heights, with (i) natural boundary conditions and (ii) realistic surface energies on the top and bottom well surfaces, along with Dirichlet conditions on the lateral surfaces. Moreover, the WORS is globally stable for $\lambda$ small enough in both cases and unstable as $\lambda$ increases. We numerically compute novel mixed 3D solutions for large $\lambda$ and $\epsilon$ followed by a numerical investigation of the effects of surface anchoring on the WORS, exemplifying the relevance of the WORS solution in a 3D context. 
\end{abstract} 


\end{frontmatter}



\section{Introduction}
Nematic liquid crystals (NLCs) are classical examples of mesophases or intermediate phases of matter between the solid and liquid phases, with a degree of long-range orientational order \cite{dg}. Nematics are directional materials with locally preferred directions of molecular alignment, described as nematic ``directors'' in the literature. The directional nature of nematics makes them highly susceptible to external light and electric fields, making them the working material of choice for the multi-billion dollar liquid crystal display industry \cite{bahadur1984liquid}. Recently, there has been substantial interest in new applications for NLCs in nano-technology, microfluidics, photonics and even security applications \cite{lagerwall2012new}. We build on a batch of papers on NLCs in square wells, originally reported in \cite{tsakonas2007multistable} and followed up in recent years in \cite{luo2012multistability}, \cite{kralj2014order}, \cite{kusumaatmaja2015free}, \cite{lewis2014colloidal}, \cite{Walton2018}, \cite{canevari2017order}, \cite{wang2018order} etc. In \cite{tsakonas2007multistable}, the authors experimentally and numerically study NLC equilibria inside square wells with tangent boundary conditions on lateral surfaces, which means that the nematic molecules on these surfaces preferentially lie in the plane of the surfaces. They study shallow wells and argue that it is sufficient to study the nematic profile on the square cross-section and hence, model NLC equilibria on a square with tangent boundary conditions, which require the nematic directors to be tangent to the square edges creating a necessary mismatch at the square corners. They report two experimentally observed NLC equilibria on micron-sized wells, labelled as the \emph{diagonal} solution for which the nematic director is along a square diagonal and a \emph{rotated} solution for which the nematic director rotated by $\pi$ radians between a pair of parallel edges. They further model this system within a reduced two-dimensional continuum Landau-de Gennes (LdG) approach and recover the diagonal and rotated solutions numerically. The reduction from a 3D well to a 2D square domain can be rigorously justified using $\Gamma$-convergence \cite{wang2018order}.

In \cite{kralj2014order}, the authors study the effects of square size on NLC equilibria for this model problem with tangent boundary conditions. They measure square size in units of a material-dependent length scale - the biaxial correlation length, which is typically in the nanometer regime. For micron-sized squares, the authors recover the diagonal and rotated solutions within a continuum LdG approach as before. As they reduce the square size, particularly from the micron to the nano-scale, they find a unique Well Order Reconstruction Solution (WORS) for squares smaller than a certain critical size, which in turn depends on the material constants and temperature. The WORS is an interesting NLC equilibria for two reasons - (i) it partitions the square into four quadrants and the nematic director is approximately constant in each quadrant according to the tangent condition on the corresponding edge and (ii) the WORS has a defect line along each square diagonal and the two mutually perpendicular defect lines intersect at the square centre, yielding the quadrant structure. Indeed, we speculate that this distinctive defect line could be a special optical feature of the WORS, if experimentally realised. The WORS has been analysed in \cite{canevari2017order} and \cite{wang2018order}, in terms of solutions of the Allen-Cahn equation and it is rigorously proven that the WORS is globally stable for sufficiently small squares i.e. for nano-scale geometries. Recent work shows that the WORS is also observable in molecular simulations and is hence, not a continuum artefact.

A potential criticism is that the WORS is an artefact of the 2D square domain and is hence, not relevant for 3D scenarios. In this paper, we address the important question - does the WORS survive in a three-dimensional square box? As proven in \cite{wang2018order}, the WORS does survive in the \emph{thin film limit} but can we observe the WORS for square wells with a finite height? The answer is affirmative and we identify two physically relevant 3D scenarios for which the WORS exists, for all values of the well height and for all temperatures below the nematic supercooling temperature i.e. for temperatures that favour a bulk ordered nematic phase. The paper is organised as follows. In Section~\ref{sub:ldg}, we review the LdG theory for NLCs and introduce the domain and the boundary conditions in Section~\ref{sect:Omega}. Our analytical results are restricted to Dirichlet tangent conditions for the nematic directors on the lateral surfaces of the well, phrased in the LdG framework. In Section~\ref{sect:natural}, we work with 3D wells that have natural boundary conditions on the top and bottom surfaces and study the existence, stability and qualitative properties of the WORS as a special case of a more general family of LdG equilibria; we believe these results to be of general interest. In Section~\ref{sect:surface}, we work with 3D wells that have realistic surface energies that favour planar boundary conditions on the top and bottom and again prove the existence of the WORS for arbitrary well heights and low temperatures, accompanied by interesting companion results for surface energy. In Section~\ref{sec:numerics}, we perform a detailed numerical study of the 3D LdG model on 3D square wells. We discover novel mixed 3D solutions, that interpolate between different diagonal solutions, when the WORS is unstable. Further, we numerically study the effect of surface anchoring on the lateral surfaces on the stability of the WORS, in contrast to the analysis which is restricted to Dirichlet conditions on the lateral surfaces. The WORS ceases to exist as we weaken the tangential boundary conditions on the lateral surfaces; this is expected from \cite{kralj2014order} since the tangent conditions naturally induce the symmetry of the WORS in severe confinement. Our numerical results yield quantitative estimates for the existence of the WORS as a function of the anchoring strength on the lateral surfaces and these estimates can be of value in further work. We summarise our conclusions in Section~\ref{sec:summary}.

\section{Preliminaries}

\subsection{The Landau-de Gennes model}
\label{sub:ldg}

We work with the Landau-de Gennes (LdG) theory for nematic liquid crystals.
The LdG theory is a powerful continuum theory for nematic liquid crystals
and describes the nematic state by a macroscopic order parameter  --- 
the LdG $\Qvec$-tensor, which is a symmetric traceless $3\times 3$ matrix i.e. 
\[
 \Qvec\in S_0 := \left\{ \Qvec\in \mathbb{M}^{3\times 3}\colon
 Q_{ij} = Q_{ji}, \ Q_{ii} = 0 \right\}.
\]
A $\Qvec$-tensor is said to be (i) isotropic if~$\Qvec=0$, (ii) uniaxial 
if $\Qvec$ has a pair of degenerate non-zero eigenvalues and can be written in the form
\[
\Qvec = s\left(\nvec \otimes \nvec - \frac{\mathbf{I}}{3} \right)
\]
where $\nvec$ is the eigenvector with the non-degenerate eigenvalue  and 
(iii) biaxial if~$\Qvec$ has three distinct eigenvalues.

We assume that the domain is a three-dimensional well, filled with nematic liquid crystals, 
\[
 V := \Omega\times (0, \, h),
\]
where $\Omega\subseteq\R^2$ is the two-dimensional cross-section of the well
(more precisely, a truncated square, as described in 
Section~\ref{sect:Omega}) and~$h$ is the well height
\cite{tsakonas2007multistable, kralj2014order}. Let~$\Gamma$
be the union of the top and bottom plates, that is,
\[
 \Gamma := \Omega\times\{0, \, h\}.
\]

In the absence of surface anchoring energy, we work with a simple form of the LdG energy given by~\cite{dg}
\begin{equation} \label{eq:dim-energy} 
 \mathcal{F}_\lambda[\Qvec] :=  \int_{V} \left(\frac{L}{2} \left|\nabla\Qvec \right|^2
 + f_b(\Qvec) \right) \d V
\end{equation}
The term~$\abs{\nabla\Qvec}^2 := \partial_k Q_{ij} \partial_k Q_{ij}$,
for $i$, $j$, $k = 1, \, 2, \, 3$, is an elastic energy density
which penalises spatial inhomogeneities and $L>0$
is a material-dependent elastic constant.
The thermotropic bulk potential, $f_b$, is given by
\begin{equation} \label{eq:3}
 f_b(\Qvec) := \frac{A}{2} \tr\Qvec^2 - \frac{B}{3} \tr\Qvec^3 + \frac{C}{4}\left(\tr\Qvec^2 \right)^2.
\end{equation}
The variable~$A = \alpha (T - T^*)$ is the re-scaled temperature, 
$\alpha$, $B$, $C>0$ are material-dependent constants
and~$T^*$ is the characteristic supercooling 
temperature~\cite{dg,newtonmottram}.
It is well-known that all stationary points of $f_b$ are either uniaxial or isotropic~\cite{dg,newtonmottram,ejam2010}.
The re-scaled temperature~$A$ has three characteristic values: 
(i)~$A=0$, below which the isotropic phase $\Qvec=0$ loses stability,
(ii) the nematic-isotropic transition temperature, 
$A={B^2}/{27 C}$, at which $f_b$ is minimized by the isotropic
phase and a continuum of uniaxial states with $s=s_+ ={B}/{3C}$ and
$\nvec$ arbitrary, and (iii) the nematic superheating temperature, 
$A = {B^2}/{24 C}$, above which the isotropic state is the
unique critical point of $f_b$.
For a given $A<0$, let
$\mathcal{N} := \left\{ \Qvec \in S_0\colon 
\Qvec = s_+ \left(\nvec\otimes \nvec - \Ivec/3 \right) \right\}$ 
denote the set of minimizers of the bulk potential, $f_b$, with
\begin{equation} \label{eq:s+}
 s_+ := \frac{B + \sqrt{B^2 + 24|A| C}}{4C}
\end{equation}
and~$\nvec \in S^2$ arbitrary. 

We non-dimensionalize the system using a change of variables, 
$\bar{\rvec} = \rvec/ \lambda$,
where $\lambda$ is a characteristic length scale of the cross-section $\Omega$.
The rescaled domain and the rescaled top and bottom surfaces become
\begin{equation}
\overline{V} := \overline{\Omega} \times (0, \, \epsilon),
\qquad \overline{\Gamma} := \overline{\Omega} \times \{0, \, \epsilon\}
\end{equation}
where $\overline{\Omega}$ is the rescaled two-dimensional
domain and~$\epsilon := h / \lambda$. 
The re-scaled LdG energy functional is 
\begin{equation} \label{eq:LdG}
 \overline{\mathcal{F}_\lambda}[\Qvec] := \frac{\mathcal{F}_\lambda[\Qvec]}{L \lambda} =
 \int_{\overline{V}} \left(\frac{1}{2}\left| \overline{\nabla} \Qvec \right|^2 
 + \frac{\lambda^2}{L} f_b(\Qvec) \right) \, \overline{\d V}.
\end{equation}
In~\eqref{eq:LdG}, 
$\overline{\nabla}$ is the gradient with respect to
the re-scaled spatial coordinates, $\overline{\d V}$ 
is the re-scaled volume element and~$\overline{\d S}$ is the re-scaled area element.
In what follows, we drop the \emph{bars} and all statements 
are to be understood in terms of the re-scaled variables.

Critical points of ~\eqref{eq:LdG} satisfy the
Euler-Lagrange system of partial differential equations 
\begin{equation} \label{eq:EL}
 \Delta \Qvec = \frac{\lambda^2}{L} 
 \left\{ A\Qvec - B\left(\Qvec\Qvec - \frac{\Ivec}{3}|\Qvec|^2 \right)
 + C|\Qvec|^2 \Qvec \right\} \! .
\end{equation}

\subsection{The 2D domain, $\Omega$,
and the boundary conditions on the lateral surfaces}
\label{sect:Omega}

Let $\Omega$ be our two-dimensional (2D) cross-section which 
is a truncated square with diagonals along the coordinate axes:
\begin{equation} \label{Omega}
 \Omega := \left\{(x, \, y)\in\R^2\colon
 |x| < 1 - \eta, \ |y| < 1 - \eta, 
 \ |x+y| < 1, \ |x-y| < 1 \right\}
\end{equation}
(see Figure~\ref{fig:square}).
Here~$\eta\in (0, \, 1)$ is a small, but fixed parameter.
The boundary, $\partial\Omega$, consists of four ``long'' edges~$C_1$,\ldots, $C_4$,
parallel to the lines ~$y = x$ and~$y = -x$,
and four ``short'' edges~$S_1, \, \ldots, \, S_4$, of length~$2\eta$, 
parallel to the $x$ and $y$-axes respectively.
The four long edges~$C_i$ are labeled counterclockwise and $C_1$ 
is the edge contained in the first quadrant, i.e.
\[
 C_1 := \left\{(x, \, y)\in\R^2\colon x + y = 1, \ \eta \leq x \leq 1 - \eta \right\} \! .
\]
The short edges~$S_i$ are introduced to remove the sharp square vertices.
They are also labeled counterclockwise and
$S_1 := \{(1 - \eta, \, y)\in\R^2\colon |y|\leq \eta\}$.

\begin{figure}[t]
	\centering
 	\includegraphics[height=5cm]{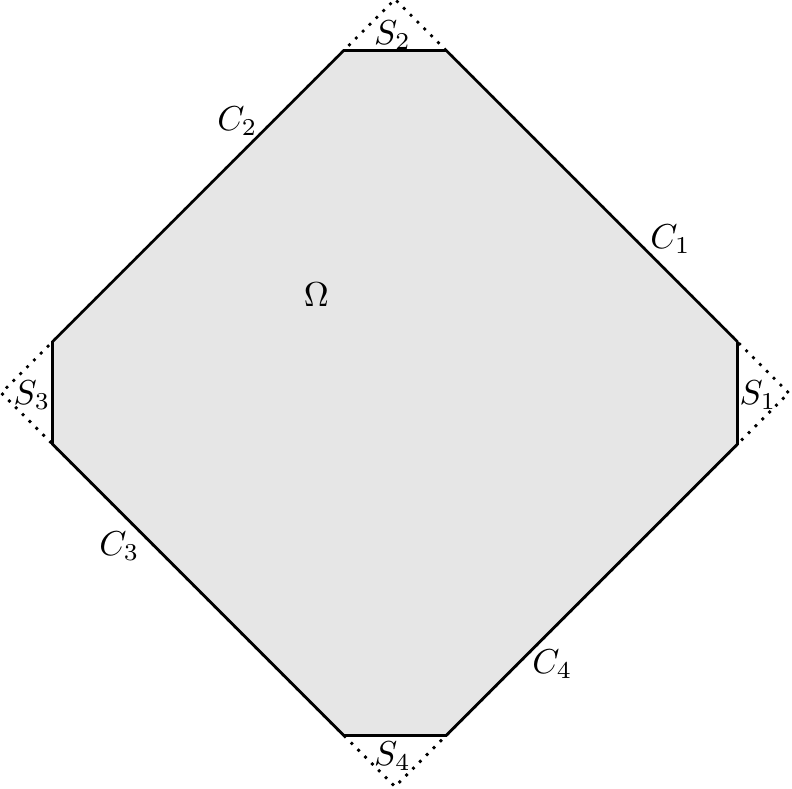} 
	\caption{The `truncated square'~$\Omega$. }
	\label{fig:square}
\end{figure}

We impose Dirichlet conditions
on the lateral surfaces of the well,
$\partial\Omega\times (0, \, \epsilon)$:
\begin{equation} \label{eq:bc-Dir}
 \Qvec = \Qvec_{\mathrm{b}} \quad \textrm{on } 
 \partial V\setminus\Gamma = \partial\Omega\times (0, \, \epsilon),
\end{equation}
where the boundary datum~$\Qvec_{\mathrm{b}}$ is independent of the $z$-variable,
$\partial_z\Qvec_{\mathrm{b}}\equiv 0$.
Following the literature on planar multistable
nematic systems~\cite{tsakonas2007multistable, luo2012multistability, kralj2014order}, we impose 
\emph{tangent} uniaxial Dirichlet conditions on the long edges, $C_1, \, \ldots, \, C_4$:
\begin{equation} \label{eq:bc1}
 \Qvec_{\mathrm{b}}(x, \, y, \, z) := \begin{cases} 
  s_+\left( \nvec_1 \otimes \nvec_1 - \Ivec/3 \right)  & \textrm{if } (x, \, y)\in C_1 \cup C_3 \\
  s_+\left( \nvec_2 \otimes \nvec_2 - \Ivec/3 \right)  & \textrm{if } (x, \, y)\in C_2 \cup C_4;
 \end{cases}
\end{equation}
where $s_+$ is defined in (\ref{eq:s+})
and
\begin{equation} \label{eq:n12}
 \nvec_1 := \frac{1}{\sqrt{2}}\left(-1, \, 1, \, 0 \right), 
 \qquad  \nvec_2 := \frac{1}{\sqrt{2}}\left(1, \, 1, \, 0 \right).
\end{equation}
We prescribe Dirichlet conditions on the short edges too, in terms of a function,
$g_0\colon [-\eta, \, \eta]\to [-s_+/2, s_+/2]$, 
chosen to eliminate discontinuities of the tangent Dirichlet boundary condition e.g.
\[
 g_0(s) := \frac{s_+}{2\eta} s, \qquad 
 \textrm{for } -\eta \leq s \leq\eta,
\]
but the choice of $g_0$ does not affect qualitative predictions or numerical results.
We define
\begin{equation} \label{eq:bc2}
 \Qvec_{\mathrm{b}}(x, \, y, \, z) := \begin{cases}
  g_0(y) \left(\nvec_1 \otimes \nvec_1 - \nvec_2\otimes \nvec_2 \right)
  - \dfrac{s_+}{6}\left(2 \zhat\otimes\zhat
  - \nvec_1\otimes \nvec_1 - \nvec_2\otimes \nvec_2 \right)  
  & \textrm{if } (x, \, y)\in S_1\cup S_3, \\
  g_0(x)\left(\nvec_1 \otimes \nvec_1 - \nvec_2\otimes \nvec_2 \right)
  - \dfrac{s_+}{6}\left(2 \zhat\otimes\zhat
  - \nvec_1\otimes \nvec_1 - \nvec_2\otimes \nvec_2 \right) 
  & \textrm{if } (x, \, y)\in S_2\cup S_4.
 \end{cases}
\end{equation}
Given the Dirichlet conditions~\eqref{eq:bc-Dir},
 our admissible class of~$\Qvec$-tensors is
\begin{equation} \label{eq:admissible}
 \mathcal{B} := \left\{ \Qvec \in W^{1,2}(V, \, S_0)\colon
 \Qvec = \Qvec_{\mathrm{b}} \ \textrm{on } 
 \partial\Omega\times(0, \, \epsilon) \right\}.
\end{equation}

\section{The Well Order Reconstruction Solution (WORS) in a three-dimensional (3D) context and related results}
\label{sec:Wors_3D}

In \cite{kralj2014order}, the authors numerically report the Well Order Reconstruction solution (WORS) on the 2D domain, $\Omega$, with the Dirichlet conditions (\ref{eq:bc1}); which is further analysed in \cite{canevari2017order}. At a fixed temperature $A = -\frac{B^2}{3C}$, the WORS corresponds to a classical solution of the Euler-Lagrange equations (\ref{eq:EL}) of the form
\begin{equation}
\label{eq:wors}
\Qvec_{WORS} = q \left( \nvec_1 \otimes \nvec_1 - \nvec_2 \otimes \nvec_2 \right) - \frac{B}{6C} \left(2 \zhat \otimes \zhat -\nvec_1 \otimes \nvec_1 -  \nvec_2\otimes \nvec_2 \right) 
\end{equation}
with a single degree of freedom, $q\colon\Omega\to \mathbb{R}$ which satisfies the Allen-Cahn equation and has the following symmetry properties
\[
q \left(x, 0 \right) =  q\left( 0, y \right) = 0,  \qquad  xy \, q(x,y) \geq 0.
\]
The WORS has a constant set of eigenvectors, unit-vectors $\nvec_1, \nvec_2$ 
defined by~\eqref{eq:n12} and the coordinate unit-vector $\zhat$, and very importantly, has two mutually perpendicular defect lines, along the square diagonals, intersecting at the square centre, described by the nodal lines of $q$ above. These are defect lines in the sense that $\Qvec_{WORS}$ is uniaxial along these diagonal lines with negative order parameter, which physically implies that the nematic molecules lie in the plane of the square without a preferred in-plane direction along the defect lines i.e. they are locally disordered in the square plane along the defect lines.
In \cite{canevari2017order}, the authors prove that the WORS is globally stable for $\lambda$ small enough and unstable for $\lambda$ large enough. The analysis in \cite{canevari2017order} is restricted to a special temperature but numerics show that the WORS exists for all $A<0$ with the diagonal defect lines, and the eigenvalue associated with  $\zhat$ is negative for all $A<0$. The negative eigenvalue (associated with $\zhat$) implies that nematic molecules lie in the $(x,y)$-plane and for non-zero $q$, there is a locally defined nematic director in the $(x,y)$-plane. In the next sections, we study the relevance of the WORS in 3D contexts i.e. does the WORS survive in 3D scenarios and what can be said about its qualitative properties?

\subsection{Natural boundary conditions on the top and bottom plates}
\label{sect:natural}


In this section, we study a special class of LdG critical points, including the WORS, with natural boundary conditions
on the top and bottom plates.
Minimizers of~$\mathcal{F}_\lambda$ (see~\eqref{eq:LdG}),
in the admissible class~$\mathcal{B}$ in~\eqref{eq:admissible}
satisfy the Euler-Lagrange system~\eqref{eq:EL},
subject to the Dirichlet boundary conditions~\eqref{eq:bc-Dir}
along with \emph{natural} or Neumann boundary conditions
on the top and bottom plates i.e.
\begin{equation} \label{eq:bc-nat}
 \partial_z\Qvec = 0 \qquad \textrm{on } \Gamma = \Omega\times\{0, \, \epsilon\} .
\end{equation}
Throughout this section,
we will treat~$L>0$, $B$, $C>0$ as fixed parameters,
while~$\lambda$ and~$A$ may vary.

\begin{proposition} \label{prop:2D}
  For any~$\lambda>0$ and~ $A < 0$,
 there exist minimizers~$\Qvec$ of ~$\mathcal{F}_\lambda$,
in~\eqref{eq:LdG}, in the admissible class~$\mathcal{B}$, 
 (see~\eqref{eq:admissible}). Moreover, minimizers are independent of the $z$-variable,
 that is~$\partial_z\Qvec=0$ on~$V$, and they minimize the 2D functional
 \begin{equation} \label{eq:2D-LdG}
  I[\Qvec] := \int_{\Omega} 
  \left(\frac{1}{2}\left|\nabla\Qvec \right|^2 
  + \frac{\lambda^2}{L} f_b(\Qvec) \right) \d S
 \end{equation}
 in the class
 \begin{equation} \label{eq:2D-admissible}
  \mathcal{B}^\prime := \left\{ \Qvec \in W^{1,2}(\Omega, \, S_0)\colon
  \Qvec = \Qvec_{\mathrm{b}} \ \textrm{on } \partial\Omega \right\} \! .
 \end{equation}
\end{proposition}

This result can be proved, e.g., as in~\cite[Theorem~0]{Bethuel1992}.
\newline \emph{Any} ($z$-independent) critical point of the functional~$I$
in the admissible class~$\mathcal{B}^\prime$ is also a solution of the 
three-dimensional system~\eqref{eq:EL},
subject to the boundary conditions~\eqref{eq:bc-Dir} and~\eqref{eq:bc-nat}.
This necessarily implies that the WORS is a LdG critical point on 3D wells V, of arbitrary height $\epsilon$, with natural boundary conditions on  $\Gamma$.
Therefore, in the rest of this section, we restrict ourselves to a 2D problem ---  the analysis of critical points of~$I$
in ~$\mathcal{B}^\prime$. 

Our first result concerns the existence of a Well Order Reconstruction Solution (WORS)-like solution for all $A<0$ as proven below.  

\begin{proposition} \label{prop:WORS}
 For any~$\lambda>0$ and~$A<0$, there exists a solution~$(q_1^{WORS}, \, q_3^{WORS})$
 of the system
 \begin{equation} \label{PDE-WORS}
   \left\{\begin{aligned}
   \Delta q_1 &= \frac{\lambda^2}{L} q_1 \left\{A + 2B q_3 + 2C\left(q_1^2 + 3q_3^2 \right)\right\} \\
   \Delta q_3 &= \frac{\lambda^2}{L} q_3 \left\{A - Bq_3 + 2C\left(q_1^2 + 3q_3^2\right)\right\} 
   + \frac{\lambda^2B}{3L}q_1^2 
  \end{aligned} \right.
 \end{equation}
 subject to the boundary conditions
 \begin{equation} 
 q_1(x, \, y) = q_{1\mathrm{b}} (x, \, y) :=
 \begin{cases}
  -{s_+}/{2}  & \textrm{if  } (x, \, y)\in C_1\cup C_3 \\
  {s_+}/{2}   & \textrm{if  } (x, \, y)\in C_2\cup C_4 \\
  g_0(y)        & \textrm{if  } (x, \, y)\in S_1\cup S_3 \\
  g_0(x)        & \textrm{if  } (x, \, y)\in S_2\cup S_4.
 \end{cases}
\end{equation} 
and $q_3 = -{s_+}/{6}$ on $\partial\Omega$, that satisfies
\begin{equation} \label{WORS}
 xy\, q_1(x, \, y) \geq 0, \quad 
 q_3(x, \, y) < 0
 \qquad \textrm{for any } (x, \, y)\in\Omega.
\end{equation}
Then 
\begin{equation} \label{Q-WORS2D}
 \begin{split}
  \Qvec(x, \, y) &= q_1^{WORS}(x, \, y) \left(\nvec_1 \otimes \nvec_1 - \nvec_2\otimes \nvec_2 \right)
  + q_3^{WORS}(x, \, y) \left(2 \zhat\otimes\zhat
  - \nvec_1\otimes \nvec_1 - \nvec_2\otimes \nvec_2 \right) \! ,
 \end{split}
\end{equation}
is a WORS solution of the Euler-Lagrange system \eqref{eq:EL} on~$V$, subject to the Dirichlet conditions \eqref{eq:bc-Dir} and natural boundary conditions on~$\Gamma$.
\end{proposition}
\begin{proof}
 We follow the approach in~\cite{canevari2017order}. Let
 \[
  \Omega_+ := \left\{(x, \, y)\in \Omega\colon x>0, \ y > 0\right\} 
 \]
 be the portion of~$\Omega$ that is contained in the first quadrant.
 For solutions of the form (\ref{Q-WORS2D}), the LdG energy reduces to
 \begin{equation} \label{eq:LdG-WORS2}
  \begin{split}
   G[q_1, \, q_3] &:= \int_{\Omega_+} \left\{
   \abs{\nabla q_1}^2 + 3 \abs{\nabla q_3}^2 
   + \frac{\lambda^2}{L} \left(A (q_1^2 + 3 q_3^2 )
   + 2 B q_3 q_1^2 - 2 B q_3^3 + C(q_1^2 + 3 q_3^2)^2\right)\right\} \d S.
  \end{split}
 \end{equation}
 We minimize~$G$ in the admissible class
 \[
  \mathcal{G} := \left\{(q_1, \, q_3)\in H^1(\Omega_+)^2\colon q_3 \leq 0 \textrm{ in } \Omega_+, \ \
  q_1 = q_{1\mathrm{b}} \textrm{ on } \partial\Omega\cap \overline{\Omega_+}, \ \
  q_3 = -s_+/6 \textrm{ on } \partial\Omega\cap \overline{\Omega_+}, \ \
  q_1 = 0 \textrm{ on } \partial\Omega_+\setminus\partial\Omega\right\} \!.
 \]
 We impose no boundary conditions for~$q_3$ on~$\partial\Omega_+\setminus\partial\Omega$.
 The function~$q_{1\mathrm{b}}$ is compatible with the Dirichlet conditions~\eqref{eq:bc-Dir}.
 The class~$\mathcal{G}$ is closed and convex. Therefore,
 a routine application of the direct method of the calculus of variations shows 
 that a minimizer~$(q_1^{WORS}, \, q_3^{WORS})$ exists. Moreover,
 we can assume without loss of generality that $q_1^{WORS}\geq 0$ on~$\Omega_+$
 --- otherwise, we replace $q_1^{WORS}$ with~$|q^{WORS}_1|$ and note
 that $G[q_1^{WORS}, \, q_3^{WORS}] = G[|q_1^{WORS}|, \, q_3^{WORS}]$.
 
 We now claim that~$q_3^{WORS}<0$ in~$\Omega_+$. To prove this,
 let us consider a function~$\varphi\in H^1(\Omega_+)$ such that~$\varphi\geq 0$
 in~$\Omega_+$ and~$\varphi = 0$ on~$\partial\Omega\cap \overline{\Omega_+}$.
 For sufficiently small~$t\geq 0$, the function
 $q_3^t := q_3^{WORS} - t\varphi$ is an admissible perturbation of~$q_3^{WORS}$,
 and hence, we have
 \[
  \frac{\d}{\d t}_{|t = 0} G[q_1^{WORS}, \, q_3^t] \geq 0
 \]
 because $(q_1^{WORS}, \, q_3^{WORS})$ is a minimizer. The derivative on the
 left-hand side can be computed explicitly and we obtain
 \begin{equation} \label{subsol}
  \int_{\Omega_+} \left\{ -6 \nabla q_3^{WORS}\cdot\nabla\varphi
   - \frac{\lambda^2}{L} f(q_1^{WORS}, \, q_3^{WORS}) \, \varphi\right\} \d S \geq 0
 \end{equation}
 where
 \begin{equation} \label{f}
  f(q_1, \, q_3) := 6 q_3 \left(A - Bq_3 + 6C q_3^2\right) 
   + 2(B + 6C q_3)q_1^2
 \end{equation}
 For~$|q_3|$ sufficiently small, we have
 $A - Bq_3 + 6C q_3^2 < 0$ and~$B + 6C q_3>0$, because~$A<0$ and~$B>0$.
 Therefore, there exists~$\delta>0$ (depending only on~$A$, $B$, $C$)
 such that
 \begin{equation} \label{signf}
  f(q_1, \, q_3) > 0 \qquad \textrm{for any } q_1\in\R  
  \textrm{ and } q_3\in [-\delta, \, 0].
 \end{equation}
 Now, we define
 \begin{equation} \label{varphi-subsol}
  \varphi := \begin{cases}
              q_3^{WORS} + \delta  & \textrm{where } q_3^{WORS} > -\delta \\
              0                    & \textrm{where } q_3^{WORS}\leq -\delta.
             \end{cases} 
 \end{equation}
 By taking~$\delta < s_+/6$, we can make sure that~$\varphi = 0$
 on~$\partial\Omega\cap \overline{\Omega_+}$.
 Then, we can substitute~$\varphi$ into~\eqref{subsol} and obtain
 \begin{equation*} 
  \int_{\{q_3^{WORS}>-\delta\}} \left\{ 6 \abs{\nabla q_3^{WORS}}^2
   + \frac{\lambda^2}{L} 
   f(q_1^{WORS}, \, q_3^{WORS}) \, (q_3^{WORS} + \delta)\right\} 
   \d S \leq 0.
 \end{equation*}
 Due to~\eqref{signf}, we conclude that
 $q_3^{WORS}\leq -\delta < 0$ in~$\Omega_+$.
 In particular, $(q_1^{WORS}, \, q_3^{WORS})$ lies in the interior of the
 admissible set~$\mathcal{G}$ and hence, it solves the Euler-Lagrange 
 system~\eqref{PDE-WORS} for the functional~$G$, together with the natural boundary condition
 $\partial_\nu q_3=0$ on~$\partial \Omega_+ \setminus\partial \Omega$.
 
 We extend~$(q_1^{WORS}, \, q_3^{WORS})$ to the whole of~$\Omega$
 by reflections about the planes~$\{x = 0\}$ and~$\{y = 0\}$:
 \[
  q_1^{WORS}(x, \, y) := \mathrm{sign}(xy)\,  q_1^{WORS}(\abs{x}, \, \abs{y}), \qquad
  q_3^{WORS}(x, \, y) := q_3^{WORS}(\abs{x}, \, \abs{y})
 \]
 for any~$(x, \, y)\in \Omega\setminus\overline{\Omega_+}$.
 An argument based on elliptic regularity,
 as in~\cite[Theorem~3]{dangfifepeletier},
 shows that $(q_1^{WORS}, \, q_3^{WORS})$ is a solution 
 of~\eqref{PDE-WORS} on~$\Omega$, satisfies the boundary conditions
 and~\eqref{WORS}, by construction.
\end{proof}

%


The WORS is a special case of critical points of~\eqref{eq:2D-LdG},
that have $\zhat$ as a constant eigenvector and can be completely described by three
degrees of freedom, i.e. they can be written as
\begin{equation} \label{Q-ansatz}
 \begin{split}
  \Qvec(x, \, y) &= q_1(x, \, y) \left(\nvec_1 \otimes \nvec_1 - \nvec_2\otimes \nvec_2 \right)
  + q_2(x, \, y) \left(\nvec_1 \otimes \nvec_2 + \nvec_2\otimes \nvec_1 \right) \\
  &\qquad\qquad +
  q_3(x, \, y) \left(2 \zhat\otimes\zhat
  - \nvec_1\otimes \nvec_1 - \nvec_2\otimes \nvec_2 \right) \! ,
 \end{split}
\end{equation}
where~$q_1$, $q_2$, $q_3$ are scalar functions and~$\nvec_1$,
$\nvec_2$ are given by~\eqref{eq:n12}.
For solutions of the form~\eqref{Q-ansatz}, the LdG Euler-Lagrange 
system~\eqref{eq:LdG} reduces to
\begin{equation} \label{eq:EL123}
 \left\{\begin{aligned}
   \Delta q_1 &= \frac{\lambda^2}{L} q_1 \left\{A + 2B q_3 + 2C\left(q_1^2 + q_2^2 + 3q_3^2 \right)\right\} \\
   \Delta q_2 &= \frac{\lambda^2}{L} q_2 \left\{A + 2B q_3 + 2C\left(q_1^2 + q_2^2 + 3q_3^2 \right)\right\} \\
   \Delta q_3 &= \frac{\lambda^2}{L} q_3 \left\{A - Bq_3 + 2C\left(q_1^2 + q_2^2 + 3q_3^2\right)\right\} 
   + \frac{\lambda^2B}{3L}\left( q_1^2 + q_2^2 \right),
  \end{aligned} \right.
\end{equation}
This is precisely the Euler-Lagrange system associated with the functional
\begin{equation} \label{eq:J}
  J[q_1, \, q_2, \, q_3]: = \int_{\Omega} \left( 
   \abs{\nabla q_1}^2 + \abs{\nabla q_2}^2 + 3 \abs{\nabla q_3}^2 
   + \frac{\lambda^2}{L} F(q_1, \, q_2, \, q_3)\right) \d S, 
\end{equation}
where~$F$ is the polynomial potential given by
\begin{equation} \label{eq:F}
 F(q_1, \, q_2, \, q_3) :=  A \left( q_1^2 + q_2^2 + 3 q_3^2 \right)
   + 2 B q_3 \left( q_1^2 + q_2^2\right) - 2 B q_3^3 + C\left( q_1^2 + q_2^2 + 3 q_3^2 \right)^2.
\end{equation}
The Dirichlet boundary condition~\eqref{eq:bc-Dir} for~$\Qvec$
translates into boundary conditions for~$q_1$, $q_2$ and~$q_3$:
\begin{equation} \label{eq:BC123}
 q_1 = q_{1\mathrm{b}}, \quad q_2 = 0, \quad q_3 = -{s_+}/{6} 
 \qquad \textrm{on  } \partial\Omega,
\end{equation}
where the function~$q_{1\mathrm{b}}$ is defined by
\begin{equation} \label{eq:BC1}
 q_{1\mathrm{b}} (x, \, y) :=
 \begin{cases}
  -{s_+}/{2}  & \textrm{if  } (x, \, y)\in C_1\cup C_3 \\
  {s_+}/{2}   & \textrm{if  } (x, \, y)\in C_2\cup C_4 \\
  g_0(y)        & \textrm{if  } (x, \, y)\in S_1\cup S_3 \\
  g_0(x)        & \textrm{if  } (x, \, y)\in S_2\cup S_4.
 \end{cases}
\end{equation}
By adapting the methods in \cite{Ignatetal},
we can construct solutions~$(q_1, \, q_2, \, q_3)$
to the system~\eqref{eq:EL123}, subject to the boundary 
conditions~\eqref{eq:BC123}, that
satisfy~$q_3<0$ in~$\Omega$ and are locally stable. The WORS is a specific example of such a solution with a constant eigenframe and two degrees of freedom. In fact, the results in \cite{wang2018order} show that the WORS loses stability with respect to solutions of the form \eqref{Q-ansatz}, with $q_2 \neq 0$, as $\lambda$ increases.



We say that a solution~$(q_1, \, q_2, \, q_3)$ of~\eqref{eq:EL123} is
locally stable if, for any perturbations $\varphi_1$, $\varphi_2$,
$\varphi_3\in C^1_{\mathrm{c}}(\Omega)$, there holds
\begin{equation} \label{eq:second_var}
 \delta^2 J(q_1, \, q_2, \, q_3) [\varphi_1, \, \varphi_2, \, \varphi_3] := 
 \frac{\d^2}{\d t^2}_{|t = 0} J[q_1 + t\varphi_1, \, q_2 + t\varphi_2, \, q_3 + t\varphi_3]\geq 0.
\end{equation}
Given a locally stable solution~$(q_1, \, q_2, \, q_3)$ of~\eqref{eq:EL123},
the corresponding $\Qvec$-tensor, defined by~\eqref{Q-ansatz},
is a solution of~\eqref{eq:EL} and is locally stable in the restricted 
class of $\Qvec$-tensors that have $\zhat$ as a constant eigenvector.

\begin{proposition} \label{prop:stable}
 For any~$A<0$, there exists a solution~$(q_{1,*}, \, q_{2,*}, \, q_{3,*})$
 of the system~\eqref{eq:EL123}, subject to the boundary
 conditions~\eqref{eq:BC123}, that is locally stable
 and has~$q_{3,*}<0$ everywhere in~$\Omega$.
\end{proposition}
\begin{proof}
 Let~$\mathcal{A}$ be the set of triplets
 $(q_1, \, q_2, \, q_3)\in H^1(\Omega)^3$
 that satisfy the boundary conditions~\eqref{eq:BC123}. 
 The boundary data are piecewise of class~$C^1$, so the class~$\mathcal{A}$ is non-empty.
 Moreover, $\mathcal{A}$ is convex and closed in~$H^1(\Omega)^3$.
 
 To construct solutions with negative~$q_3$, we first introduce the class
 \begin{equation} \label{A-}
   \mathcal{A}^- := \left\{(q_1, \, q_2, \, q_3)\in\mathcal{A}
   \colon q_3\leq 0 \ \textrm{ a.e. on }\Omega \right\} \!.
 \end{equation}
 The class~$\mathcal{A}^-$ is a non-empty, convex and closed subset of~$H^1(\Omega)^3$.
 A routine application of the direct method in the calculus of variations 
 shows that the functional~$J$, defined by~\eqref{eq:J},
 has a minimizer~$(q_{1,*}, \, q_{2,*}, \, q_{3,*})$ in the class~$\mathcal{A}^-$.
 To complete the proof, it suffices to show that~$q_{3,*}\leq-\delta$ in~$\Omega$,
 for some strictly positive constant~$\delta$. Once this inequality is proved,
 it will follow that~$(q_{1,*}, \, q_{2,*}, \, q_{3,*})$ lies in the interior of~$\mathcal{A}^-$
 and hence, it is a locally stable solution of the Euler-Lagrange system~\eqref{eq:EL123}.
 
 To prove that~$q_{3,*}\leq-\delta$ in~$\Omega$, we follow the same method
 as in Proposition~\ref{prop:WORS}. Let~$\varphi\in H^1(\Omega_+)$ be
 such that~$\varphi\geq 0$ in~$\Omega$ and~$\varphi = 0$ on~$\partial\Omega$,
 then
 \[
  \frac{\d}{\d t}_{|t = 0} J[q_{1,*}, \, q_{2,*}, \, q_{3,*} - t\varphi] \geq 0,
 \]
 and hence,
 \begin{equation} \label{subsol*}
  \int_{\Omega_+} \left\{ -6 \nabla q_{3,*}\cdot\nabla\varphi
   - \frac{\lambda^2}{L} f_*(q_{1,*}, \, q_{2,*}, \, q_{3,*}) \, \varphi\right\} \d S \geq 0,
 \end{equation}
 where
 \[
  f_*(q_1, \, q_2, \, q_3) := 6 q_3 \left(A - Bq_3 + 6C q_3^2\right) 
   + 2(B + 6C q_3)\left(q_1^2 + q_2^2\right).
 \]
 As before, there exists a number~$\delta>0$ 
 (depending only on~$A$, $B$, $C$) such that
 \begin{equation*} 
  f_*(q_1, \, q_2, \, q_3) > 0 \qquad \textrm{for any } q_1\in\R, \ q_2\in\R  
  \textrm{ and } q_3\in [-\delta, \, 0].
 \end{equation*}
 We can now show that~$q_{3,*}\leq -\delta$ in~$\Omega$
 by repeating the same arguments of Proposition~\ref{prop:WORS}.
\end{proof}



We now consider solutions of~\eqref{eq:EL123}, \eqref{eq:BC123} 
that satisfy~$q_3<0$ in~$\Omega$, and prove bounds on $q_3$ as a function of the re-scaled temperature $A$.

\begin{lemma} \label{lemma:maxprinciple}
 Any solution~$(q_1, \, q_2, \, q_3)$ of the system~\eqref{eq:EL123},
 subject to ~\eqref{eq:BC123}, satisfies
 \[
  q_1^2 + q_2^2 + 3q_3^2 \leq \frac{s_+^2}{3} \qquad \textrm{in } \Omega,
 \]
 where $s_+$ is the constant defined by~\eqref{eq:s+}.
\end{lemma}
This lemma be deduced from the corresponding
maximum principle for the full LdG system~\eqref{eq:EL};
see for instance~\cite[Proposition~3]{majumdar2010landau}.

\begin{lemma} \label{lemma:Hardy123}
 Let~$(q_1, \, q_2, \, q_3)$ be a solution of the system~\eqref{eq:EL123},
 subject to ~\eqref{eq:BC123},
 such that~$q_3<0$ everywhere in~$\Omega$. Then $  q_1^2 + q_2^2 < 9q_3^2 $ everywhere in $\Omega$.

\end{lemma}
\begin{proof}
 Define the functions~$\xi_1:= -q_1/q_3$ and~$\xi_2 := -q_2/q_3$. 
 Then, for~$k\in\{1, \, 2\}$, we have
 \begin{gather*}
  \nabla\xi_k = -\frac{1}{q_3}\nabla q_k + \frac{q_k}{q_3^2} \nabla q_3 \\
  \Delta\xi_k = -\frac{1}{q_3}\Delta q_k + \frac{q_k}{q_3^2} \Delta q_3
   + \frac{2}{q_3^2}\nabla q_3\cdot\nabla q_k 
   - \frac{2q_k}{q_3^3} \abs{\nabla q_3}^2
   = -\frac{1}{q_3}\Delta q_k + \frac{q_k}{q_3^2} \Delta q_3 
   - \frac{2}{q_3}\nabla q_3\cdot\nabla\xi_k 
 \end{gather*}
 Using the system~\eqref{eq:EL123}, for~$k\in\{1, \, 2\}$ we obtain
 \begin{equation} \label{hardy31}
  \begin{split}
   \Delta\xi_k + \frac{2}{q_3}\nabla q_3\cdot\nabla\xi_k 
   &= \frac{\lambda^2}{L}\bigg\{-\frac{q_k}{q_3}\left(A + 2B q_3 + 2C(q_1^2 + q_2^2 + 3q_3^2)\right) \\
   &\qquad\qquad
   + \frac{q_k}{q_3}\left(A  - B q_3 + 2C(q_1^2 + q_2^2 + 3q_3^2)\right)\bigg\}  
   + \frac{\lambda^2B}{3L}\frac{q_k}{q_3^2}\left( q_1^2 + q_2^2 \right) \\
   &= \frac{\lambda^2B}{3L} q_k \left(-9 + \xi_1^2 + \xi_2^2\right) \! .
  \end{split}
 \end{equation}
 Now, we define a non-negative function~$\xi$
 by~$\xi^2 := \xi_1^2+\xi_2^2$. We have
 \[
  \Delta(\xi^2/2) = \xi_1\Delta\xi_1 + \xi_2\Delta\xi_2
  + \abs{\nabla\xi_1}^2 + \abs{\nabla\xi_2}^2
 \]
 and hence, thanks to~\eqref{hardy31},
 \[
  \Delta(\xi^2/2) + \frac{2}{q_3}\nabla q_3\cdot
  \left(\xi_1\nabla\xi_1 + \xi_2\nabla\xi_2\right)
  = \frac{\lambda^2B}{3L} (q_1\xi_1 + q_2\xi_2)\left(\xi^2 - 9\right)
  + \abs{\nabla\xi_1}^2 + \abs{\nabla\xi_2}^2.
 \]
 Finally, we obtain 
 \begin{equation} \label{hardy32}
  \Delta(\xi^2/2) + \frac{2}{q_3}\nabla q_3\cdot\nabla(\xi^2/2)
  \geq \underbrace{-\frac{\lambda^2B}{3L} \frac{q_1^2 + q_2^2}{q_3}}_{\geq 0}
  \left(\xi^2 - 9\right).
 \end{equation}
 From the boundary conditions~\eqref{eq:BC123}, we know
 that~$\xi = \xi_1\leq 3$ on~$\partial\Omega$. 
 Then, the (strong) maximum principle applied to the differential
 inequality~\eqref{hardy32} implies that $\xi^2 < 9$
 everywhere inside~$\Omega$. Thus, the lemma follows.
\end{proof}

We define
\[
 s_- := \frac{B - \sqrt{B^2 + 24|A| C}}{4C}<0.
\]
In the following propositions, we prove bounds on~$q_3$, in terms of~$s_+$
(see~\eqref{eq:s+}) and~$s_-$.

\begin{proposition} \label{prop:bounds-highA}
 Let $-\frac{B^2}{3C} \leq A  < 0$ so that
 \begin{equation}\label{ineq1}
 \frac{s_{-}}{3}\geq-\frac{s_{+}}{6}\geq-\frac{B}{6C}. 
\end{equation} Let $(q_1, \, q_2, \, q_3)$ 
 be any solution of the PDE system~\eqref{eq:EL123},
 satisfying the boundary conditions~\eqref{eq:BC123},
 with $q_3<0$ in $\Omega$. Then
 \begin{equation} \begin{split}
 -\frac{s_{+}}{6}\leq q_3\leq\frac{s_{-}}{3} \quad \text{in}\quad\Omega.
 \end{split} \end{equation}
\end{proposition}
\begin{proof}
 Firstly, we shall prove the upper bound 
 $q_3\leq s_{-}/3$ in $\Omega$. Assume for a contradiction, 
 that the maximum of $q_3$ is attained at some point 
 $(x_0, \, y_0)\in\Omega$ such that $q_3(x_0, \, y_0)>s_{-}/3$.
 Then, using~\eqref{ineq1}, the following inequalities hold:
 \begin{equation*} \begin{split}
  Aq_3(x_0,y_0)-Bq_3^2(x_0,y_0)+6Cq_3^3(x_0,y_0)
    &> A\left(-\frac{B}{6C}\right) + B\left(-\frac{B^2}{36C^2}\right)
    + 6C\left(\frac{-B^3}{216C^3}\right)\\
  &>\left(-\frac{B^2}{3C}\right)\left(-\frac{B}{6C}\right)-\frac{B^3}{18C^2}=0,\\
 \end{split} \end{equation*}
 and
 \begin{equation*}
   2Cq_3(x_0,y_0)+\frac{B}{3}>2C\left(-\frac{B}{6C}\right)+\frac{B}{3}=0. 
 \end{equation*}
 We evaluate both sides of the equation for $q_3\in C^2(\Omega)$
 in~\eqref{eq:EL123} at the point~$(x_0, \, y_0)$:
 \begin{equation*} \begin{split}
  \underbrace{\Delta q_3(x_0,y_0)}_{\leq 0}
  &=\underbrace{\frac{\lambda^2}{L}\left\{Aq_3(x_0,y_0)-Bq_3^2(x_0,y_0)+6Cq_3^3(x_0,y_0)\right\}}_{>0} \\
  &+\underbrace{\frac{\lambda^2}{L}\left\{2Cq_3(x_0,y_0)+\frac{B}{3}\right\}
  \left(q_1^2(x_0,y_0)+q_2^2(x_0,y_0)\right)}_{\geq 0},
 \end{split} \end{equation*}
 which leads to a contradiction. Since~$q_3 = -s_+/6\leq s_-/3$ 
 on~$\partial\Omega$, we conclude that $q_3\leq s_-/3$ on~$\Omega$.

 Now let's prove the weaker lower bound $q_3\geq-\frac{B}{6C}$ 
 in $\Omega$. Assume for contradiction that the minimum of $q_3$ 
 is attained at some point $(x_1,y_1)\in\Omega$ such that
 $q_3(x_1,y_1)<-\frac{B}{6C}$. Then the following inequalities hold:
 \begin{equation*} \begin{split}
 Aq_3(x_1,y_1)-Bq_3^2(x_1,y_1)+6Cq_3^3(x_1,y_1)&<A\left(-\frac{B}{6C}\right)+B\left(-\frac{B^2}{36C^2}\right)+6C\left(\frac{-B^3}{216C^3}\right) \\
 &<\left(\frac{B^2}{3C}\right)\left(\frac{B}{6C}\right)-\frac{B^3}{18C^2}=0,
 \end{split} \end{equation*}
 and
 \begin{equation*}
  2Cq_3(x_1,y_1)+\frac{B}{3} < 2C\left(-\frac{B}{6C}\right)+\frac{B}{3}=0.
 \end{equation*}
Recalling the equation for $q_3\in C^2(\Omega)$ in~\eqref{eq:EL123} and the boundary conditions~\eqref{eq:BC123}, we get an immediate contradiction and obtain the lower bound
$q_3\geq-\frac{B}{6C}$.
 
 We are now ready to prove the optimal lower bound 
 $q_3\geq-\frac{s_{+}}{6}$ in $\Omega$. 
 Recalling Lemma~\ref{lemma:Hardy123} and $q_3\geq-\frac{B}{6C}$, we have that:
 \begin{equation}\label{eq:MC}
  \begin{split}
 \Delta q_3 &=\frac{\lambda^2}{L}\left\{Aq_3-Bq_3^2+6Cq_3^3\right\}+\frac{\lambda^2}{L}\left\{2Cq_3+\frac{B}{3}\right\}(q_1^2+q_2^2) \\
 &\leq \frac{\lambda^2}{L}\left\{Aq_3-Bq_3^2+6Cq_3^3\right\}+\frac{\lambda^2}{L}\left\{2Cq_3+\frac{B}{3}\right\}9q_3^2 \\
 &\leq \frac{\lambda^2}{L}\left\{Aq_3+2Bq_3^2+24Cq_3^3\right\} \quad \text{in}\quad\Omega.
  \end{split}
 \end{equation}
 Assume for a contradiction, that the minimum of $q_3$ is attained at some point $(x_2, \, y_2)\in\Omega$ such that $q_3(x_2,y_2)<-\frac{s_{+}}{6}$. Then, using (\ref{ineq1}), the following inequality holds:
 \begin{equation} \begin{split}
 Aq_3(x_2,y_2)+2Bq_3^2(x_2,y_2)+24Cq_3^3(x_2,y_2)&<A\left(-\frac{s_{+}}{6}\right)-2B\left(-\frac{s_{+}^2}{36}\right)+24C\left(\frac{-s_{+}^3}{216}\right)\nonumber \\
 &<\left(\frac{B^2}{3C}\right)\left(\frac{s_{+}}{6}\right)+2B\left(\frac{B^2}{36C^2}\right)+24C\left(-\frac{s_{+}^3}{216}\right) \nonumber\\
 &<\frac{B^3}{9C^2}-\frac{B}{9}\left(\frac{B^2}{C}\right)=0,
 \end{split} \end{equation}
 which when combined with the equation~\eqref{eq:MC}, yields $\Delta q_3(x_2, y_2)<0$.
 This is a contradiction, and the desired result follows.
\end{proof}

\begin{corollary} \label{cor:bounds}
 Assume that $A=-\frac{B^2}{3C}$ and let $(q_1, \, q_2, \, q_3)$ be a solution of~\eqref{eq:EL123}
 subject to boundary conditions~\eqref{eq:BC123} with $q_3<0$ in $\Omega$.
 Then $q_3 = -\frac{s_{+}}{6}$ in~$\Omega$.
\end{corollary}

Finally, we have the following inequalities for $A<-\frac{B^2}{3C}$:
\begin{gather}
\frac{s_{-}}{3}<-\frac{s_{+}}{6}<-\frac{B}{6C}. \label{ineq2}
\end{gather}

\begin{proposition} \label{prop:bounds-lowA}
 Assume that $A<-\frac{B^2}{3C}$. Let $(q_1, \, q_2, \, q_3)$ be any solution of the
 PDE system~\eqref{eq:EL123} satisfying ~\eqref{eq:BC123}
 with $q_3<0$ in $\Omega$. Then
 \begin{equation}
 \frac{s_{-}}{3}\leq q_3\leq-\frac{s_{+}}{6} \quad \text{in}\quad\Omega.
 \end{equation}
\end{proposition}
\textbf{Remark:} The proof of Proposition~\ref{prop:bounds-lowA} is completely analogous to
that of Proposition~\ref{prop:bounds-highA}. We first prove the lower bound $q_3\geq s_-/3$,
then the weaker upper bound $q_3\leq-\frac{B}{6C}$, and finally the sharp upper bound $q_3\leq -s_+/6$.
Each step is obtained by repeating almost word by word the arguments of Proposition~\ref{prop:bounds-highA}.
We omit the details for brevity. $\Box$

\subsubsection{Stability/Instability of the WORS with natural boundary conditions.} 

We first recall a result from \cite{canevari2017order} that ensures that the WORS is globally stable with natural boundary conditions on $\Gamma$, for arbitrary well heights or all values of $\epsilon$.

\begin{lemma} \label{lemma:uniqueness-s}
 For any $A < 0$, there exists~$\lambda_0 > 0$ (depending only
 on~$A$, $B$, $C$, $L$) such that, for~$\lambda < \lambda_0$, the
 functional~\eqref{eq:2D-LdG} has a unique critical point 
 in the class~\eqref{eq:2D-admissible}.
\end{lemma}

This result follows from a general uniqueness criterion for critical
points of functionals of the form~\eqref{eq:LdG}; see, e.g., 
\cite[Lemma~8.2]{lamy2014}, \cite[Lemma~3.2]{canevari2017order}. The WORS exists for all $\lambda$ and $A<0$ and an immediate consequence is that the WORS is the unique LdG energy minimizer for sufficiently small $\lambda$.


We now study the instability of the WORS 
with respect to in-plane perturbations of the eigenframe and these perturbations necessarily have a non-zero $q_2$ component,
when~$\lambda$ is large and~$A$ is low enough. To this end, we take a 
function~$\varphi\in C^1_{\mathrm{c}}(\Omega)$ and consider the perturbation
\[
 \Qvec_t(x, \, y) := \Qvec(x, \, y) + t\varphi(x, \, y) 
 \left(\nvec_1\otimes\nvec_2 + \nvec_2\otimes\nvec_1 \right) \!,
\]
where~$\nvec_1$, $\nvec_2$ are defined by~\eqref{eq:n12} 
and~$t\in\R$ is a small parameter. We compute the second
variation of the LdG energy ~\eqref{eq:2D-LdG}, about the WORS solution as discussed in Proposition~\ref{prop:WORS}:
\begin{equation} \label{H}
 H_\lambda[\varphi] := \frac{1}{2}\frac{\d^2}{\d t^2} I[\Qvec_t]_{|t=0} =
 \int_{\Omega} \left(\abs{\nabla\varphi}^2 + \frac{\lambda^2}{L} \varphi^2 
 \left(A + 2 Bq_{3} + 2C(q_{1}^2 + 3 q_{3}^2)\right)\right) \d S
\end{equation}
(see \cite[Section~5.3]{wang2018order}).

\begin{proposition} \label{prop:unstable}
 Let $A\leq -\frac{B^2}{3C}$. Let~$(q_1, \, q_2, \, q_3)$ be 
 a solution of~\eqref{eq:EL123}, subject to the boundary conditions~\eqref{eq:BC123},
 such that $q_2 = 0$ and~$q_3 < 0$ everywhere in~$\Omega$, such as the WORS-solution constructed in Proposition~\ref{prop:WORS}.
 For any function~$\varphi\in C^1_{\mathrm{c}}(\Omega)$
 that is not identically equal to zero, there exists a 
 number~$\lambda_0>0$ (depending on~$A$, $B$, $C$, $L$ and~$\varphi$) such that
 $H_\lambda[\varphi]<0$ when~$\lambda\geq\lambda_0$.
\end{proposition}
\begin{proof}
 Due to Lemma~\ref{lemma:maxprinciple} and Proposition~\ref{prop:bounds-lowA}, we have
 \begin{equation*} \begin{split}
 A+2Bq_3+2C(q_1^2 + 3q_3^2)
 &\leq A-\frac{Bs_{+}}{3}+\frac{2Cs_{+}^2}{3} \\
 &=A-\frac{B}{3}\left(\frac{B+\sqrt{B^2+24|A|C}}{4C}\right)+\frac{2C}{3}\left(\frac{2B^2+2B\sqrt{B^2+24|A|C}+24|A|C}{16C^2}\right) \\
 &=A+|A|=0.
 \end{split} \end{equation*}
 The equality holds if and only if $q_{3} = -s_+/6$ and~$q_{1}^2 + 3q_{3}^2 = s_+^2/3$,
 that is, if and only if $\abs{q_{1}} = s_+/2$ and~$q_{3} = -s_+/6$. However,
 from Lemma~\ref{lemma:Hardy123} we know that $3q_{3} < q_{1} < -3q_{3}$
 inside~$\Omega$, so we must have
 \[
   A + 2Bq_{3} + 2C(q_{1}^2 + 3q_{3}^2) < 0 \qquad
   \textrm{everywhere inside } \Omega.
 \]
 Then, for any fixed~$\varphi\in C^\infty_{\mathrm{c}}(\Omega)$
 that is not identically equal to zero, the quantity~$H_\lambda[\varphi]$
 defined by~\eqref{H} becomes strictly negative for~$\lambda$ large enough.
\end{proof}

\subsection{Surface anchoring on the top and bottom plates}
\label{sect:surface}


In this section, we consider more experimentally relevant 
boundary conditions on the top and bottom plates, 
$\Gamma:= \Omega \times\{0, \, \epsilon\}$.
Instead of natural boundary conditions, we impose surface
energies on~$\Gamma$. The free energy functional,
in dimensionless units, becomes
\begin{equation} \label{eq:LdG-s}
 \mathcal{F}_\lambda[\Qvec] := 
 \int_V \left(\frac{1}{2}\abs{\nabla\Qvec}^2 
 + \frac{\lambda^2}{L}f_b(\Qvec)\right) \d V
 + \frac{\lambda}{L} \int_{\Gamma} f_s(\Qvec) \, \d S,
\end{equation}
and~$f_s$ is the surface anchoring energy density defined by
\cite{OsipovHess, Sluckin, SenSullivan, golovaty2017dimension}
\begin{equation} \label{eq:fs}
 f_s(\Qvec) := \alpha_z \left(\Qvec\zhat\cdot\zhat + \frac{s_+}{3}\right)^2
 + \gamma_z \abs{(\Ivec - \zhat\otimes\zhat)\Qvec\zhat}^2,
\end{equation}
where~$\alpha_z$ and~$\gamma_z$ are positive coefficients.
We remark that the second term in~\eqref{eq:fs}, $\gamma_z |(\Ivec - \zhat\otimes\zhat)\Qvec\zhat|^2$,
is equal to zero if and only if $\Qvec\zhat$ is parallel to~$\zhat$.
Therefore, the surface energy density~$f_s$ favours $\Qvec$-tensors that have~$\zhat$ as an eigenvector, with 
constant eigenvalue~$-s_+/3$, on the top and bottom plates. 
We have Dirichlet boundary conditions \eqref{eq:bc-Dir} 
on the lateral surface; and the admissible class is  $\mathcal{B}$,
defined by~\eqref{eq:admissible}.

\begin{lemma} \label{lemma:EL-surface}
 Critical points of the functional~\eqref{eq:LdG-s}, in the admissible
 class~$\mathcal{B}$ defined by~\eqref{eq:admissible},
 satisfy the Euler-Lagrenge system~\eqref{eq:EL},
 subject to Dirichlet boundary conditions~\eqref{eq:bc-Dir} on
 the lateral surfaces and 
 \begin{equation} \label{eq:bc-Neu}
  \partial_\nu\Qvec + \frac{\lambda}{L} \Gvec(\Qvec) = 0 
  \qquad \textrm{on } \Gamma.
 \end{equation}
 Here, $\nu$ is the outward-pointing unit normal to~$V$
 and~$\Gvec$ is defined by
 \[
  \begin{split}
   \Gvec(\Qvec) := \left( \begin{matrix}
    -\dfrac{2}{3}\alpha_z\left(Q_{33} + \dfrac{s_+}{3}\right) & 0 
       & \gamma_z Q_{13} \\
    0 & -\dfrac{2}{3}\alpha_z\left(Q_{33} + \dfrac{s_+}{3}\right)
       & \gamma_z Q_{23} \\
    \gamma_z Q_{13} & \gamma_z Q_{23} & 
    \dfrac{4}{3}\alpha_z\left(Q_{33} + \dfrac{s_+}{3}\right)
   \end{matrix} \right) \! .
  \end{split}
 \]
\end{lemma}

\begin{remark} \label{remark:Lagrange}
 The matrix~$\Gvec(\Qvec)$ is symmetric and traceless. Therefore,
 the Lagrange multipliers associated with the symmetry and tracelessness
 constraints have already been embedded in the definition of~$\Gvec$.
\end{remark}
\begin{remark} \label{remark:2-3D}
 Because of the boundary condition~\eqref{eq:bc-Neu}, 
 $z$-independent solutions ($\partial_z\Qvec=0$) may not, in general,
 be solutions of the 3D problem with surface energy anchoring
 on the top and bottom plates. However, when~$A = -B^2/(3C)$
 we know that there exist $z$-independent solutions with~$Q_{33}=-s_+/3$;
 they correspond to triplets~$(q_1, \, q_2, \, q_3)$
 with constant~$q_3=-s_+/6$ (see Corollary~\ref{cor:bounds}).
 These $z$-independent solutions with constant~$Q_{33}$ are also solutions
 of the 3D problem with surface energies on the top and bottom plates.
\end{remark}

\begin{proof}[Proof of Lemma~\ref{lemma:EL-surface}]
 Let~$\Qvec\in\mathcal{B}$ be a critical point for~$\mathcal{F}_\lambda$,
 and let~$\Pvec\in H^1(V, \, S_0)$ be any perturbation
 such that $\Pvec = 0$ on~$\partial V\setminus\Gamma$.
 We compute the first variation of~$\mathcal{F}_\lambda$
 with respect to~$\Pvec$:
 \[
  \begin{split}
   0=&\frac{\d}{\d t}_{|t=0} \mathcal{F}_\lambda [\Qvec+t\Pvec]
    = \int_V \left(\nabla\Qvec:\nabla\Pvec 
    + \frac{\lambda^2}{L} \left(A\Qvec\cdot\Pvec - B\Qvec^2\cdot\Pvec
    + C\abs {\Qvec}^2\Qvec\cdot\Pvec\right)\right) \d V \\
   &\qquad\qquad + \frac{\lambda}{L}
    \int_{\Gamma} \left(2\alpha_z(\Pvec\zhat\cdot\zhat)
    \left(\Qvec\zhat\cdot\zhat + \frac{s_+}{3}\right)
    + 2\gamma_z (\Ivec - \zhat\otimes\zhat)\Qvec\zhat\cdot 
    (\Ivec - \zhat\otimes\zhat)\Pvec\zhat\right) \d S,
  \end{split}
 \]
 where~$\Qvec\cdot\Pvec := \tr(\Qvec\Pvec) = Q_{ij}P_{ij}$.
 By integrating by parts, and noting that~$\tr\Pvec = 0$, we obtain:
 \begin{equation} \label{EL1}
  \begin{split}
   &\int_V \left(-\Delta\Qvec + 
    \frac{\lambda^2}{L} \left( A\Qvec - B\Qvec^2
    + \frac{B}{3}\abs{\Qvec}^2\Ivec + C\abs {\Qvec}^2\Qvec\right)\right)
    \cdot\Pvec \, \d V +
    \int_\Gamma \partial_\nu\Qvec\cdot\Pvec \, \d S \\
   &\qquad\qquad + \frac{\lambda}{L} \int_{\Gamma}
    \left(2\alpha_z(\Pvec\zhat\cdot\zhat)
    \left(\Qvec\zhat\cdot\zhat + \frac{s_+}{3}\right)
    + 2\gamma_z (\Ivec - \zhat\otimes\zhat)\Qvec\zhat\cdot 
    (\Ivec - \zhat\otimes\zhat)\Pvec\zhat\right) \d S = 0.
  \end{split}
 \end{equation}
 We now deal with the integral on~$\Gamma$.
 We first remark that
 \begin{equation} \label{EL1.1}
  \begin{split}
   (\Pvec\zhat\cdot\zhat) 
   \left(\Qvec\zhat\cdot\zhat + \frac{s_+}{3}\right)
   &= \left(Q_{33} + \frac{s_+}{3}\right)P_{33} \\
   &= \left(\begin{matrix}
     -\dfrac{1}{3}\left(Q_{33} + \dfrac{s_+}{3}\right) & 0 & 0 \\
     0 & -\dfrac{1}{3}\left(Q_{33} + \dfrac{s_+}{3}\right) & 0 \\
     0 & 0 & \dfrac{2}{3}\left(Q_{33} + \dfrac{s_+}{3}\right)
   \end{matrix}\right)\cdot\Pvec
  \end{split}
 \end{equation}
 because~$\tr\Pvec = 0$. We also have
 \begin{equation} \label{EL1.2}
  \begin{split}
   (\Ivec - \zhat\otimes\zhat)\Qvec\zhat\cdot 
    (\Ivec - \zhat\otimes\zhat)\Pvec\zhat
    = \sum_{i=1}^2 Q_{i3}P_{i3}
    = \frac{1}{2}\left(\begin{matrix}
     0 & 0 & Q_{13} \\
     0 & 0 & Q_{23} \\
     Q_{13} & Q_{23} & 0 \\
    \end{matrix}\right)\cdot\Pvec
  \end{split}
 \end{equation}
 Using~\eqref{EL1}, \eqref{EL1.1} and~\eqref{EL1.2} we obtain
 \begin{equation*} 
  \begin{split}
   &\int_V \left(-\Delta\Qvec + \frac{\lambda^2}{L} 
    \left( A\Qvec - B\Qvec^2 + \frac{B}{3}\abs{\Qvec}^2\Ivec
    + C\abs {\Qvec}^2\Qvec\right)\right) \cdot\Pvec \, \d V 
    + \int_\Gamma \left(
    \partial_\nu\Qvec + \frac{\lambda}{L}
    \Gvec(\Qvec)\right)\cdot\Pvec \, \d S = 0
  \end{split}
 \end{equation*}
 for any~$\Pvec\in H^1( V, \, S_0)$ such that $\Pvec= 0$
 on~$\partial V\setminus\Gamma$, and the lemma follows.
\end{proof}

\begin{lemma}\label{lemma:maxprinciple3d}
 There exists a constant~$M$ (depending only on~$A$, $B$,
 $C$ but \emph{not on} $\lambda$, $L$, $\epsilon$) such that
 any solution~$\Qvec$ of the system~\eqref{eq:EL}, subject to the
 boundary conditions~\eqref{eq:bc-Dir} and~\eqref{eq:bc-Neu},
 satisfies
 \[
  \abs{\Qvec} \leq M \qquad \textrm{in } V.
 \]
\end{lemma}
\begin{proof}
 Let~$\Pvec := \Qvec + s_+(\zhat\otimes\zhat)/2$.
 We have $\partial_\nu(\abs{\Pvec}^2/2) = \partial_\nu\Pvec\cdot\Pvec 
 = \partial_\nu\Qvec\cdot\Pvec$ and hence,
 by~\eqref{eq:bc-Neu}, we deduce that
 \begin{equation} \label{maxprinc1}
  \begin{split}
  -\frac{L}{\lambda}\partial_\nu(\abs{\Pvec}^2/2)
   = \Gvec(\Qvec)\cdot\Pvec &= \gamma_z\sum_{i=1}^2 Q_{i3}^2
    + \dfrac{2}{3}\alpha_z\left(Q_{33} + \dfrac{s_+}{3}\right)
    \left(-Q_{11}- Q_{22} + 2Q_{33} + s_+ \right) \\
   &= \gamma_z\sum_{i=1}^2 Q_{i3}^2
    + 2\alpha_z\left(Q_{33} + \dfrac{s_+}{3}\right)^2\geq 0
    \qquad \textrm{on }\Gamma.
  \end{split}
 \end{equation}
 Similarly, we manipulate the Euler-Lagrange system to obtain
 \begin{equation} \label{maxprinc2}
  \begin{split}
   \frac{L}{\lambda^2}\Delta(\abs{\Pvec}^2/2)
   &= \frac{L}{\lambda^2}\Delta\Qvec\cdot\left(\Qvec + 
      \dfrac{s_+}{2}\zhat\otimes\zhat\right) 
    + \frac{L}{\lambda^2}\abs{\nabla\Qvec}^2 \\
   &\geq A\abs{\Qvec}^2 - B\tr(\Qvec^3) + C\abs {\Qvec}^4
    + \frac{s_+}{2}\left( \left(A  + C\abs {\Qvec}^2\right)
    Q_{33} - B Q_{3k}Q_{3k} + \dfrac{B}{3}\abs{\Qvec}^2\right)
  \end{split}
 \end{equation}
 The right-hand side of~\eqref{maxprinc2} 
 is a quartic polynomial in~$\Qvec$, with leading order 
 term~$C\abs{\Qvec}^4$ and~$C>0$. Therefore, there exists 
 a positive number~$M_1$ (depending on $A$, $B$ and~$C$ only)
 such that the right-hand side of~\eqref{maxprinc2}
 is positive when~$\abs{\Qvec}\geq M_1$. By the triangular inequality,
 we have
 \[
  \abs{\Pvec} \geq M_2 := M_1 + \frac{s_+}{2}
  \quad \implies \quad \abs{\Qvec} =
  \abs{\Pvec - \frac{s_+}{2}\zhat\otimes\zhat} \geq M_1
 \]
 and hence the right-hand side of~\eqref{maxprinc2} is positive when
 $\abs{\Pvec}\geq M_2$. Finally, the boundary 
 datum~$\Qvec_{\mathrm{b}}$ on the lateral surfaces,
 defined by~\eqref{eq:bc-Dir}, satisfies
 $\abs{\Qvec_{\mathrm{b}}}\leq (2/3)^{1/2}s_+$ 
 on~$\partial\Omega\times(0, \, \epsilon)$.
 By applying the maximum principle 
 to~\eqref{maxprinc1} and~\eqref{maxprinc2}, we obtain that
 \[
  \abs{\Pvec} \leq \max\left \{M_2, \, 
  \left(\sqrt{\frac{2}{3}} + \frac{1}{2}\right)s_+ \right\} 
  \qquad \textrm{in }  V
 \]
 Then, by the triangular inequality, $\Qvec$ is also
 bounded in terms of~$A$, $B$ and~$C$ only.
\end{proof}


Adapting the methods in~\cite{canevari2017order}, for any values of~$\lambda$
and~$\epsilon$, it is possible to construct
a WORS-like solution for this 3D problem with surface anchoring
on the top and bottom plates. The WORS has a constant eigenframe and, hence
it can be completely described in terms of two
degrees of freedom as before:
\begin{equation} \label{Q-WORS}
 \begin{split}
  \Qvec(x, \, y, \, z) &= q_1(x, \, y, \, z)
  \left(\nvec_1 \otimes \nvec_1 - \nvec_2\otimes \nvec_2 \right)
  + q_3(x, \, y, \, z) \left(2 \zhat\otimes\zhat
  - \nvec_1\otimes \nvec_1 - \nvec_2\otimes \nvec_2 \right) \! ,
 \end{split}
\end{equation}
where~$q_1$, $q_3$ are scalar functions,~$\nvec_1$,
$\nvec_2$ are given by~\eqref{eq:n12} and $q_1$ is constrained to vanish on the diagonals with symmetry 
\begin{equation} \label{WORS3}
 x y \, q_1(x, \, y, \, z) \geq 0 \qquad \textrm{for any } (x, \, y, \, z)\in V.
\end{equation}

\begin{proposition} \label{prop:WORS3D}
 For any~$\lambda$, $\epsilon$ and~$A$, there exists a solution of the form ~\eqref{Q-WORS}, of the system~\eqref{eq:EL},
 subject to the boundary conditions~\eqref{eq:bc-Dir} and~\eqref{eq:bc-Neu},
 satisfies~\eqref{WORS3} with $q_1=0$ along the square diagonals and has~$q_3 < 0$ on~$V$.
\end{proposition}
\begin{proof}
 Let
 \[
  V_+ := \left\{(x, \, y, \, z)\in V\colon x>0, \ y > 0\right\} \!.
 \]
 Following the approach in~\cite{canevari2017order}, we consider
 the functional
 \begin{equation} \label{eq:LdG-WORS3}
  \begin{split}
   G[q_1, \, q_3] &:= \int_{V_+} \left\{
   \abs{\nabla q_1}^2 + 3 \abs{\nabla q_3}^2 
   + \frac{\lambda^2}{L} \left(A (q_1^2 + 3 q_3^2 )
   + 2 B q_3 q_1^2 - 2 B q_3^3 + C(q_1^2 + 3 q_3^2)^2\right)\right\} \d S \\
   &\qquad\qquad + \frac{2\lambda\alpha_z}{L}
   \int_{\Gamma\cap\overline{V}_+} \left(q_3 + \frac{s_+}{6}\right)^2 \d S,
  \end{split}
 \end{equation}
obtained by substituting the ansatz~\eqref{Q-WORS}
 into \eqref{eq:LdG-s}.
 We minimize~$G$ among the finite-energy pairs $(q_1, \, q_3)\in H^1(V_+)^2$,
 subject to the constraint~$q_3\leq 0$ on~$V$ and to the boundary conditions
 \begin{equation} \label{eq:bc3D}
  q_1 = q_{1\mathrm{b}} \ \textrm{ and } \ q_3 = -{s_+}/{6} 
  \quad \textrm{on  } (\partial\Omega\times (0, \, \epsilon) )\cap \overline{V_+},
  \qquad q_1 = 0 \quad \textrm{on } \partial V_+ \setminus\partial V,
 \end{equation}
 where the function~$q_{1\mathrm{b}}$ is defined by~\eqref{eq:BC1}.
 A routine application of the direct method of the calculus of variations shows 
 that a minimizer~$(q_1^{WORS}, \, q_3^{WORS})$ exists. Without loss of generality,
 we can assume that $q_1^{WORS}\geq 0$ on~$V_+$; otherwise, we replace $q_1^{WORS}$
 with~$|q^{WORS}_1|$ and note that $G[q_1^{WORS}, \, q_3^{WORS}] = G[|q_1^{WORS}|, \, q_3^{WORS}]$.
 
 We claim that~$q_3^{WORS}\leq-\delta$ for some strictly positive constant~$\delta$,
 depending only on~$A$, $B$ and~$C$. The proof of this claim follows the argument in
 Proposition~\ref{prop:WORS}. We take a perturbation~$\varphi\in H^1(V_+)$ such that
 $\varphi\geq 0$ in~$V_+$ and~$\varphi = 0$ on~$\partial V\cap\overline{V_+}$.
 Then, $q_3^t := q_3^{WORS} - t\varphi$, for~$t\geq 0$, is
 an admissible perturbation for~$q_3^{WORS}$
 and from the optimality condition
 \[
  \frac{\d}{\d t}_{|t = 0} G[q_1^{WORS}, \, q_3^t] \geq 0
 \]
 we deduce
 \begin{equation} \label{subsol3}
  \int_{V_+} \left\{ -6 \nabla q_3^{WORS}\cdot\nabla\varphi
   - \frac{\lambda^2}{L} f(q_1^{WORS}, \, q_3^{WORS}) \, \varphi\right\} \d V
   - \frac{4\lambda\alpha_z}{L} \int_{\Gamma\cap\overline{V}_+} 
   \left(q_3 + \frac{s_+}{6}\right)\varphi \, \d S \geq 0.
 \end{equation}
 The function~$f(q_1^{WORS}, \, q_3^{WORS})$ is defined by~\eqref{f},
 and by~\eqref{signf} we know that there exists a constant~$\delta \in (0, \, s_+/6)$
 such that $f(q_1, \, q_3)>0$ for any~$q_1\in\R$ and any~$q_3\in [-\delta, \, 0]$.
 We choose~$\varphi$ as in~\eqref{varphi-subsol} and, 
 due to~\eqref{subsol3}, we deduce that
 $q_3^{WORS}\leq -\delta$ in~$V_+$.
 
 Since~$q_3^{WORS}$ is strictly negative, we can consider perturbations 
 of the form $q_3^t := q_3^{WORS} + t\varphi$, irrespective of the sign of~$\varphi$,
 provided that~$\abs{t}$ is sufficiently small. As a consequence,
 $(q_1^{WORS}, \, q_3^{WORS})$ solves the Euler-Lagrange system
 \begin{equation} \label{eq:EL13}
  \left\{\begin{aligned}
   \Delta q_1 &= \frac{\lambda^2}{L} q_1 \left\{A + 2B q_3 + 2C\left(q_1^2 + 3q_3^2 \right)\right\} \\
   \Delta q_3 &= \frac{\lambda^2}{L} q_3 \left\{A - Bq_3 + 2C\left(q_1^2  + 3q_3^2\right)\right\} 
   + \frac{\lambda^2B}{3L}q_1^2 
  \end{aligned} \right.
\end{equation}
 on~$V_+$, as well as the boundary conditions
 \begin{equation} \label{eq:bc_neu-wors}
  \partial_\nu q_1 = 0, \quad
  \partial_\nu q_3 + \frac{4\lambda\alpha_z}{3L} \left(q_3 + \frac{s_+}{6}\right) = 0
  \qquad \textrm{on } \Gamma\cap\overline{V}_+
 \end{equation}
 and~$\partial_\nu q_3=0$ on~$\partial V_+ \setminus\partial V$.
 We extend~$(q_1^{WORS}, \, q_3^{WORS})$ to the whole of~$V$
 by reflections about the planes~$\{x = 0\}$ and~$\{y = 0\}$:
 \[
  q_1^{WORS}(x, \, y, \, z) := \mathrm{sign}(xy)\,  q_1^{WORS}(\abs{x}, \, \abs{y}, \, z), \qquad
  q_3^{WORS}(x, \, y, \, z) := q_3^{WORS}(\abs{x}, \, \abs{y}, \, z)
 \]
 for any~$(x, \, y, \, z)\in V\setminus\overline{V_+}$.
 The functions $q_1^{WORS}$, $q_3^{WORS}$, defined above, solve the Euler-Lagrange 
 system~\eqref{eq:EL13} on $\Omega\setminus(\{x=0\}\cup\{y=0\})$.
 In fact, an argument based on elliptic regularity,
 along the lines of~\cite[Theorem~3]{dangfifepeletier},
 shows that $(q_1^{WORS}, \, q_3^{WORS})$ is a solution of~\eqref{eq:EL13} on the whole of~$\Omega$. 
 Finally, using~\eqref{eq:EL13}, \eqref{eq:bc3D} and~\eqref{eq:bc_neu-wors}, we can check that
 the $\Qvec$-tensor associated with~$(q_1^{WORS}, \, q_3^{WORS})$, 
 as defined by~\eqref{Q-WORS}, has all the required properties.
\end{proof}


Adapting a general criterion for uniqueness of critical points
(see, e.g., \cite[Lemma~8.2]{lamy2014}),
we can show that the functional~\eqref{eq:LdG-s} has a unique
critical point in the admissible class~\eqref{eq:admissible}
when~$\lambda$ is small enough, irrespective of~$\epsilon$, which implies that the WORS is globally stable for $\lambda$ sufficiently small with surface energies too.

\begin{proposition} \label{prop:uniq_small_lambda}
 There exists a positive number~$\lambda_0$
 (depending only on~$A$, $B$, $C$) such that, 
 when~$\lambda<\lambda_0$, the system~\eqref{eq:EL}
 has a unique solution that satisfies the boundary
 conditions~\eqref{eq:bc-Dir}, \eqref{eq:bc-Neu}.
\end{proposition}

The main step of the proof is the following

\begin{lemma} \label{lemma:strictly_convex}
 For any $M > 0$, there exists a 
 $\lambda_0 = \lambda_0(M, \, A, \, B, \, C, \, L, \, \Omega)$ such that,
 for  $\lambda < \lambda_0$, the functional~$\mathcal{F}_\lambda$
 given by~\eqref{eq:LdG-s} is strictly convex in the class
 \begin{equation}
    X = \left\{\Qvec \in H^{1}(V, \, S_0)\colon
    |\Qvec| \leq M \ \textrm{ on } V, \quad
    \Qvec = \Qvec_{\mathrm{b}} \ \textrm{ on } 
    \partial\Omega\times(0, \, \epsilon)\right \} \! .
 \end{equation}
\end{lemma}

Once Lemma~\ref{lemma:strictly_convex} is proved,
Proposition~\ref{prop:uniq_small_lambda} follows. 
Indeed, let us consider the constant~$M$ given by
Lemma~\ref{lemma:maxprinciple3d}. Then, any solution
of the system~\eqref{eq:EL}, subject to the boundary
conditions~\eqref{eq:bc-Dir}, \eqref{eq:bc-Neu},
must belong to the class~$X$, by Lemma~\ref{lemma:maxprinciple3d}.
However, if $\mathcal{F}_\lambda$ is strictly convex in~$X$, then
it cannot have more than one critical point in~$X$.

\begin{proof}[Proof of Lemma~\ref{lemma:strictly_convex}]
  For any $\Qvec_1, \Qvec_2 \in X$, we have
  \begin{equation} \label{Poinc0}
    \begin{aligned}
      \mathcal{F}_{\lambda}\left( \frac{\Qvec_1 + \Qvec_2}{2} \right) 
      & = \int_{V} \frac{1}{8} |\nabla(\Qvec_1 + \Qvec_2)|^2 
        + \frac{\lambda^2}{L} f_b \left(\frac{\Qvec_1 + \Qvec_2}{2}\right) \d V 
        + \frac{\lambda}{L} \int_{\Gamma} f_s\left(\frac{\Qvec_1 + \Qvec_2}{2} \right) \d S \\
      & = \int_{V} \frac{1}{4} \left( |\nabla \Qvec_1|^2 + |\nabla \Qvec_2|^2
        - \frac{1}{2} |\nabla(\Qvec_1 - \Qvec_2)|^2 \right) 
        + \frac{\lambda^2}{L}  f_b \left( \frac{\Qvec_1 + \Qvec_2}{2} \right) \d V  \\
      &\qquad\qquad + \frac{\lambda}{L} \int_{\Gamma} 
        f_s \left(\frac{\Qvec_1 + \Qvec_2}{2} \right) \d S.  \\
    \end{aligned}
  \end{equation}
  Since $f_s(\Qvec)$ is a convex function of $\Qvec$, we have
  \begin{equation} \label{Poinc.5}
    \int_{\Gamma} f_s \left(\frac{\Qvec_1 + \Qvec_2}{2}\right) \d S
    \leq \frac{1}{2}\int_{\Gamma} \left(f_s(\Qvec_1) + f_s (\Qvec_2)\right) \d S.
  \end{equation}
  We now deal with the bulk term, $f_b$.
  Both~$\Qvec_1$ and~$\Qvec_2$ are equal to~$\Qvec_{\mathrm{b}}$
  on~$\partial\Omega\times(0, \, \epsilon)$ and hence, 
  $\Qvec_2 - \Qvec_1 = 0$ on~$\partial\Omega\times (0, \, \epsilon)$.
  For a.e. fixed $z_0 \in (0, \, \epsilon)$, 
  by the Poincar\'e inequality on $\Omega$, we have
  \begin{equation}
   \left\|\Qvec_1(\cdot, \, \cdot, \,  z_0) 
    - \Qvec_2(\cdot, \, \cdot, \,  z_0)\right\|_{L^2(\Omega)}^2 
   \leq C_1(\Omega) \left\|  \nabla_{x, y}(\Qvec_1(\cdot, \, \cdot, \,  z_0) 
    - \Qvec_2(\cdot, \, \cdot, \,  z_0))\right\|^2_{L^2(\Omega)} \! ,
  \end{equation}
  where~$C_1(\Omega)$ is a positive constant that only
  depends on the geometry of~$\Omega$.
  By integrating the previous inequality with respect
  to~$z_0\in (0, \, \epsilon)$, we deduce that
  \begin{equation} \label{Poinc}
   \int_{V}\abs{\Qvec_1- \Qvec_2}^2 \d V
   \leq C_1(\Omega) \int_{V} \abs{\nabla(\Qvec_1 - \Qvec_2)}^2 \d V \! .
  \end{equation}
  Since $|\Qvec_1| \leq M$ and
  $|\Qvec_2| \leq M$ everywhere in~$V$, we have
  \begin{equation} \label{Poinc2}
    \int_V\left| f_b \left(\frac{\Qvec_1 + \Qvec_2}{2}\right) 
    - \frac{1}{2}f_b(\Qvec_1) - \frac{1}{2}f_b(\Qvec_2)\right| \d V
    \leq \left\|f_b\right\|_{W^{2,\infty} (B_M)}
    \int_{V} \abs{\Qvec_1 - \Qvec_2}^2 \d V,
  \end{equation}
  where $B_M =  \left \{ \Qvec \in S_0\colon |\Qvec| \leq M \right\}$
  and $\left\|f_b\right\|_{W^{2,\infty} (B_M)}$ is a positive constant that
  bounds the second derivatives of~$f_b$ in~$B_M$
  (in particular, $\left\|f_b\right\|_{W^{2,\infty} (B_M)}$ 
  only depends on $M$, $A$, $B$ and~$C$).
  Combining~\eqref{Poinc} and~\eqref{Poinc2}, we find
  a positive constant~$C_2 = C_2 (f_b, \, \Omega, \, M) 
  := C_1(\Omega) \left\|f_b\right\|_{W^{2,\infty} (B_M)}$ such that
  \begin{equation} \label{Poinc3}
   \int_V f_b \left(\frac{\Qvec_1 + \Qvec_2}{2}\right) \d V \leq
    \frac{1}{2} \int_V \left(f_b(\Qvec_1) + f_b(\Qvec_2) \right) \d V 
    + C_2 \int_V \abs{\nabla(\Qvec_1 - \Qvec_2)}^2 \d V.
  \end{equation}
  Now, we use~\eqref{Poinc.5} and~\eqref{Poinc3}
  to bound the right-hand side of~\eqref{Poinc0}. We obtain
  \begin{equation}
   \begin{aligned}
      \mathcal{F}_{\lambda}\left(\frac{\Qvec_1 + \Qvec_2}{2}\right) 
      & \leq \frac{1}{2} \left(\mathcal{F}_{\lambda}(\Qvec_1) 
      + \mathcal{F}_{\lambda}(\Qvec_2)\right)
      + \left(\frac{C_2\lambda^2}{L} - \frac{1}{8}\right)
      \int_V \abs{\nabla(\Qvec_1 - \Qvec_2)}^2 \d V
   \end{aligned}
  \end{equation}
  If we take $\lambda < \lambda_0 := (\frac{L}{8C_2})^{1/2}$, 
  then we have
  \begin{equation}
     \mathcal{F}_{\lambda}\left(\frac{\Qvec_1 + \Qvec_2}{2}\right) 
      \leq \frac{1}{2} \left(\mathcal{F}_{\lambda}(\Qvec_1) 
      + \mathcal{F}_{\lambda}(\Qvec_2)\right) 
  \end{equation}
  and the equality holds if and only if~$\Qvec_1 = \Qvec_2$.
  This proves that~$\mathcal{F}_\lambda$ is strictly convex in~$X$.
\end{proof}

We deduce that the WORS-solution survives in 3D wells, independently of the well height, with both natural boundary conditions and realistic surface energies on the top and bottom surfaces. Moreover, the WORS is globally stable for $\lambda$ small enough, independent of well height and in the next section, we complement our analysis with numerical examples.

\section{Numerics}
\label{sec:numerics}

\subsection{Numerical Methods}

For computational convenience, in this section we take the cross-section of the well, $\Omega$,
to be a (non-truncated) square with sides parallel to the coordinate axes, i.e.~$\Omega = (-1, \, 1)^2$.
We consider the general $\Qvec$-tensor with five degrees of freedom
\begin{equation}\label{q12345}
\begin{aligned}
\Qvec(x,y) & = q_1(x, y, z)(\xhat \otimes \xhat - \yhat \otimes \yhat) + q_2(x, y, z) (\xhat \otimes \yhat + \yhat \otimes \xhat) \\
       & + q_3(x, y, z) (2 \zhat \otimes \zhat - \xhat \otimes \xhat - \yhat \otimes \yhat) \\
       & + q_4(x, y, z)(\xhat \otimes \zhat + \zhat \otimes \xhat) + q_5(x, y, z)(\yhat \otimes \zhat + \zhat \otimes \yhat ), \\
\end{aligned}
\end{equation}
where $\xhat$, $\yhat$ and $\zhat$ are unit-vectors in the $x$-, $y$- and $z$-directions respectively.

Moreover, instead of considering Dirichlet conditions (infinite strong anchoring) on the lateral surfaces, we consider finite anchoring on the lateral surfaces, which allows us to study nematic equilibria without excluding the corners of the well \cite{Walton2018}. More precisely, we impose surface energies on the lateral sides given by \cite{kralj2014order} 

\begin{equation}\label{side_anchoring_1}
 \begin{aligned}
   & f_s(\Qvec) = \omega_1 \left( \Qvec - g(x) \left( \xhat \otimes \xhat - \frac{1}{3} \mathbf{I} \right) \right)^2,  \quad y = 0, 1;  \\
   & f_s(\Qvec) = \omega_2 \left( \Qvec - g(y) \left( \yhat \otimes \yhat - \frac{1}{3} \mathbf{I} \right) \right)^2,  \quad x = 0, 1;  \\
 \end{aligned}
\end{equation}
where $\omega_i = \frac{W_i \lambda}{L}$ is the non-dimensionalized anchoring strength, $W_i$ is the surface anchoring, the function $g \in C^{\infty}([-1, 1])$ eliminates the discontinuity at the corners e.g
\begin{equation}
g(x) = s_{+}, \quad \forall x \in [-1 + \delta, 1 - \delta], \qquad g(-1) = g(1) = 0,
\end{equation}
for a small constant $\delta$. The choice of $g$ does not affect numerical results qualitatively.
We take $W_1 = W_2 = 10^{-2} \mathrm{Jm}^{-2}$ to account for the strong anchoring on the lateral sides of well \cite{ravnik2009landau}.

On $\Gamma$ - the top and bottom surfaces, 
the surface energy for finite tangential anchoring (\ref{eq:fs}) can be written as
\begin{equation}\label{fs_z}
 f_s(\Qvec)  =  w_z \int_{\Gamma}\left(\alpha_z \left( \Qvec \zhat \cdot \zhat + \frac{1}{3} s_{+} \right)^2 + \gamma_z \big| \left( \mathbf{I} - \zhat \otimes \zhat \right) \Qvec \zhat \big|^2\right)\d S
\end{equation}
where $w_z = \frac{W_z \lambda}{L}$ is the non-dimensionalized anchoring strength. The surface energy (\ref{fs_z}) favors planar boundary conditions on the top and bottom surface, such that $\zhat$ is an eigenvector of $\Qvec$ with associated eigenvalue $-\frac{1}{3}s^{+}$.

Instead of solving the Euler-Lagrange equations for the LdG free energy, we use the energy-minimization based numerical method \cite{gartlanddavis, wang2017topological} to find the minimizer of current system. The physical domain can be rescaled to $\Omega_c = \{ (\bar{x}, \bar{y}, \bar{z}) ~|~ \bar{x} \in [0, 2 \pi], \bar{y} \in [0, 2 \pi], \bar{z} \in [-1, 1]  \}$. Since $\Qvec$ is a symmetric and traceless matrix, $\Qvec$ can be written as
\begin{equation}
\Qvec = 
\begin{pmatrix} 
p_1 & p_2 & p_3 \\
p_2 & p_4 & p_5 \\
p_3 & p_5 & -p_1 - p_4 \\
\end{pmatrix}.
\end{equation}
We can expand $p_i$ in terms of special functions: Fourier series on $\bar{x}$ and $\bar{y}$, and Chebyshev polynomials on $\bar{z}$, i.e.
\begin{equation}\label{Expand}
  p_i(\bar{x}, \bar{y}, \bar{z}) = \sum_{l=1-L}^{L-1}\sum_{m = 1-M}^{M-1} \sum_{n = 0}^{N-1} p^{lmn}_{i} X_l(\bar{x}) Y_{m}(\bar{y})Z_{n}(\bar{z}).
\end{equation}
where $L$, $M$, $N$ specify the truncation limits of the expanded series, $X_l$ and $Y_m$ are defined as
\begin{equation}
X_{l}(\bar{x}) =
\begin{cases}
\cos l \bar{x} \quad l \geq 0, \\
\sin |l| \bar{x} \quad l < 0.
\end{cases}
\quad
Y_{m}(\bar{y}) =
\begin{cases}
\cos m \bar{y} \quad m \geq 0, \\
\sin |m| \bar{y} \quad m < 0.
\end{cases}
\end{equation}

Inserting (\ref{Expand}) into the LdG free energy (\ref{eq:LdG}) with surface energy term (\ref{side_anchoring_1}) and (\ref{fs_z}), we get a function of $\mathbf{p} = ( p_{lmn}^{i} ) \in \mathbb{R}^{D}$, where $D = (2L - 1) \times (2M - 1) \times N$.
The minimizers of function $F(\mathbf{p})$ can be found by some standard optimization methods. In the following simulation, we mainly use L-BFGS, which is a type of quasi-Newton methods and is efficient for our problem\cite{wright1999numerical}.


The energy-minimization based numerical approach with L-BFGS usually converges to a local minimizer with a proper initial guess, but that is not necessarily guaranteed. Similar to Ref. \cite{robinson2017molecular}, we can justify the stability of an obtained solution $\mathbf{p}$ by computing the smallest eigenvalue $\lambda_1$ of Hessian matrix $\mathbf{G}(\mathbf{p})$ corresponding to $\mathbf{p}$:
\begin{equation}\label{Rl}
\lambda_1 = \min_{\vvec \neq 0, \\ \vvec \in \mathbb{R}^{D}} \frac{ \langle \mathbf{G}(\mathbf{p}) \vvec, \, \vvec \rangle }{ \langle \vvec,  \vvec \rangle },
\end{equation}
where $\langle\cdot, \, \cdot \rangle$ is the standard inner product in $\mathbb{R}^{D}$.
A solution is locally stable (metastable) if $\lambda_1 > 0$.
Practically, $\lambda_1$ can be computed by solving the gradient flow equation of $\vvec$
\begin{equation}\label{eq1}
\frac{\pp \vvec}{\pp t} = - \frac{2 \gamma}{ \langle \vvec, \, \vvec \rangle } \left( \mathbf{G} \vvec - \frac{\langle \mathbf{G} \vvec, \, \vvec \rangle}{ \langle \vvec, \, \vvec \rangle}  \vvec  \right),
\end{equation}
where $\gamma(t)$ is a relaxation parameter, and $\mathbf{G}\vvec = \mathbf{G} (\mathbf{p}) \vvec$ is approximated by
\begin{equation}
\mathbf{G}(\mathbf{p}) \vvec =  - \frac{ \nabla_{D} F(\mathbf{p} + l \vvec) -  \nabla_D F (\mathbf{p} - l \vvec)}{2l},
\end{equation}
for some small constant $l$. We can choose $\gamma(t)$ properly to accelerate the convergence of the dynamic system (\ref{eq1}).

In what follows, we frequently refer to the biaxiality parameter\cite{majumdar2010landau}
\[
\beta^2 = 1 - 6 \frac{\left(\textrm{tr} \Qvec^3 \right)^2}{|\Qvec|^6}
\]
such that $0\leq \beta^2 \leq 1$ and $\beta^2 = 0$ if and only if $\Qvec$ is uniaxial or isotropic (for which we set $\beta^2=0$ by default).

\subsection{Numerical Results}

 In the following, we take $A = -\frac{B^2}{3C}$ if not stated differently, so all material constants in Landau-de Gennes free-energy are fixed.
The two key dimensionless variables are
\begin{equation}
\bar{\lambda}^2 = \frac{2C \lambda^2}{L}, \quad \epsilon = \frac{h}{\lambda},
\end{equation}
which describe the cross-sectional size and height of the square well respectively. Other dimensionless variables are related to the surface energy on all six surfaces.

\subsubsection{Strong anchoring on the lateral surfaces}
Firstly, we consider strong anchoring on the lateral surfaces, by taking $W_1 = W_2 = 10^{-2} \mathrm{J m}^{-2}$ in (\ref{side_anchoring_1}). For the surface energy on the top and bottom plates (see~\eqref{fs_z}), we take $W_z = 10^{-2} \mathrm{Jm}^{-2}$ if not stated differently.
For relatively large $\bar{\lambda}^2$,  we find the well-known diagonal and rotated solutions as stable configurations for arbitrary $\epsilon$~\cite{tsakonas2007multistable}. These are essentially described by $\Qvec$-tensors of the form
\[
\Qvec = q \left( \nvec \otimes \nvec - \frac{\mathbf{I}_2}{2} \right) + q_3 \left( \zhat\otimes \zhat - \frac{\mathbf{I}_3}{3} \right)
\]
where $\nvec = \left(\cos \theta, \sin \theta, 0 \right)$, $q>0$, $q_3 <0$, $\mathbf{I}_2$ is the identity matrix in two dimensions and $\mathbf{I}_3$ is the identity matrix in three dimensions respectively. Moreover, $\theta$ is a solution of $\Delta \theta = 0$ on a square subject to appropriately defined Dirichlet conditions \cite{lewis2014colloidal}. In the case of the diagonal solution, $\nvec$ roughly aligns along one of the square diagonals whereas for the rotated solution, $\theta$ rotates by approximately $\pi$ radians between a pair of opposite square edges. In Fig. \ref{3D_WRD}(a)-(b). we plot a diagonal and a rotated solution for $\bar{\lambda}^2 = 100$ and $\epsilon = 4$ with $A = - \frac{B^2}{3C}$. These solutions are z-invariant, as $|\pp_z \Qvec|^2 \approx 10^{-12}$ in our numerical solutions.

\begin{figure}[!htb]
  \begin{center}
    \begin{overpic}[height = 12.5em]{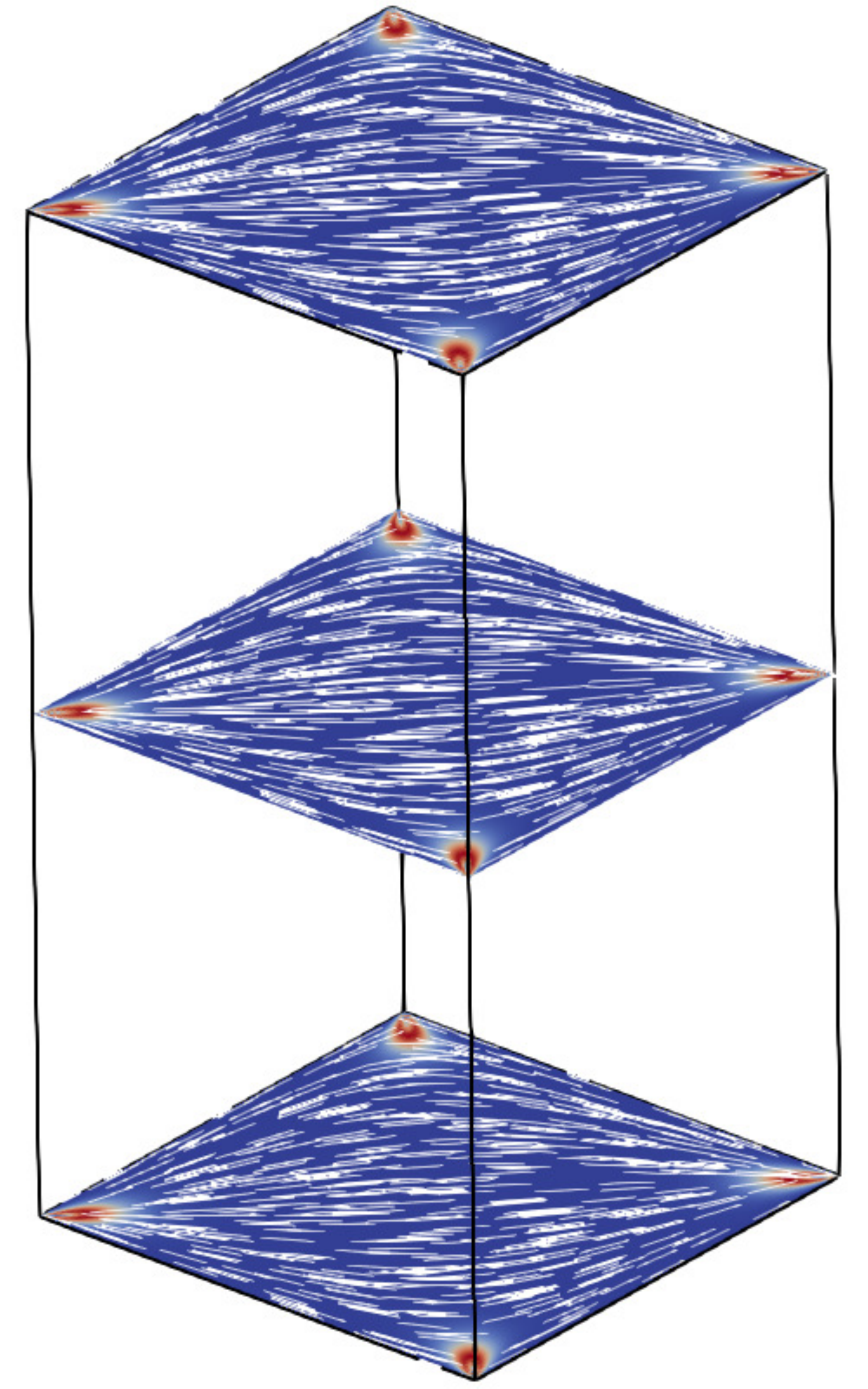}
      \put(-5, 100){(a)}
  \end{overpic}
  \hspace{1em}
    \begin{overpic}[height = 12.5em]{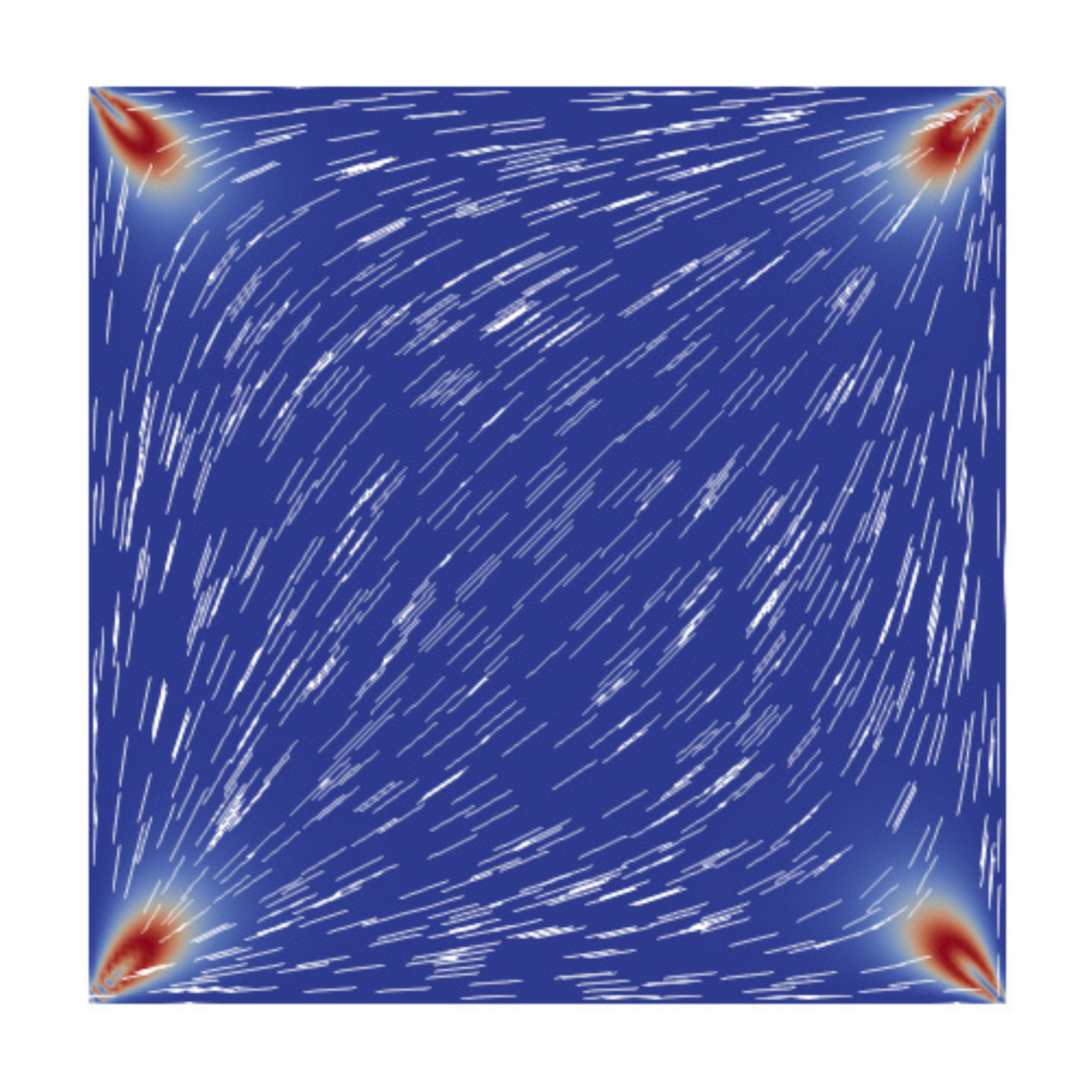}
    \end{overpic}
    \hspace{2em}
    \begin{overpic}[height = 12.5em]{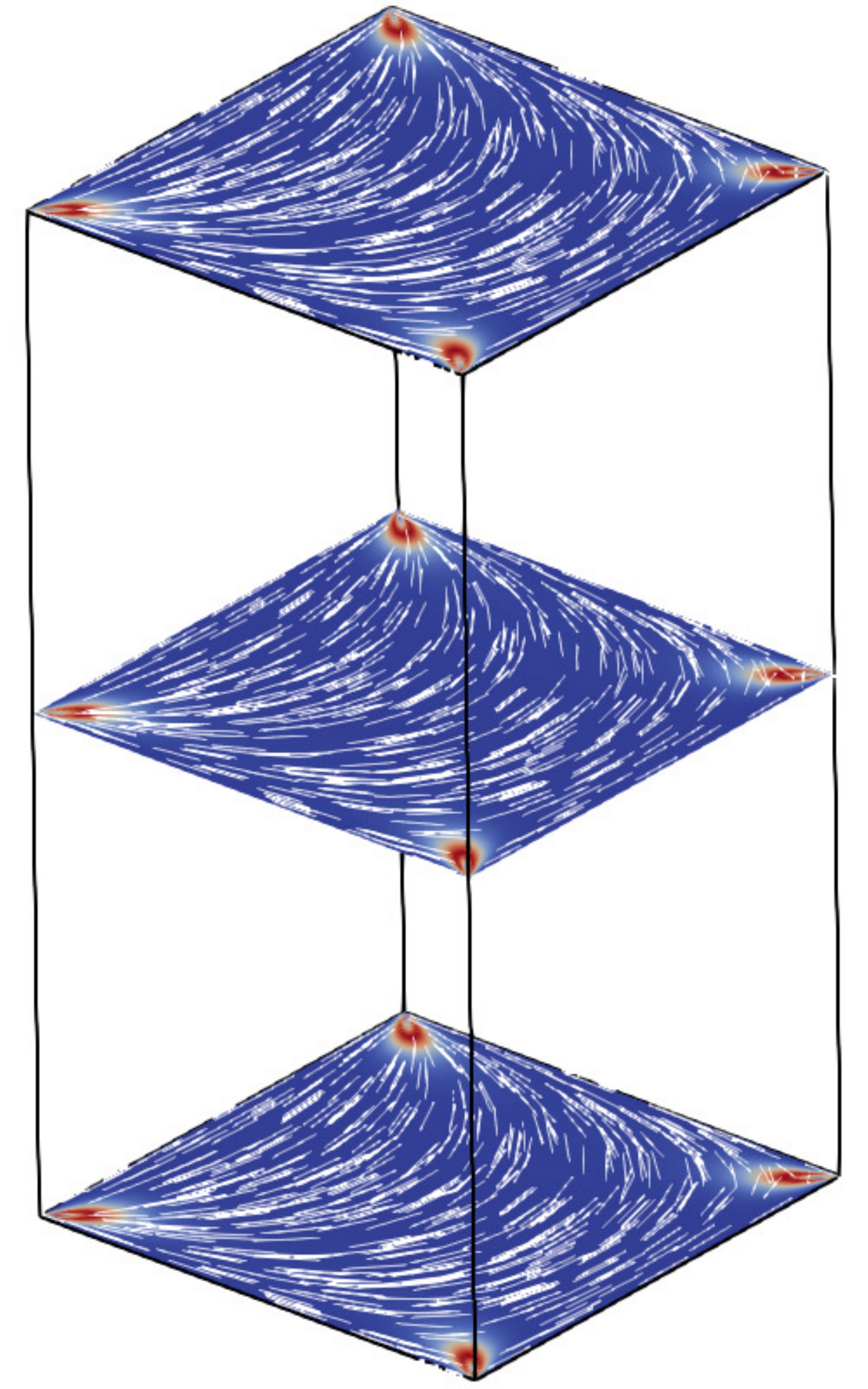}
      \put(-5, 100){(b)}
    \end{overpic}
    \hspace{1em}
    \begin{overpic}[height = 12.5em]{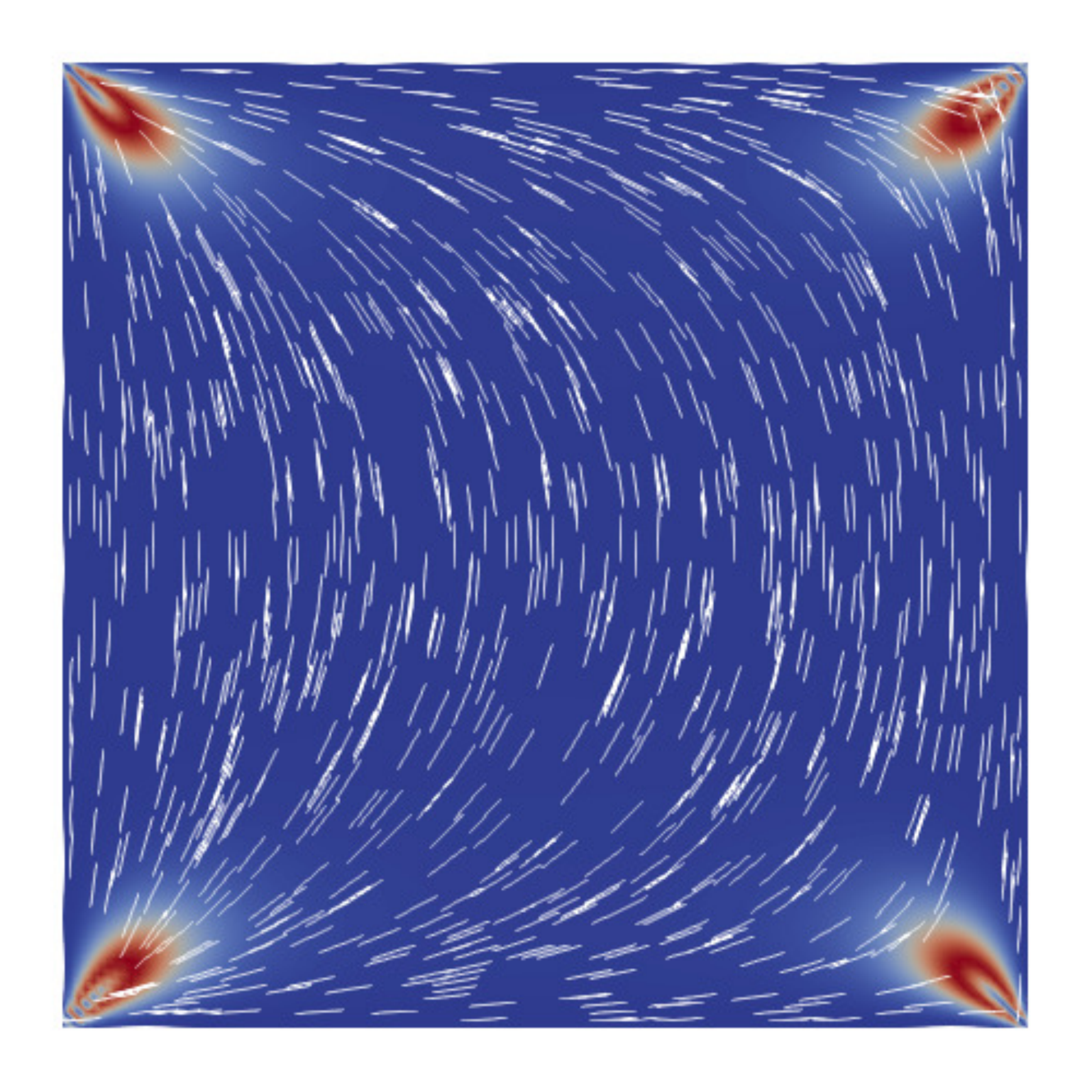}
    \end{overpic}
    \end{center}
    
    \vspace{2em}
    \begin{center}
    \begin{overpic}[height = 12.5em]{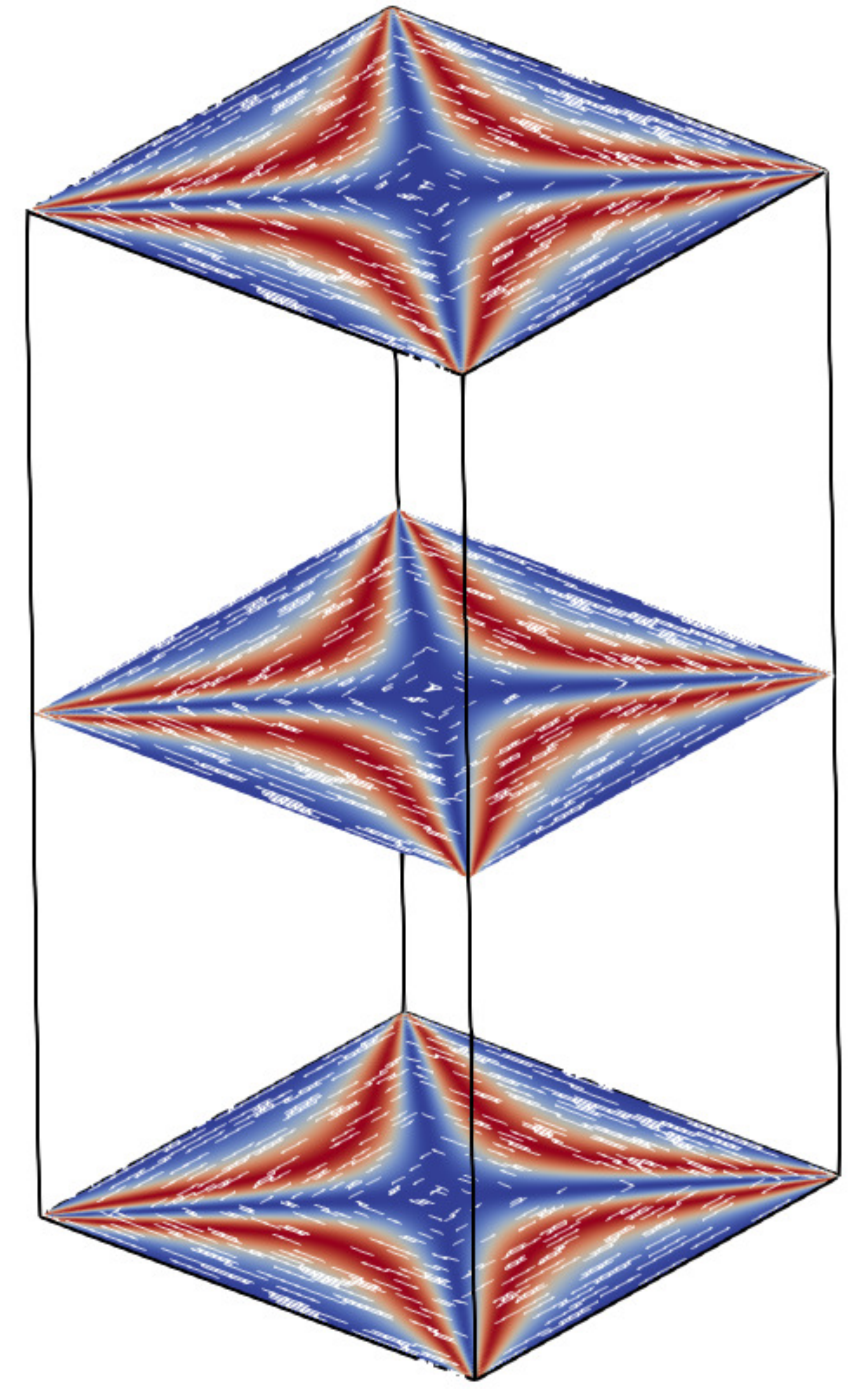}
      \put(-5, 100){(c)}
      \end{overpic}
        \hspace{1em}
    \begin{overpic}[height = 12.5em]{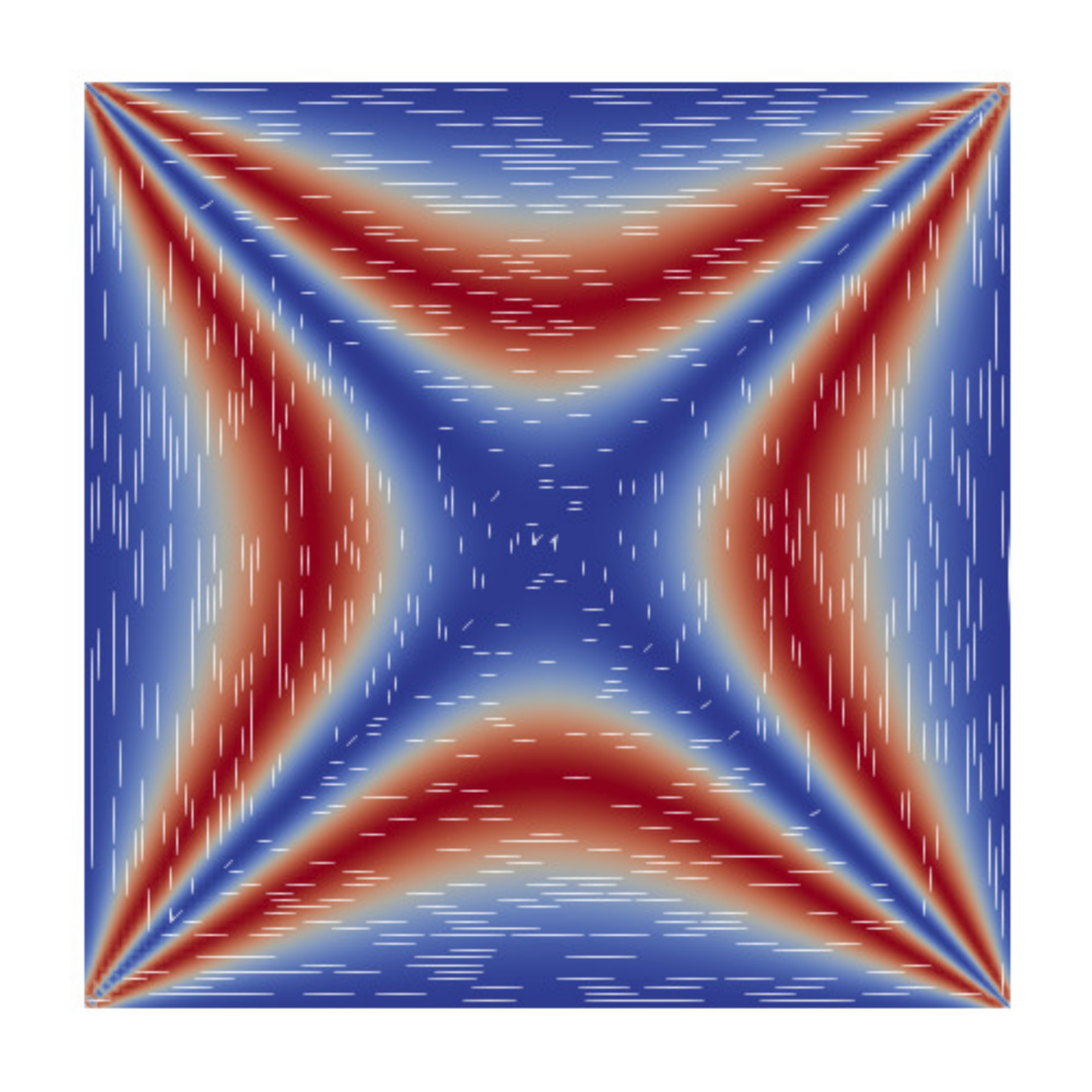}
    \end{overpic}
    \end{center}
  \caption{(a) A diagonal solution for $\bar{\lambda}^2 = 100$ and $\epsilon = 4$.
  (b) A rotated solution for $\bar{\lambda}^2 = 100$ and $\epsilon = 4$.
  (c) The WORS for $\bar{\lambda}^2 = 5$ and $\epsilon = 4$.
  The colors show the biaxiality parameter $\beta^2$, which are arranged such that the high to low values (one to zero) correspond to variations from red to blue. The white lines indicate the director direction $\nvec$.}\label{3D_WRD}
\end{figure}

\begin{figure*}[hbt]
  \centering
  \begin{overpic}[height = 12.5em]{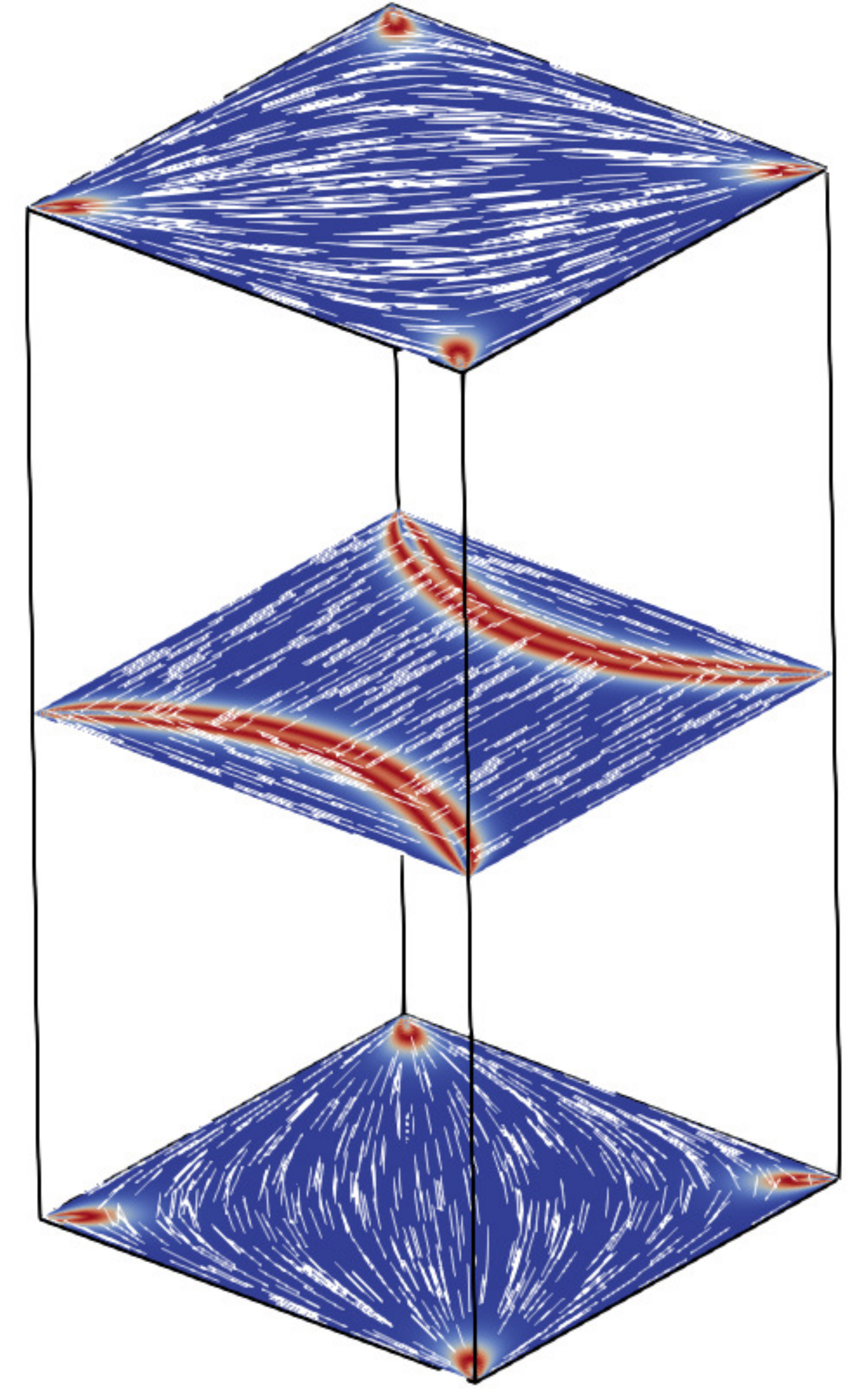}
    \put(-5, 95){(a)}
  \end{overpic}
  \begin{overpic}[height = 12.5em]{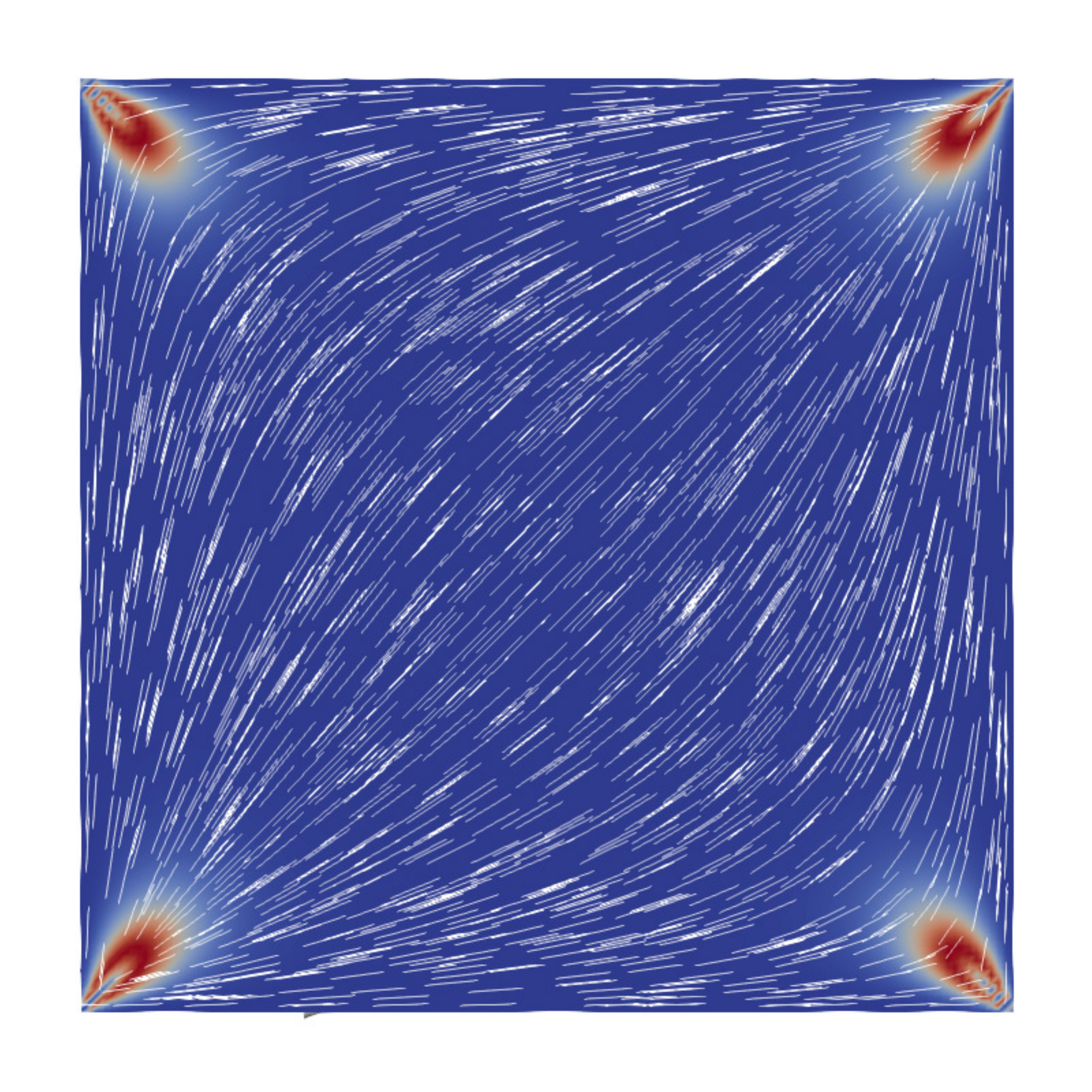}
  \end{overpic}
  \begin{overpic}[height = 12.5em]{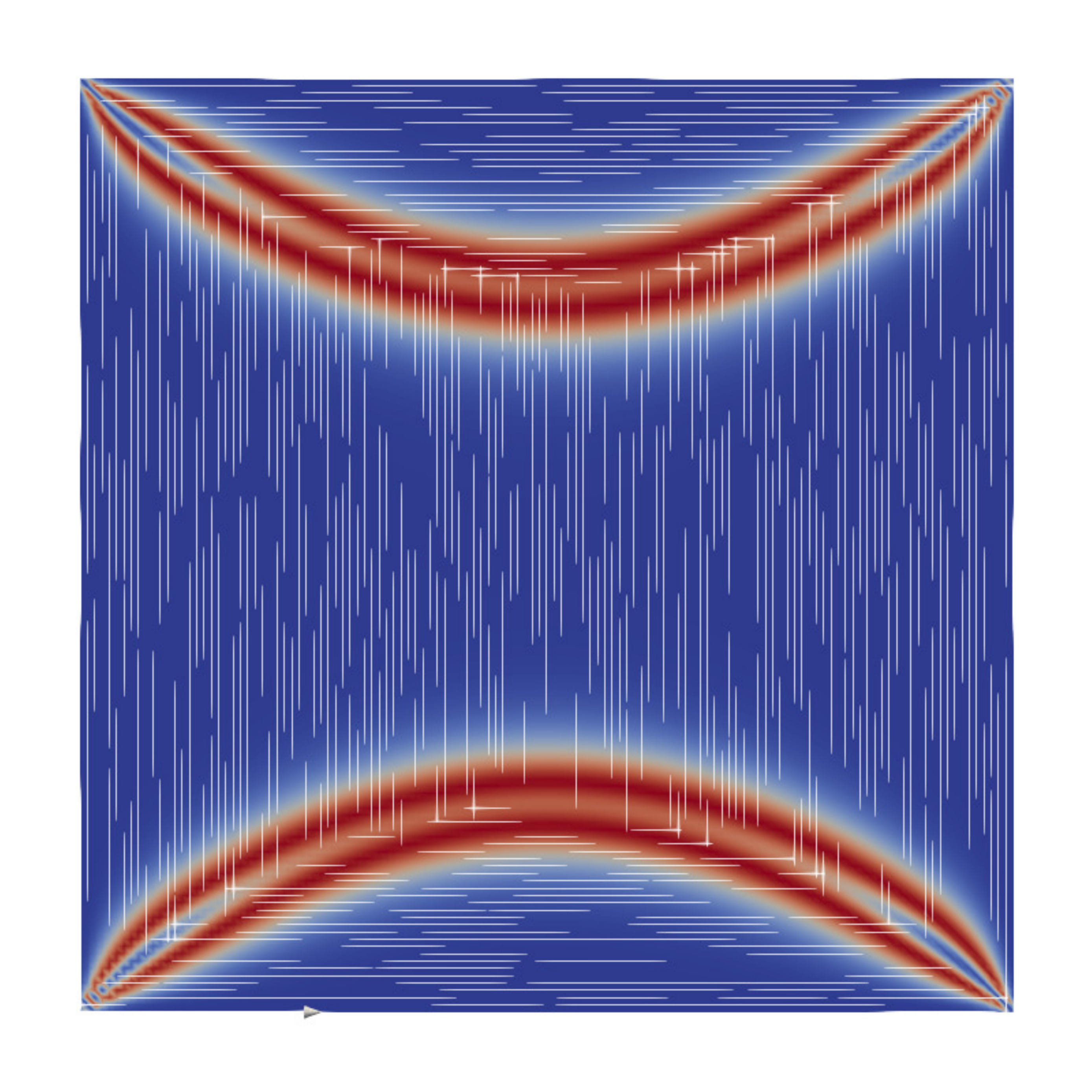}
  \end{overpic}
  \begin{overpic}[height = 12.5em]{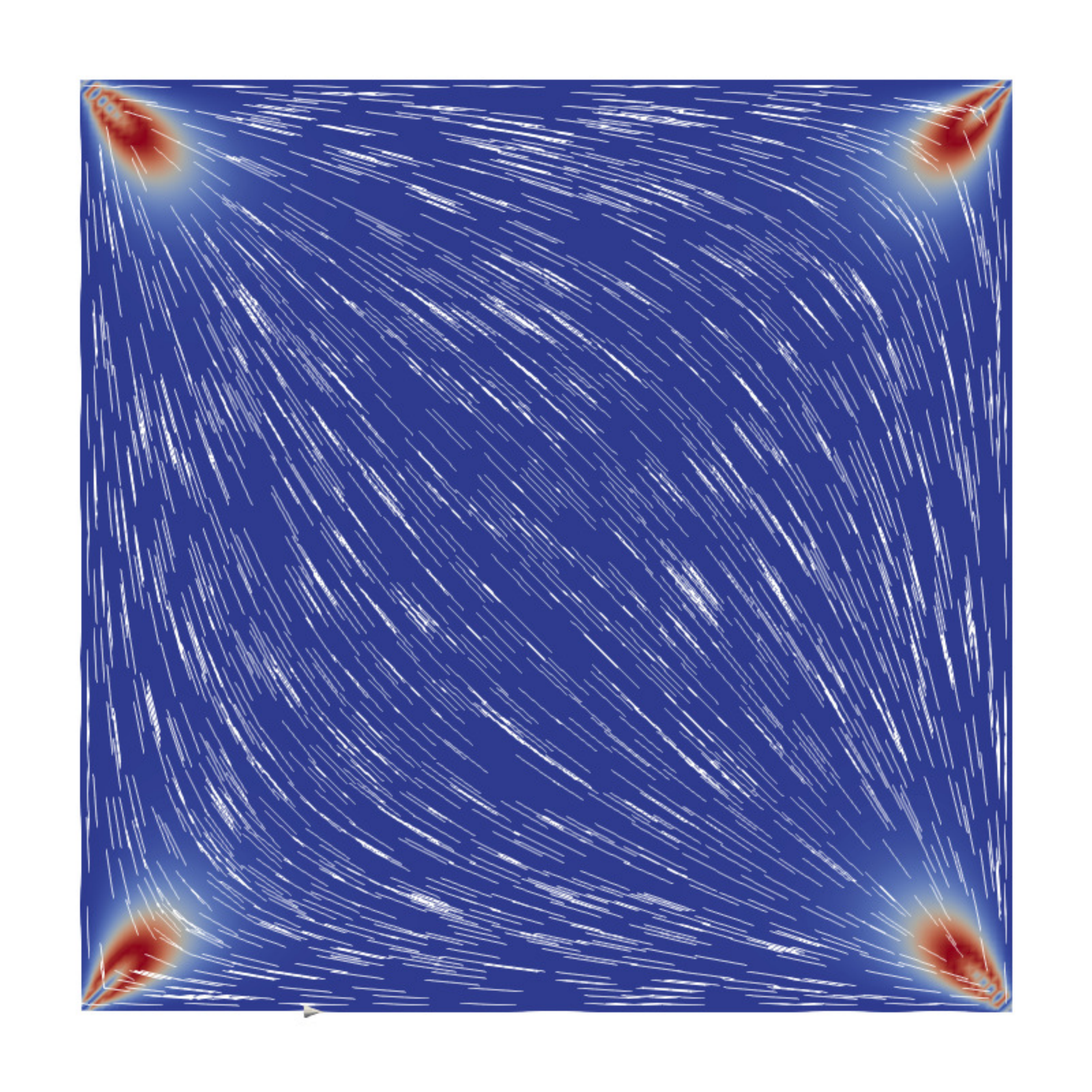}
  \end{overpic}

  \vspace{1em}
  \begin{overpic}[height = 12.5em]{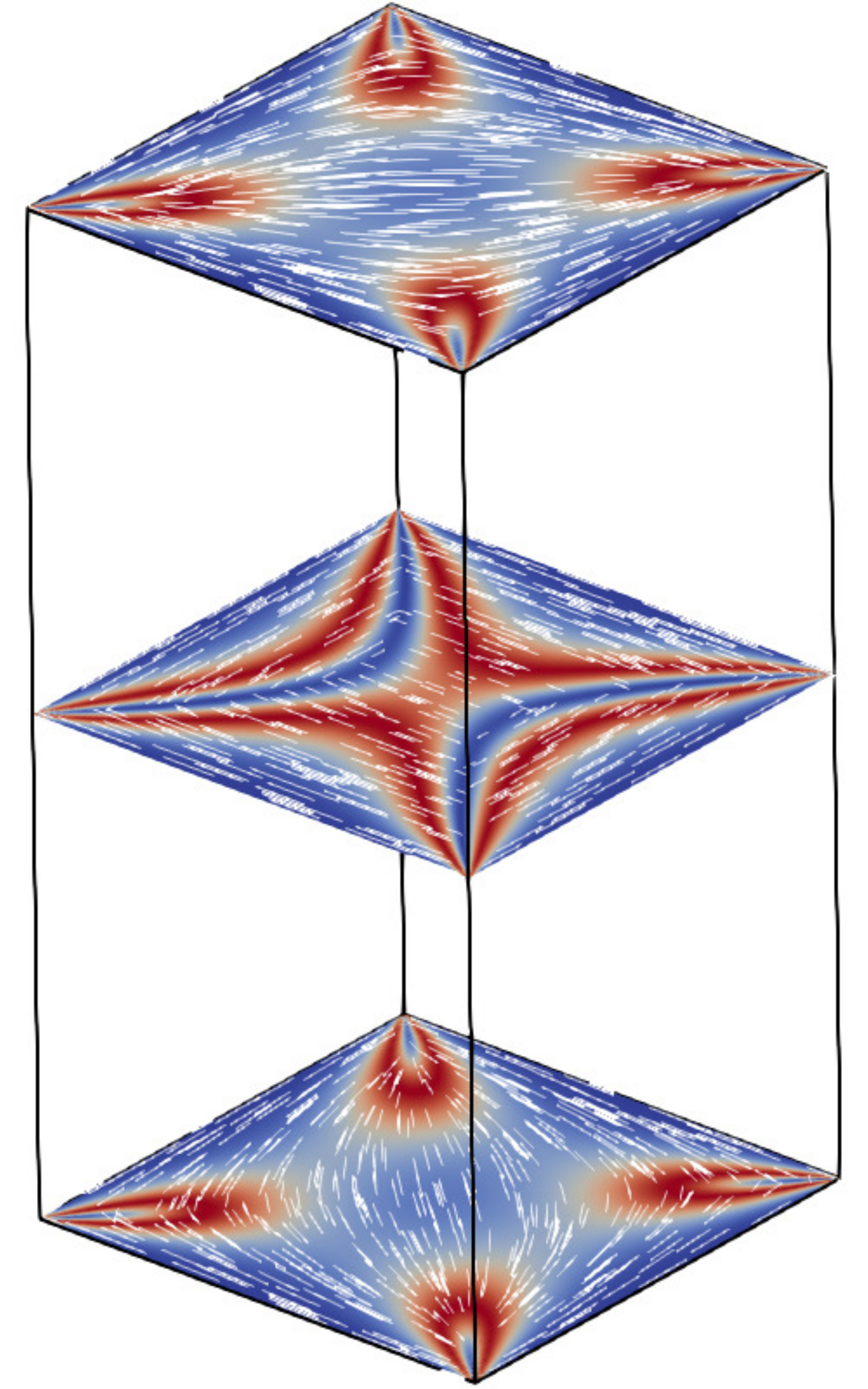}
    \put(-5, 95){(b)}
  \end{overpic}
  \begin{overpic}[height = 12.5em]{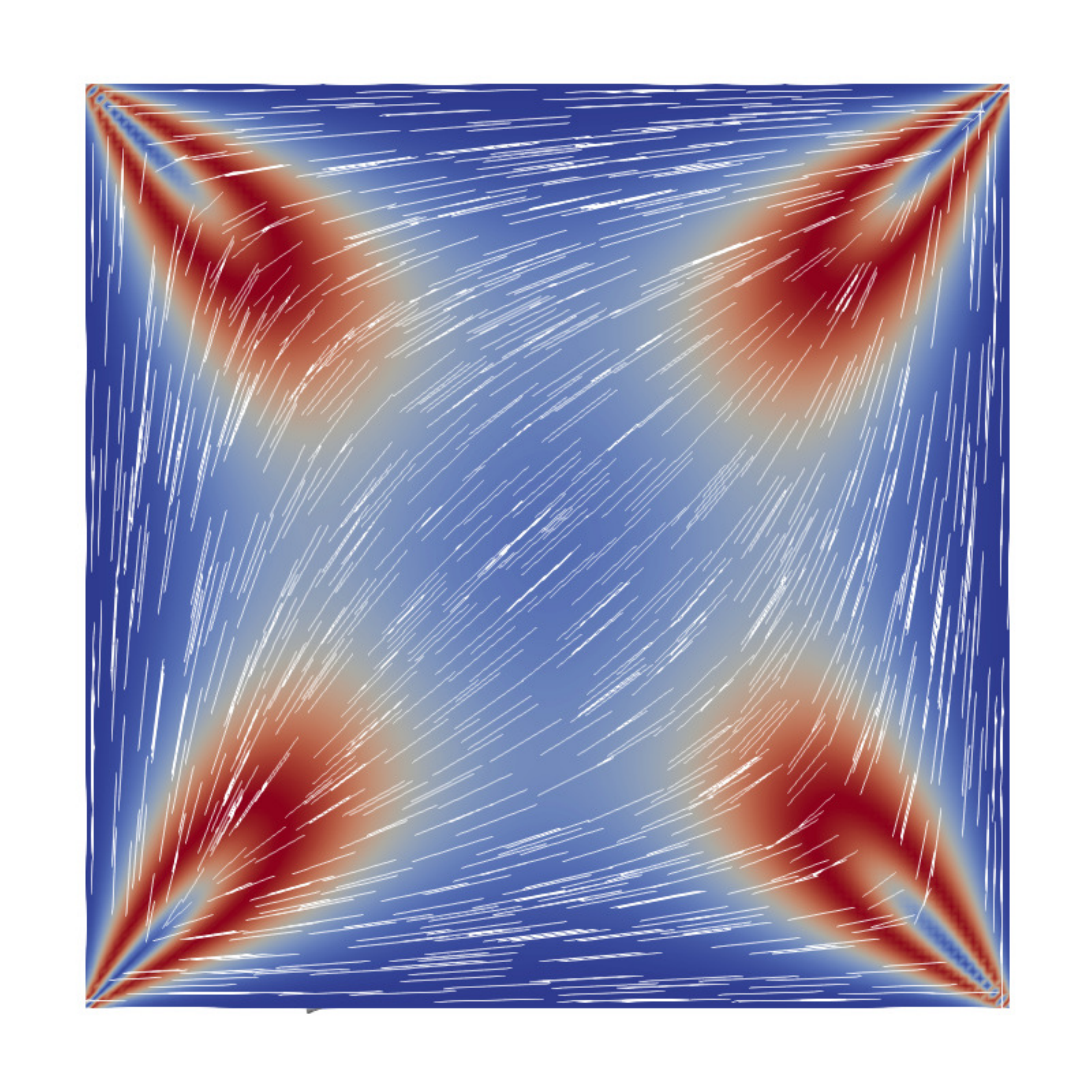}
  \end{overpic}
  \begin{overpic}[height = 12.5em]{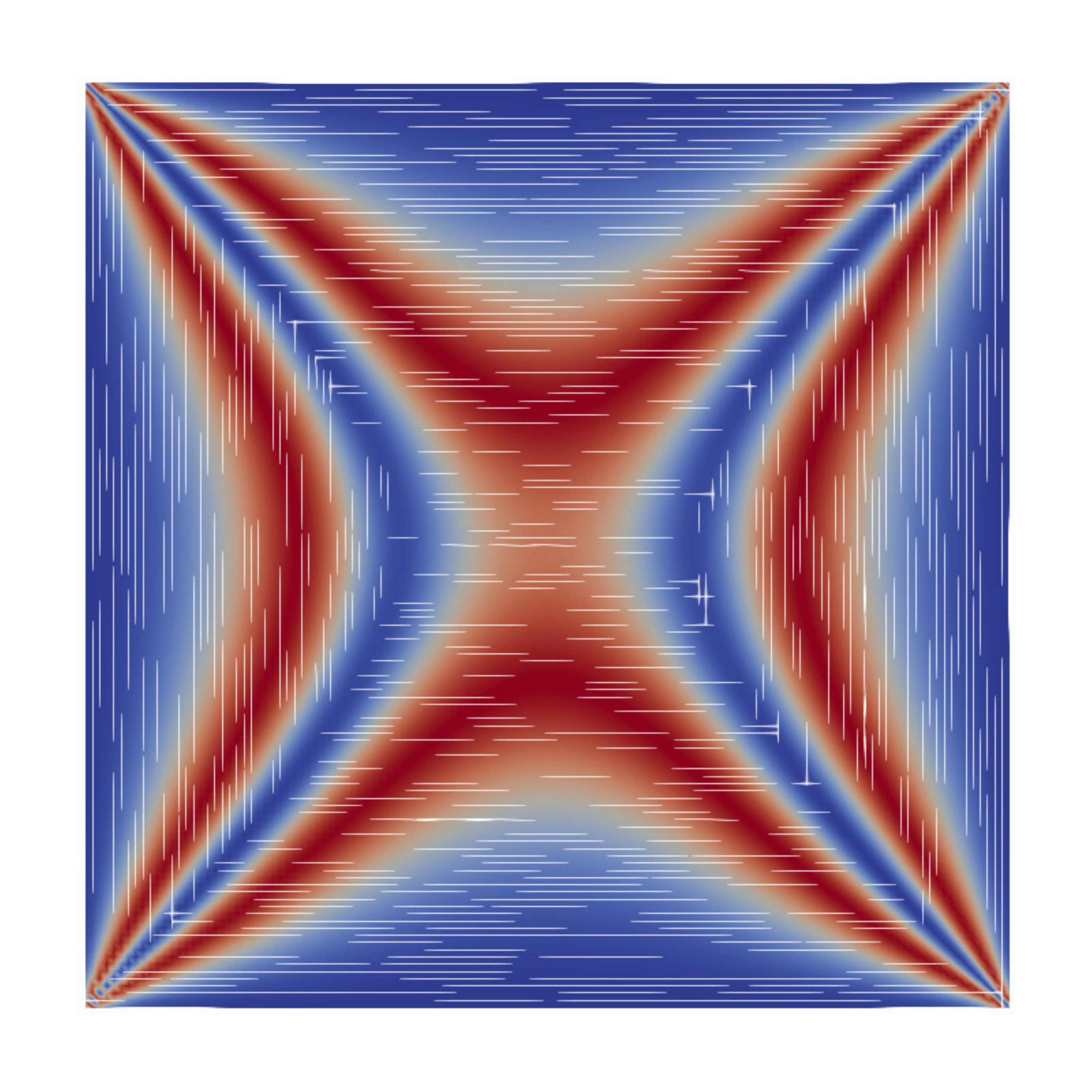}
  \end{overpic}
  \begin{overpic}[height = 12.5em]{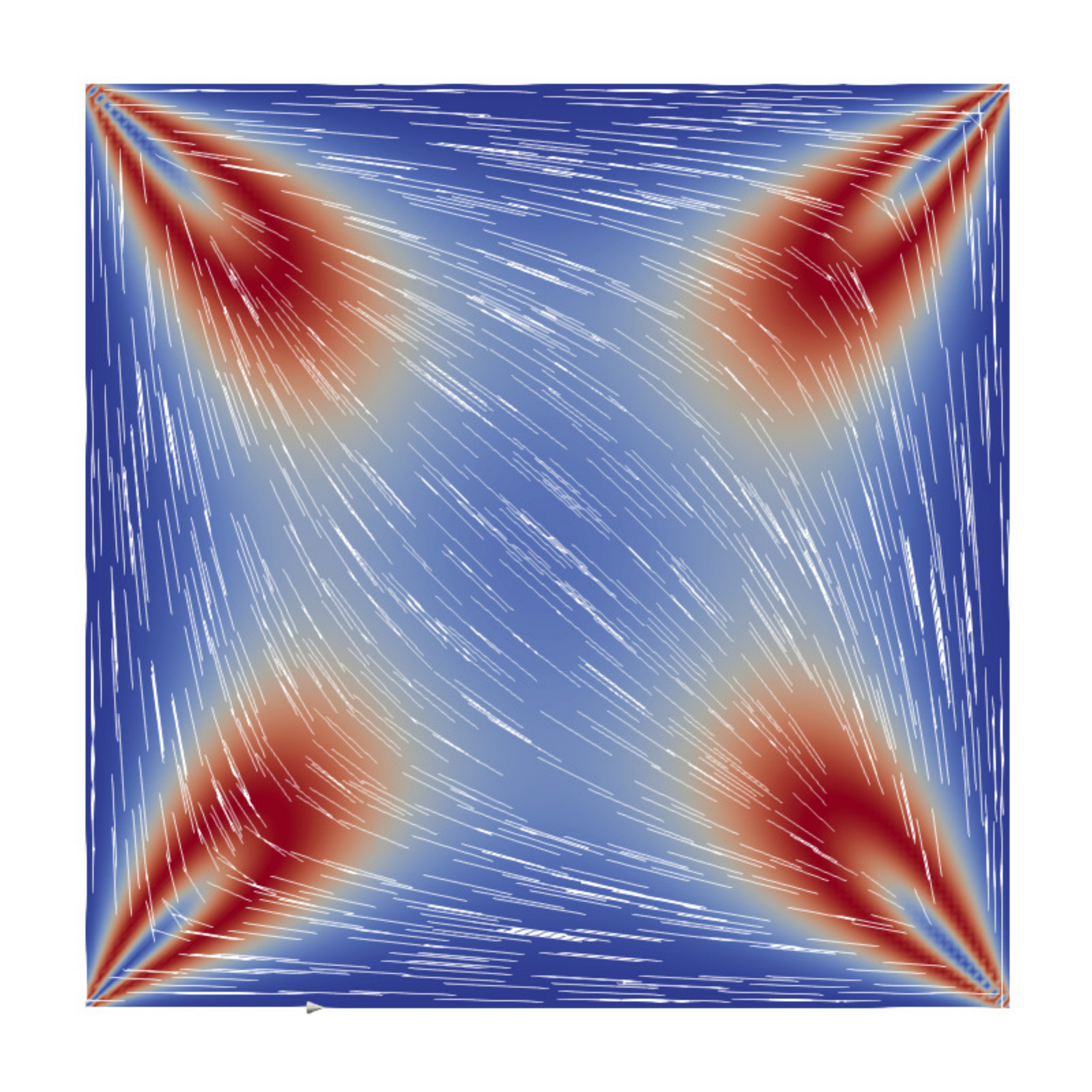}
  \end{overpic}
  \caption{(a) The locally stable mixed 3D solution for $\bar{\lambda}^2 = 100$ and $\epsilon = 4$, shown by 3D view and cross sections at  $z = 0$, $z = 2$, and $z = 4$ respectively. (b) The locally stable mixed 3D solution for $\bar{\lambda}^2 = 10$ and $\epsilon = 4$, shown by 3D view and cross sections at  $z = 0$, $z = 2$, and $z = 4$ respectively.  The colors show the biaxiality parameter $\beta^2$, which are arranged such that the high to low values (one to zero) correspond to variations from red to blue. The white lines indicate the director direction $\nvec$.}\label{3D_1}
\end{figure*}

For sufficiently small $\bar{\lambda}^2$, we always get the WORS for arbitrary $\epsilon$, in accordance with the uniqueness results for small $\lambda$ in previous sections. In Fig. \ref{3D_WRD}(c), we plot the WORS for $\bar{\lambda}^2 = 5$ and $\epsilon = 4$  with $A = - \frac{B^2}{3C}$.

Interestingly, for $\epsilon$ large enough, we can have additional mixed locally stable solutions for relatively large $\bar{\lambda}^2$. In Fig. \ref{3D_1}(a)-(b), we plot  a mixed 3D solution for $\bar{\lambda}^2 = 100$ and $10$, with $\epsilon = 4$. These mixed solutions can be obtained by taking a mixed initial condition as
$\Qvec = s_{+} \left(\nvec \otimes \nvec - \frac{1}{3} \mathbf{I}_3 \right)$ with
\begin{equation}
  \nvec(x, y, z) =
  \begin{cases}
    \frac{1}{\sqrt{2}}(1, ~1, 0),  \quad z \geq \frac{\epsilon}{2}  \\
    \frac{1}{\sqrt{2}}(1, -1, 0),  \quad z < \frac{\epsilon}{2}. \\
   \end{cases}
\end{equation}
The initial condition has two separate diagonal profiles on the top and bottom surfaces with a mismatch at the centre of the well, at $z = \frac{\epsilon}{2}$. In this case, the L-BFGS procedure converges to a locally stable solution that has different diagonal configurations on the top and bottom plates. On the middle slice, we have a BD-like profile (referring to the terminology in \cite{canevari2017order}), where the corresponding $\Qvec$ tensor is of the form
\[
\Qvec_{BD} = q_1 (\xhat \otimes \xhat - \yhat \otimes \yhat ) + q_3 \left( 2 \zhat \otimes \zhat - \xhat \otimes \xhat - \yhat \otimes \yhat \right)
\]
with two degrees of freedom, $q_3 <0$ on the middle slide and $q_1 = 0$ near a pair of parallel edges of the square cross-section ($q_1=0$ describes a transition layer between two distinct values of $q_1$). We compute the smallest eigenvalue of the Hessian matrix corresponding to this solution, which is positive and hence, this mixed solution is numerically stable. Indeed, these mixed solutions have lower free energy than rotated solutions for $\bar{\lambda}^2 = 100$ and $\epsilon = 4$.
Numerical simulations show that mixed solutions cease to exist when $\epsilon$ or $\bar{\lambda}^2$ is small enough. For $\bar{\lambda}^2 = 100$, we cannot find such solutions for $\epsilon \leq 0.8$. 
We can generate more 3D configurations by mixing diagonal and rotated configurations on the top and bottom surfaces or two different rotated solutions but these are unstable according to our numerical simulations. 

\subsubsection{Weak anchoring on the lateral surfaces}
In this section, we relax surface anchoring on the lateral surfaces and fix $W_z = 10^{-2} \mathrm{Jm}^{-2}$ on the top and bottom plates with $\epsilon = 0.2$.

In Fig. \ref{WORS_Weak}, we plot numerical solutions for $W_1 = W_2 = 10^{-2} \mathrm{Jm}^{-2}, 2 \times 10^{-3} \mathrm{Jm}^{-2}, 10^{-3} \mathrm{Jm}^{-2}$ and $10^{-4} \mathrm{Jm}^{-2}$, respectively, with $\bar{\lambda}^2 = 5$ and $\epsilon = 0.1$. All three solutions are obtained by using a diagonal-like initial condition.
In the strong anchoring case ($W_1 = W_2 = 10^{-2} \mathrm{Jm}^{-2}$), we get the WORS as expected, as the WORS is the unique critical point when $\bar{\lambda}^2$ is small enough. However, for $W_1 = W_2 = 10^{-3} \mathrm{Jm}^{-2}$, we get a diagonal-like solution in which maximum biaxiality is achieved around the corner. By further decreasing the anchoring strength, the nematic director is almost uniformly aligned along the diagonal direction. Similar results were reported in \cite{kralj2014order}.
\begin{figure}[!htb]
  \centering
  \begin{overpic}[width = 0.24\columnwidth]{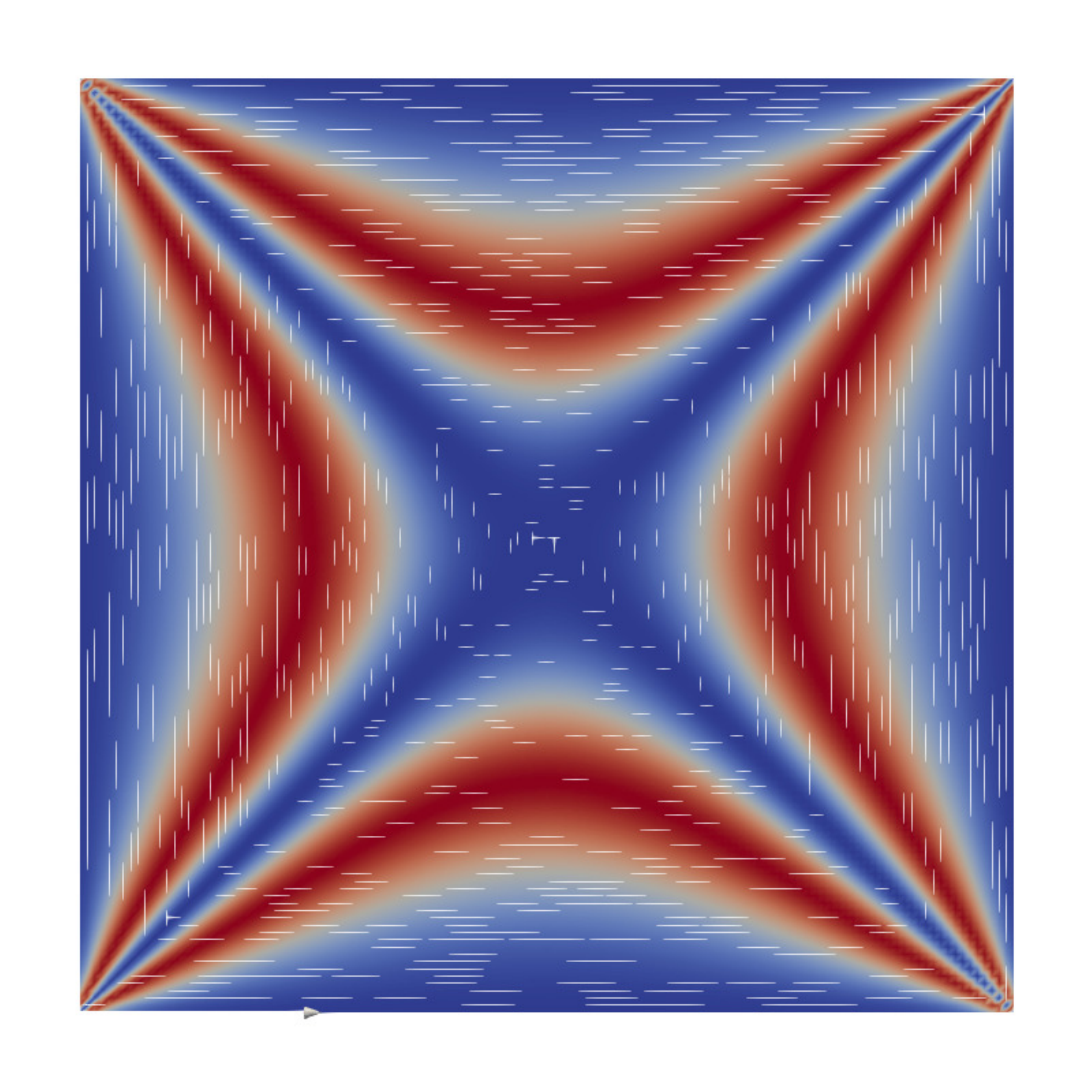}
    \put(-5, 95){(a)}
  \end{overpic}
  \hfill
  \begin{overpic}[width = 0.24\columnwidth]{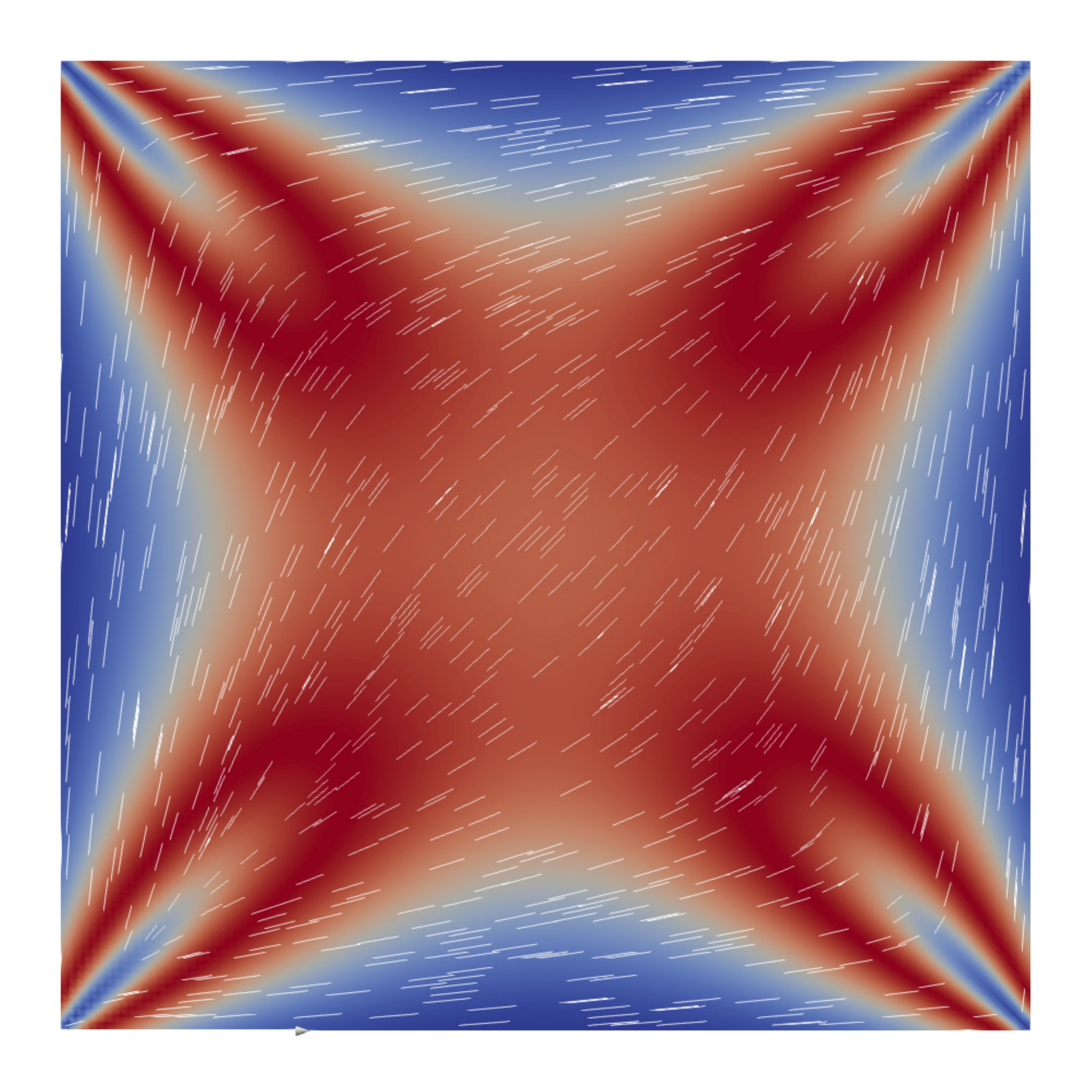}
    \put(-5, 95){(b)}
  \end{overpic}
  \hfill
  \begin{overpic}[width = 0.24\columnwidth]{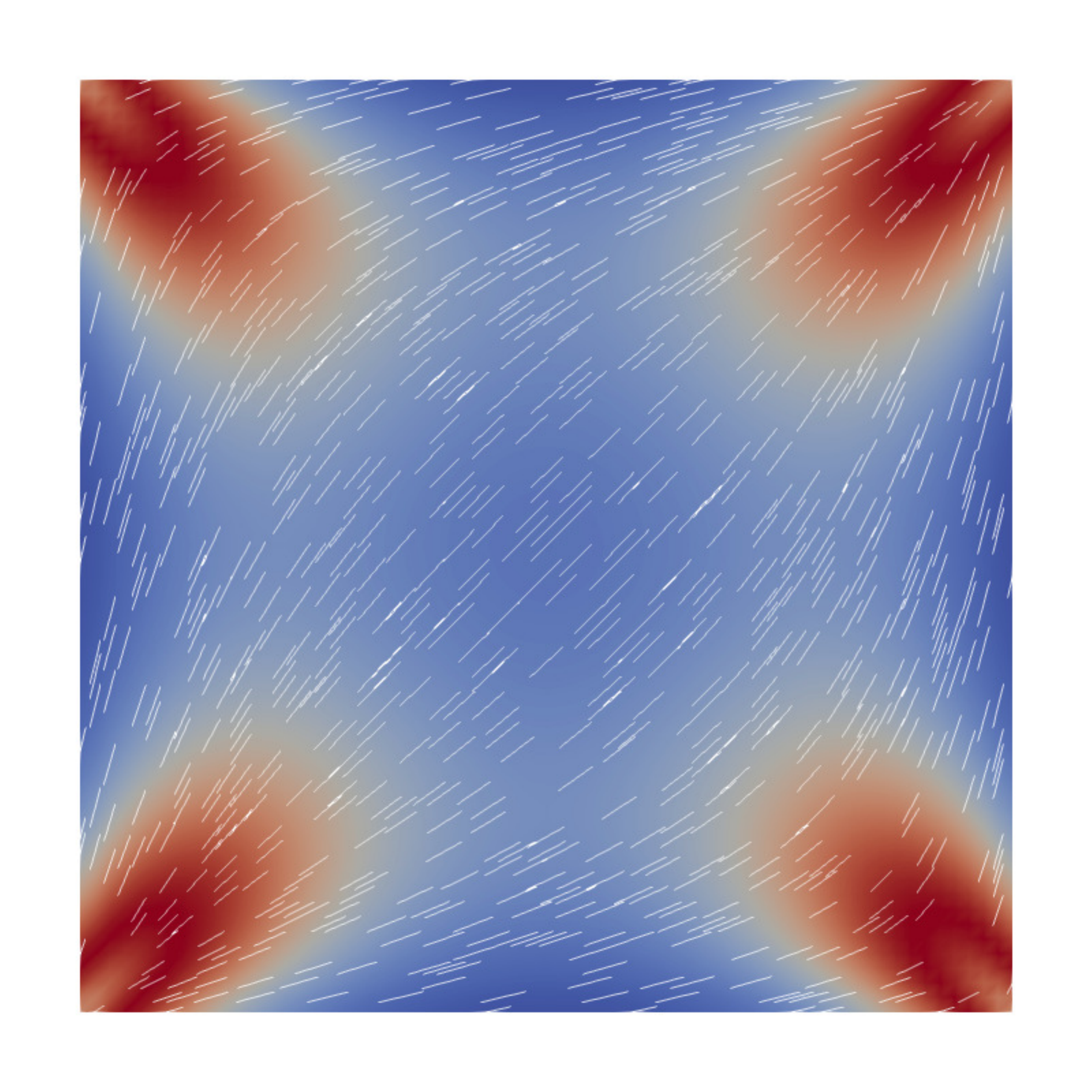}
    \put(-5, 95){(c)}
  \end{overpic}
  \hfill
  \begin{overpic}[width = 0.24\columnwidth]{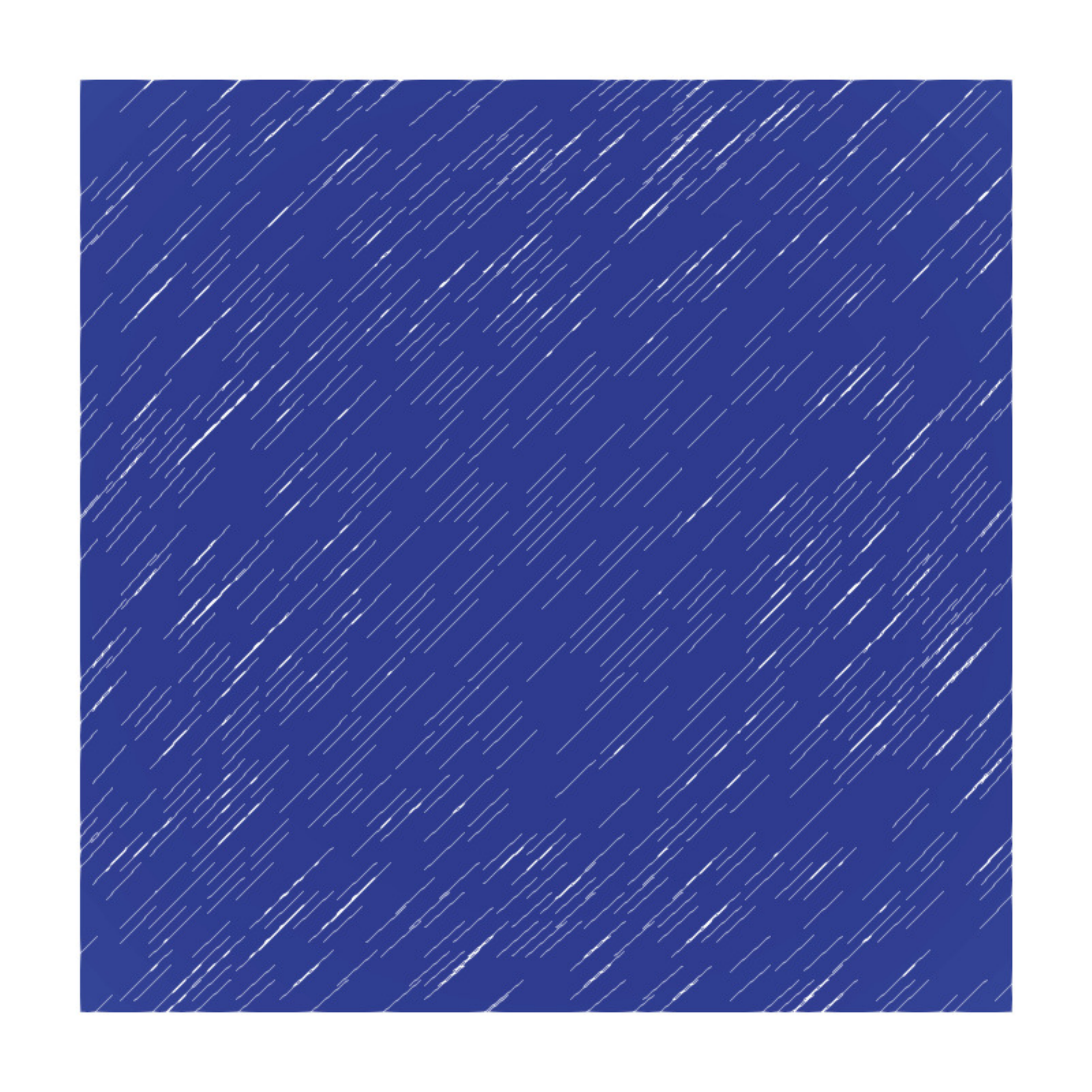}
    \put(-5, 95){(d)}
  \end{overpic}
  \caption{Transition from a diagonal solution to the WORS by increasing the anchoring strength on the lateral surfaces for $\bar{\lambda}^2 = 5$ and $\epsilon = 0.1$, shown by the cross-section at $\epsilon = 0.1$.  (a) $W_i = 10^{-2} \mathrm{Jm}^{-2}$; (c) $W_i = 10^{-3} \mathrm{Jm}^{-2}$; (d) $W_i = 10^{-4} \mathrm{Jm}^{-2}$.  The colors show the biaxiality parameter $\beta^2$, which are arranged such that the high to low values (one to zero) correspond to variations from red to blue. The white lines indicate the director direction $\nvec$.
  }\label{WORS_Weak}
\end{figure}

For $W_1 = W_2 = 10^{-3} \mathrm{Jm}^{-2}$, we can get the WORS by further decreasing $\bar{\lambda}^2$. However, the WORS ceases to exist for $W_1 = W_2 = 10^{-4} \mathrm{Jm}^{-2}$. Quantitatively, we can compute bifurcation points $\bar{\lambda}^{2}_{*}$, such that the WORS is the unique solution for $\bar{\lambda}^2 < \bar{\lambda}^{2}_{*}$, as a function of anchoring strength $W_1 = W_2 = W$, shown in Fig. \ref{Bifur}. We can find $\bar{\lambda}_{*}^2$ by decreasing $\bar{\lambda}^2$ till diagonal-like intial conditions converge to WORS, since diagonal solutions cease to exist for $\bar{\lambda}^2 < \bar{\lambda}^{2}_{*}$. The result in Fig. \ref{Bifur} is computed with $\epsilon = 0.1$. However, this result is independent of $\epsilon$ as both diagonal and the WORS are z-invariant solutions for $A = -\frac{B^2}{3C}$. 
\begin{figure}[!h]
  \centering
  \begin{overpic}[width = 0.6 \columnwidth]{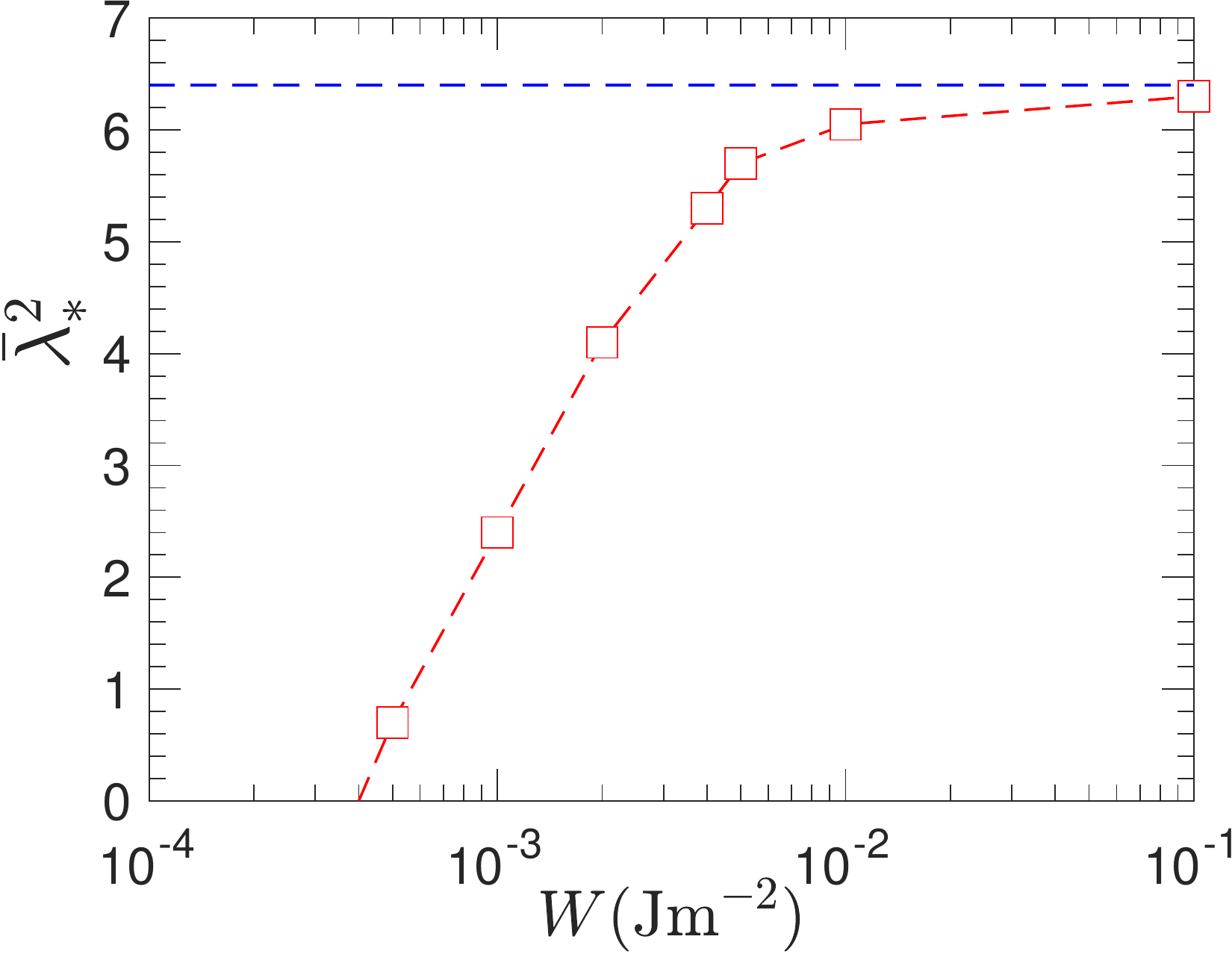}
  \end{overpic}
  \caption{Bifurcation points $\bar{\lambda}^{*}$, such that the WORS is the unique solution for $\bar{\lambda}^2 < \bar{\lambda}_{*}^2$, as a function of anchoring strength. The blue dashed line indicates the bifuration point of the WORS for Dirichlet boundary condition in a 2D square domain. }\label{Bifur}
\end{figure}


Another way to relax the surface anchoring is to consider the surface energy
\begin{equation}\label{relax_1}
  \begin{aligned}
    & \text{on}~~y = 0, 1: \\
    & f_s(\Qvec) = \omega_1   \left( \alpha  \left(\Qvec \xhat \cdot \xhat - \frac{2}{3} s_{+} \right)^2 + \gamma \big| \left( \mathbf{I} - \xhat \otimes \xhat \right) \Qvec \xhat \big|^2 \right);  \\
   & \text{on}~~x = 0, 1: \\
   & f_s(\Qvec) = \omega_2    \left( \alpha \left( \Qvec \yhat \cdot \yhat - \frac{2}{3} s_{+} \right)^2 + \gamma \big| \left( \mathbf{I} - \yhat \otimes \yhat \right) \Qvec \yhat \big|^2 \right),\\
 \end{aligned}
\end{equation}
where $\omega_i = \frac{W_i \lambda}{L}$ is the non-dimensionalized anchoring strength, $\alpha > 0$ and $\gamma > 0$ are constants. The second term in the surface energy (\ref{relax_1}) forces $\xhat$ ($\yhat$) to be an eigenvector of $\Qvec$ on the plane $y = 0,~1$ ($x = 0,~1$), while the first term forces the eigenvalue associated with  $\xhat$ ($\yhat$) to be $\frac{2}{3}s_{+}$. Since the second term in (\ref{relax_1}) can be zero if we take $\Qvec = s_{+} \left( \zhat \otimes \zhat - \frac{1}{3} \mathbf{I} \right)$, which also makes the surface energy on the top and bottom plates ($\Gamma$) zero, we  keep $\alpha$ non-zero to get interesting defect patterns.

We fix $W_i = 10^{-2} \mathrm{Jm}^{-2}$ and vary $\alpha$ and $\gamma$ to relax the anchoring on the lateral surfaces. 
Fig. \ref{Res_Relax_1} shows the numerical result for $\epsilon = 0.1$ and $\bar{\lambda}^2 = 5$ with various $\alpha$ and $\gamma$, by using diagonal-like initial conditions.
For $\alpha = \gamma = 1$ and $W_i = 10^{-2} \mathrm{Jm}^{-2}$, we can get a WORS-like solution, which has a strong biaxial region near the boundary. The biaxial regions near the boundary become larger as $\alpha$ gets smaller.  We then fix $\alpha = 1$ and vary $\gamma$. If $\gamma$ is small enough, the nematic director (the leading eigenvector of the $\Qvec$-tensor) is not tangent to the square edges and the WORS ceases to exist.
\begin{figure}[h]
  \centering
  \begin{overpic}[width = 0.24\columnwidth]{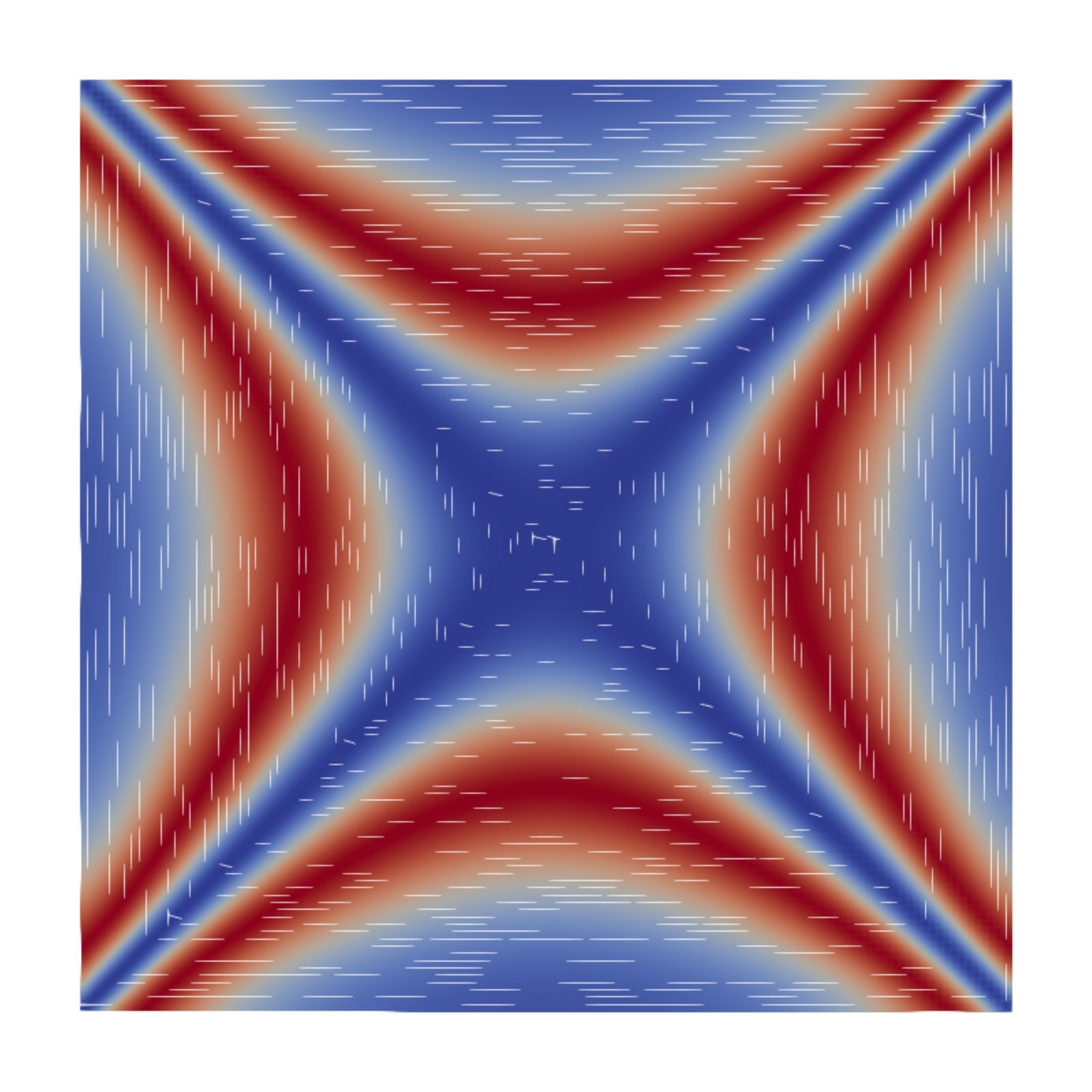}
    \put(-5, 95){(a)}
  \end{overpic}
  \hfill
  \begin{overpic}[width = 0.24\columnwidth]{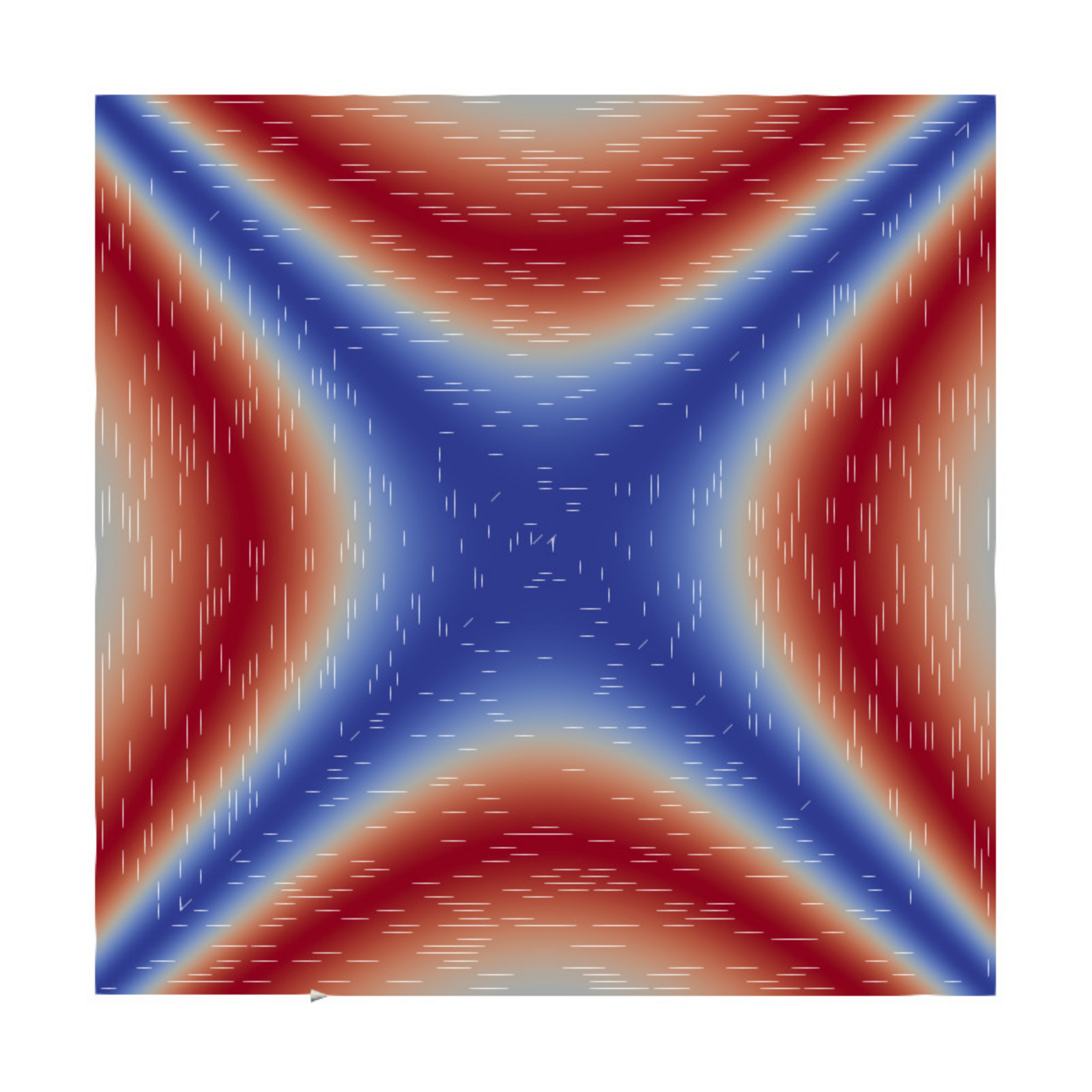}
    \put(-5, 95){(b)}
  \end{overpic}
  \hfill
  \begin{overpic}[width = 0.24\columnwidth]{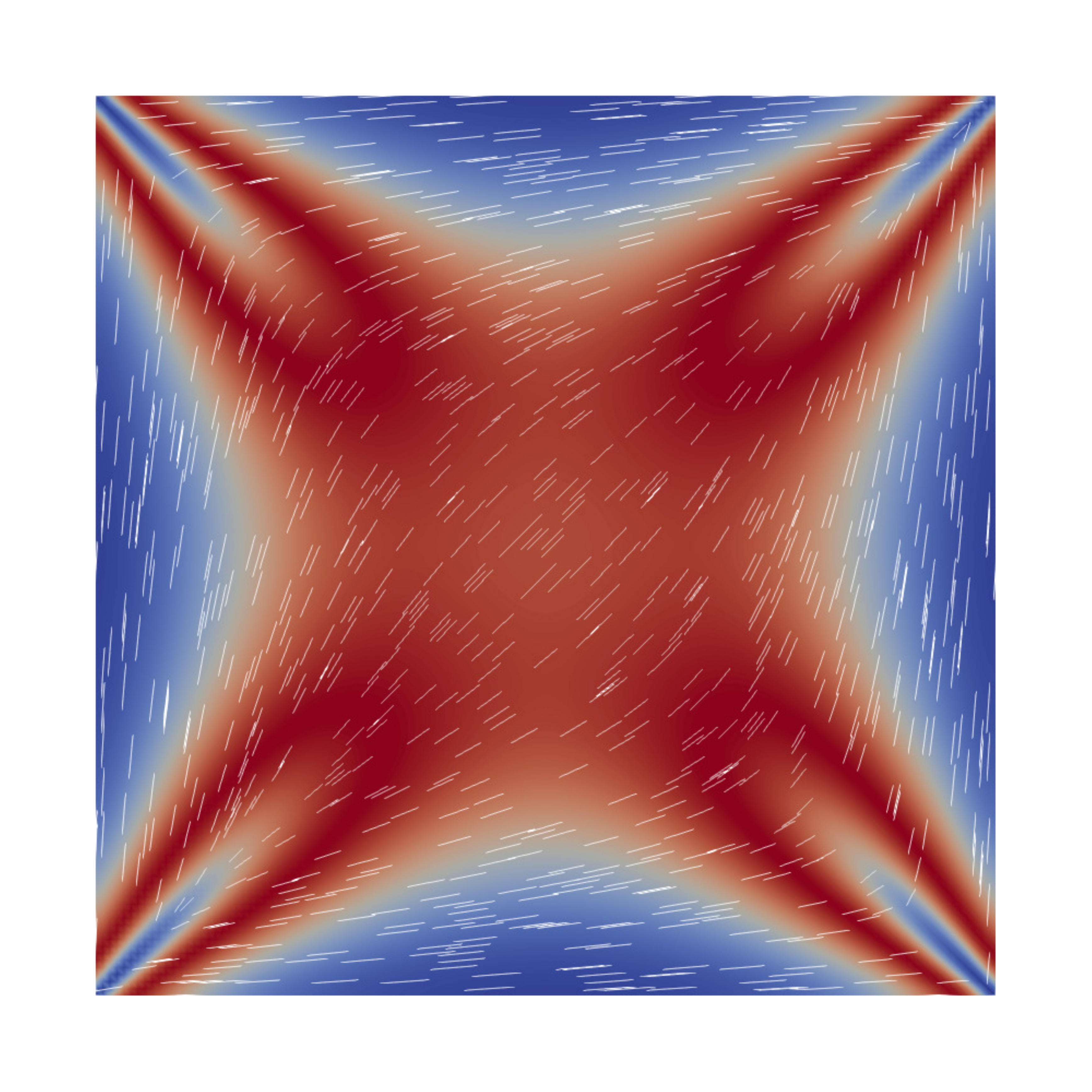}
    \put(-5, 95){(c)}
  \end{overpic}
  \hfill
  \begin{overpic}[width = 0.24\columnwidth]{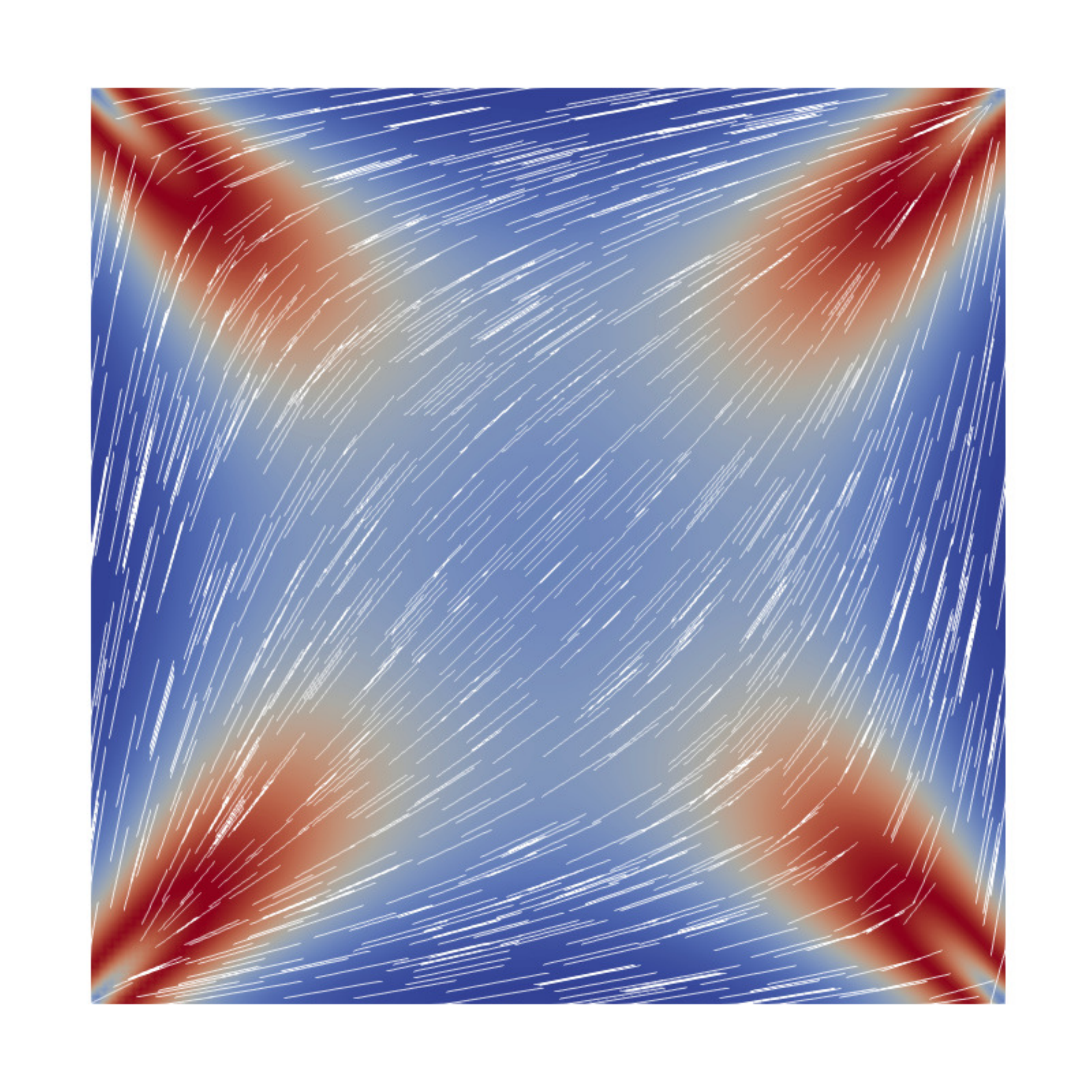}
    \put(-5, 95){(d)}
  \end{overpic}
  \caption{Numerical solutions for surface energy (\ref{relax_1}) with a diagonal-like initial condition for $\bar{\lambda}^2 = 5$ and $\epsilon = 0.2$, shown by the cross-section at $z = 0.1$: (a) $\alpha = \gamma = 1$; (b) $\alpha = 0.1$ and $\gamma = 1$; (c) $\alpha = 1$ and $\gamma = 0.2$; (d) $\alpha = 1$ and $\gamma = 0.1$. The colors show the biaxiality parameter $\beta^2$, which are arranged such that the high to low values (one to zero) correspond to variations from red to blue. The white lines indicate the director direction $\nvec$.}\label{Res_Relax_1}
\end{figure}

The above examples show that the WORS ceases to exist if the anchoring on the lateral surfaces is weak enough, and we always get a diagonal-like solution when the WORS ceases to exist.
It should be remarked that the diagonal-like solutions tend to be defect-free around the corners with weak anchoring, as the nematic directors aren't forced to be tangential
to the square edges and there is no biaxial-uniaxial or biaxial-isotropic interface near the corners.  

%

\subsubsection{Escaped Solutions}
In Ref. \cite{wang2018order}, the authors show that  there exists two escaped solutions with non-zero $q_4$ and $q_5$, and $q_3 > 0$, in the reduced 2D square domain for relatively large $\bar{\lambda}^2$. Our simulations show that these two escaped solutions can exist in 3D wells for similar values of $\bar{\lambda}^2$ if the anchoring strength on the top and bottom plates are weak enough and the escaped solutions are numerically locally stable. Fig. \ref{3D_escaped} (a)-(b) show the nematic director and biaxiality parameter in the middle slices of these two types of escaped configurations for $\bar{\lambda}^2 = 100$, $\epsilon = 4$ and $W_z = 10^{-5} \mathrm{Jm}^{-2}$, which are quite similar to the escaped configurations in a cylindrical cavity \cite{kralj1993stability}.  The value of $q_3$ in configuration \ref{3D_escaped}(a) is plotted in Fig. \ref{3D_escaped} (d) and $q_3 > 0$ in the center of well. 
\begin{figure}[!htb]
  \begin{center}
    \begin{overpic}[height = 12.5em]{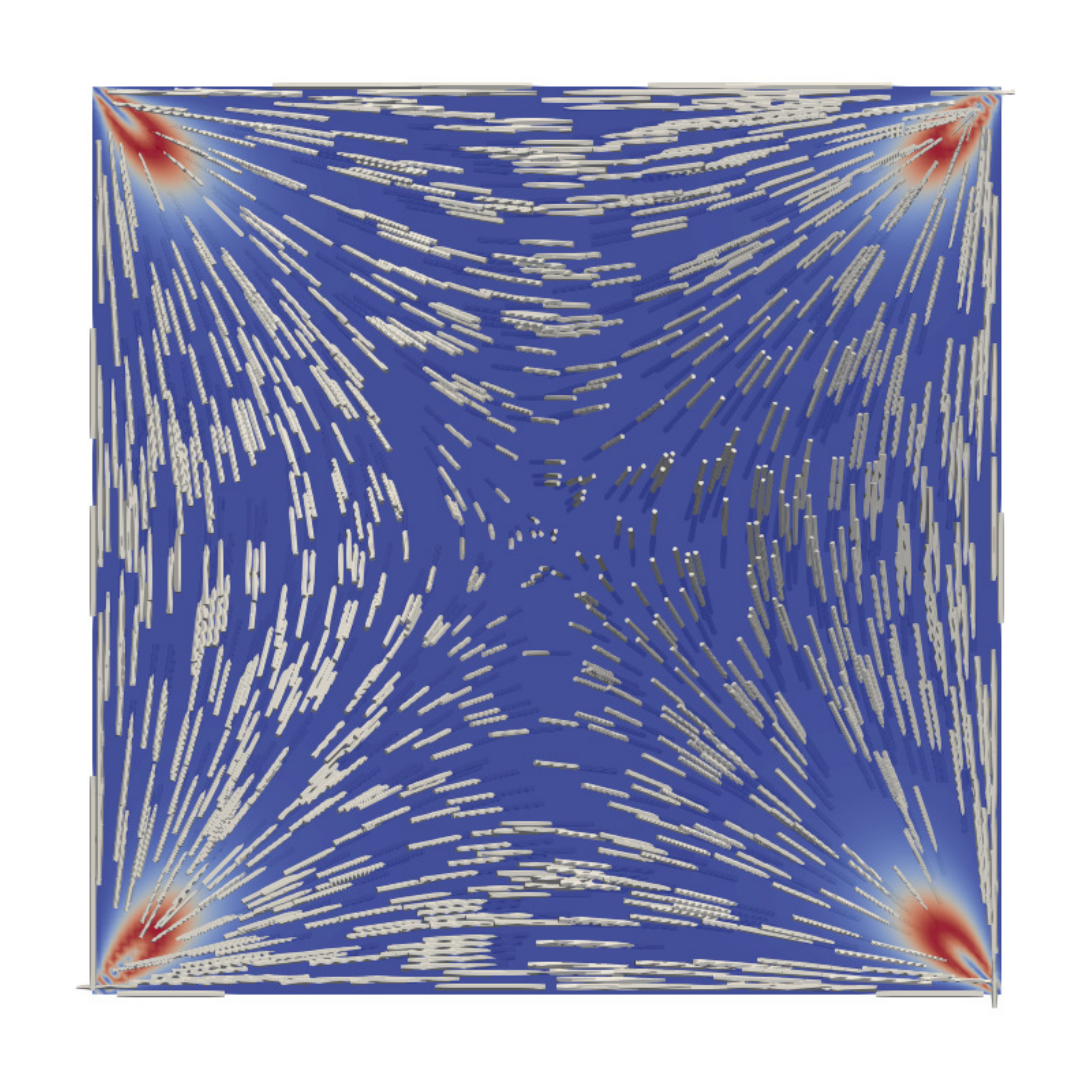}
      \put(-5, 95){(a)}
    \end{overpic}
    \hspace{1em}
    \begin{overpic}[height = 12.5em]{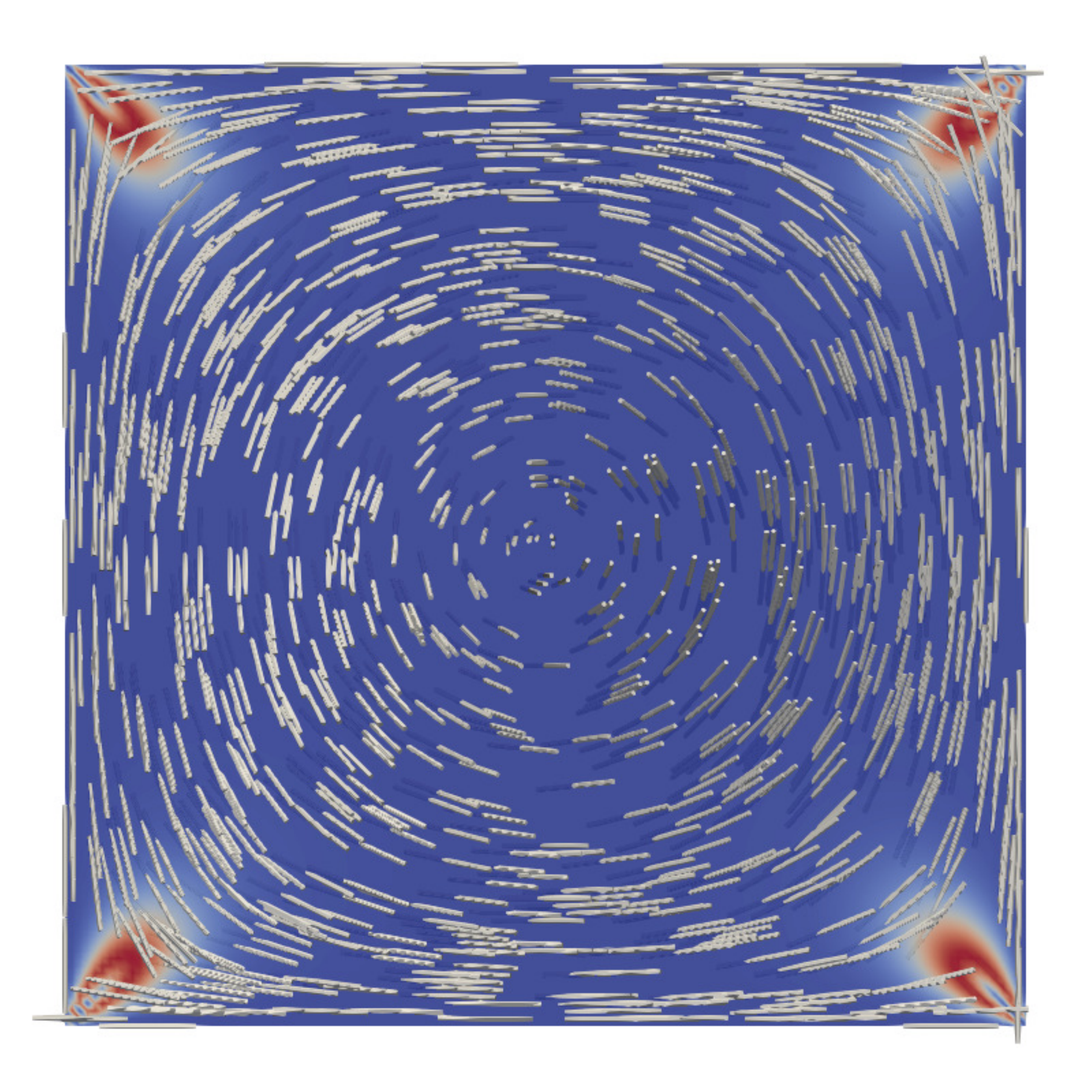}
      \put(-5, 95){(b)}
    \end{overpic}
   \hspace{1em}
    \begin{overpic}[height = 12.5em]{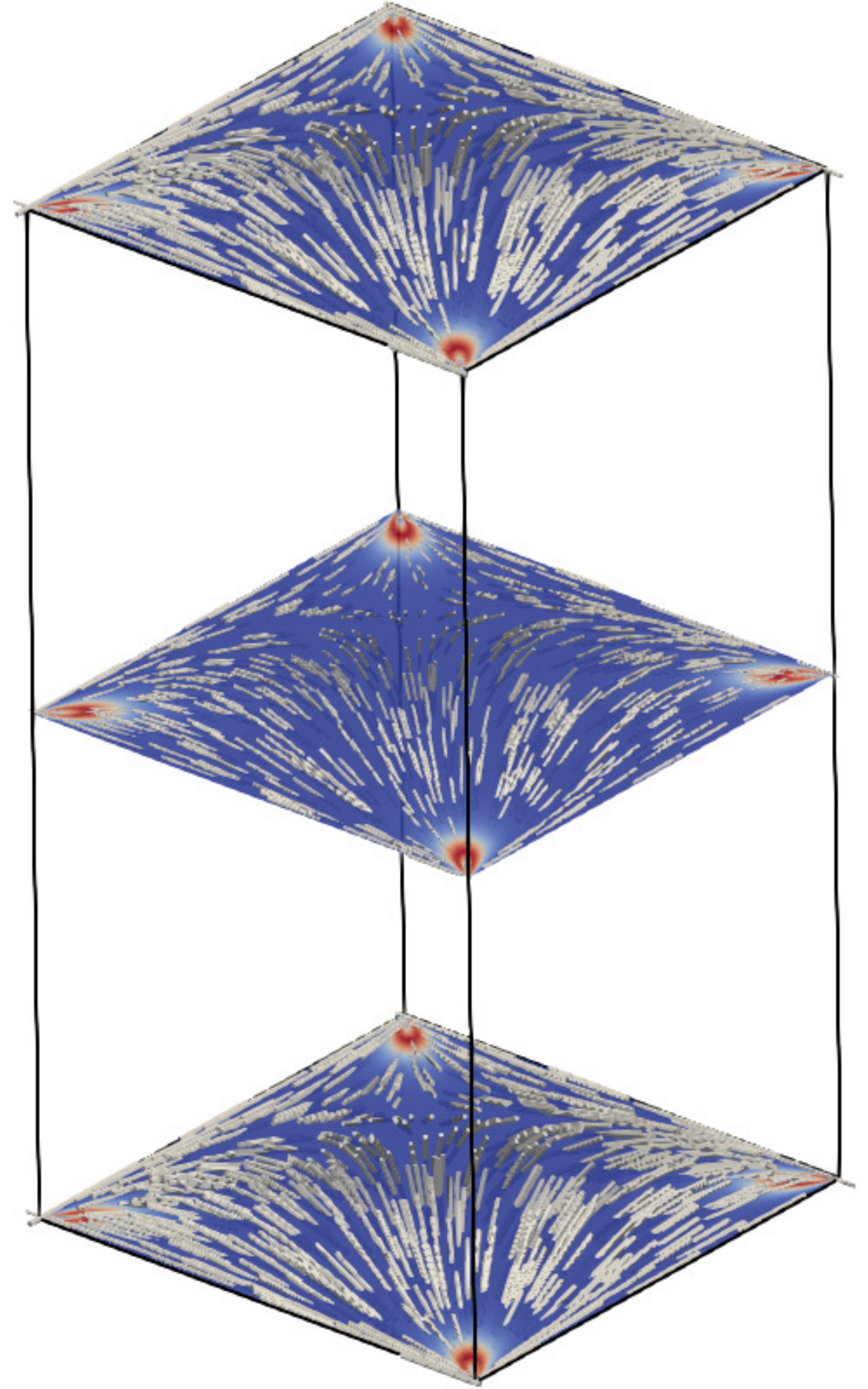}
      \put(-5, 100){(c)}
  \end{overpic}
    \hspace{1em}
        \begin{overpic}[height = 12.5em]{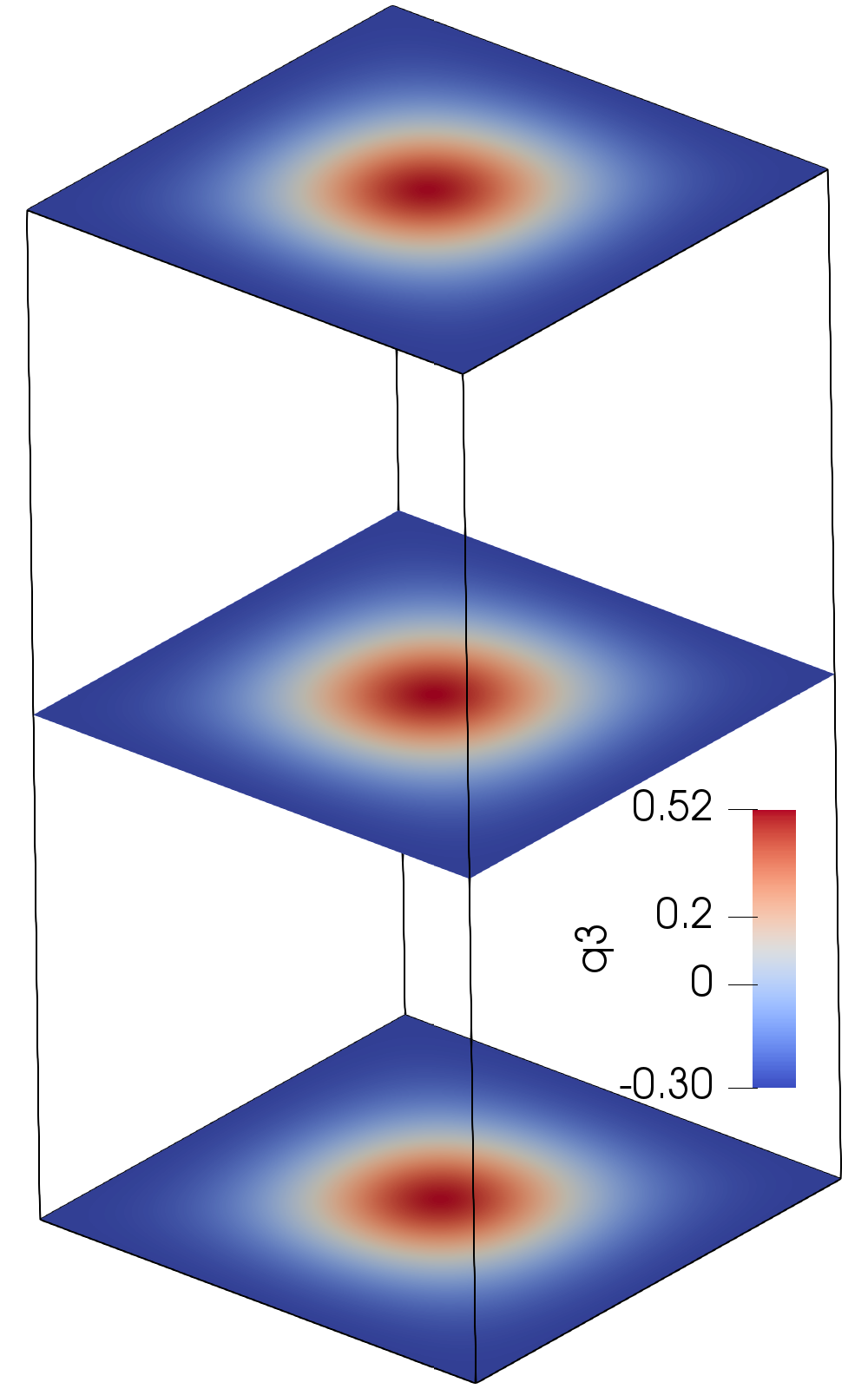}
          \put(-5, 100){(d)}
        \end{overpic}
    \end{center}
  \caption{(a) Middle slice in the escaped configuration with $-1$-disclination line in the center for $\bar{\lambda}^2 = 100$ and $\epsilon = 4$. (b) Middle slice in the escaped configuration with $+1$-disclination line in the center for $\bar{\lambda}^2 = 100$ and $\epsilon = 4$. The colors show the biaxiality parameter $\beta^2$, which are arranged such that the high to low values (one to zero) correspond to variations from red to blue. The white rods indicate the director direction $\nvec$. (c)-(d) 3D view and $q_3$ in the escaped configuration with $-1$-disclination line in the center for $\bar{\lambda}^2 = 100$ and $\epsilon = 4$.  }\label{3D_escaped}
\end{figure}

Strictly speaking, the two configurations in Fig. \ref{3D_escaped} are not z-invariant if $W_z \neq 0$. The escaped solutions cease to exist if either $\epsilon$ is small enough or if the anchoring $W_z$ is large enough.
We can compute the critical achoring strength $W_z$ on the top and bottom plates, for which the escaped configurations cease to exist, as a function of $\epsilon$ for $\bar{\lambda}^2 = 100$, shown in Fig. \ref{ES_Height}. 
\begin{figure}[!h]
  \centering
  \begin{overpic}[width = 0.6\columnwidth]{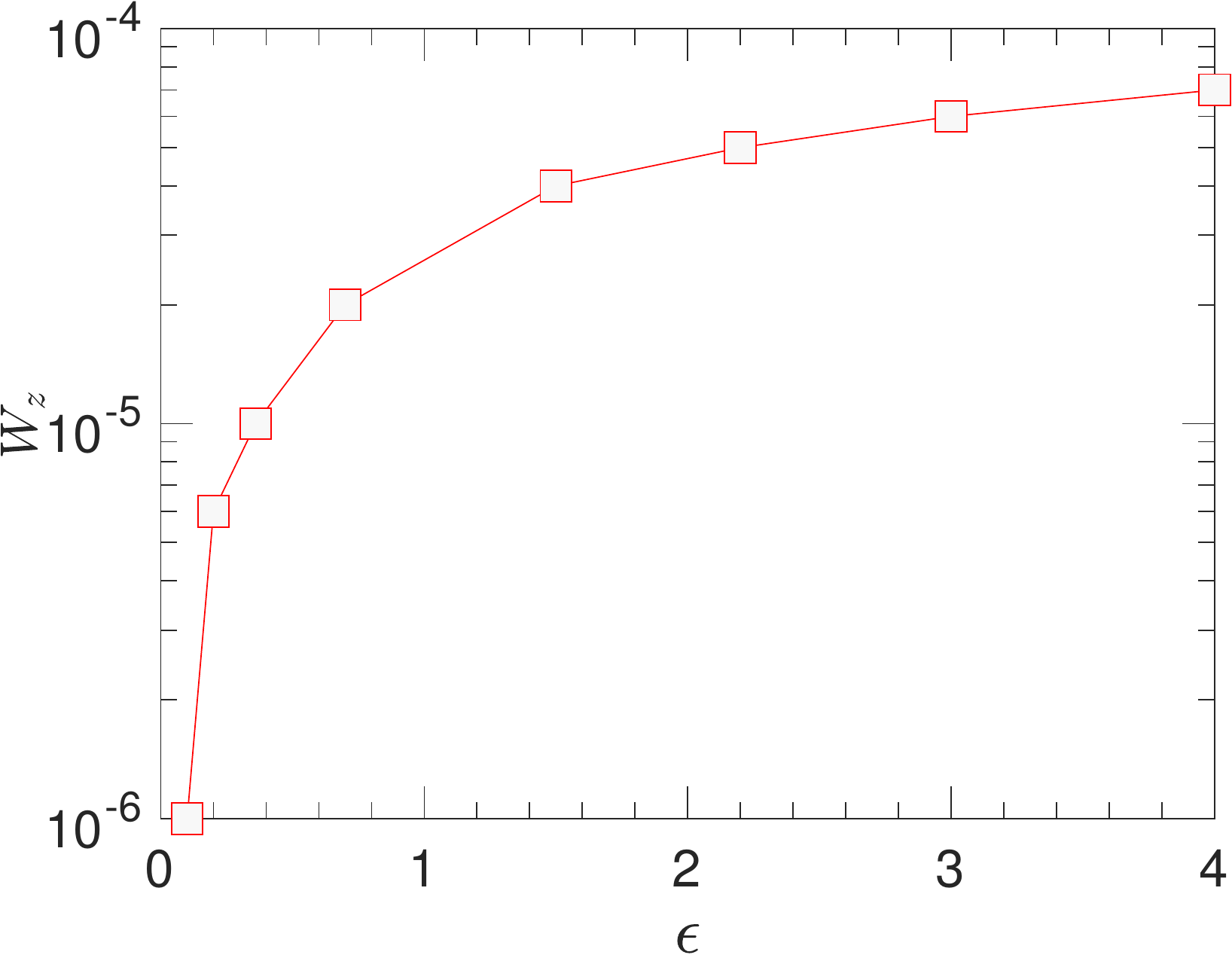}
  \end{overpic}
  \caption{Critical achoring strength on the top and bottom plates, for which the escaped configurations lose their stabilities as a function of $\epsilon$. }\label{ES_Height}
\end{figure}


\section{Summary}
\label{sec:summary}

In a batch of papers \cite{kralj2014order}, \cite{canevari2017order} and \cite{wang2018order}, the authors study WORS-type solutions or critical points of the LdG free energy on square domains, and WORS-type solutions have a constant eigenframe with a distinct diagonal defect line connecting the four square vertices. It is natural to ask if WORS-type solutions are relevant for 3D domains or if they are a 2D-artefact. Our essential finding in this paper is to show that the WORS is a LdG critical point for 3D wells with a square cross-section and experimentally relevant tangent boundary conditions on the lateral surfaces, for arbitrary well height, with both natural boundary conditions and realistic surface energies on the top and bottom surfaces. In fact, for sufficiently small $\lambda$ - the size of the square cross-section, the WORS is the global LdG minimizer for these 3D problems, exemplifying the 3D relevance of WORS-type solutions for all temperatures below the nematic supercooling temperature. 

We also numerically demonstrate the existence of stable mixed 3D solutions with two different diagonal profiles on the top and bottom well surfaces, for wells with sufficiently large $\epsilon$ and $\lambda$. These are again interesting from an applications point of view and are 3D solutions that are not covered by a purely 2D study. It is interesting to see that whilst the BD solution is an unstable LdG critical point on a 2D square domain, it interpolates between the two distinct diagonal profiles for a stable mixed 3D solution. Further work will be based on a study of truly 3D solutions that are not $z$-invariant and if they can related to the 2D solutions on squares reported in previous work i.e. can we use the zoo of 2D LdG critical points on a square domain reported in \cite{robinson2017molecular, wang2018order} to construct exotic 3D solutions on a 3D square well? This will be of substantial mathematical and applications-oriented interest.

\section{Acknowledgements}

G.C.'s  research  was supported by the Basque Government through the BERC 2018-2021 program; by the Spanish Ministry of Science, Innovation and Universities: BCAM Severo Ochoa accreditation SEV-2017-0718; and by the Spanish Ministry of Economy and Competitiveness: MTM2017-82184-R. A.M. was supported by fellowships EP/J001686/1 and EP/J001686/2, is supported by an OCIAM Visiting Fellowship and the Keble Advanced Studies Centre. Y.W. would like to thank Professor Chun Liu, for his constant support and helpful advice.

\bibliographystyle{plain}
\bibliography{LC}

\begin{thebibliography}{10}

\bibitem{bahadur1984liquid}
Birendra Bahadur.
\newblock Liquid crystal displays.
\newblock {\em Molecular Crystals and Liquid Crystals}, 109(1):3--93, 1984.

\bibitem{Bethuel1992}
F.~Bethuel, H.~Brezis, B.~D. Coleman, and F.~H{\'e}lein.
\newblock Bifurcation analysis of minimizing harmonic maps describing the
  equilibrium of nematic phases between cylinders.
\newblock {\em Archive for Rational Mechanics and Analysis}, 118(2):149--168,
  Jun 1992.

\bibitem{canevari2017order}
Giacomo Canevari, Apala Majumdar, and Amy Spicer.
\newblock Order reconstruction for nematics on squares and hexagons: A
  landau--de gennes study.
\newblock {\em SIAM Journal on Applied Mathematics}, 77(1):267--293, 2017.

\bibitem{Sluckin}
Clive A CroxtonClive~A Croxton.
\newblock Orientational wetting and related phenomena in nematics.
\newblock In C.A. Croxton, editor, {\em Fluid Interfacial Phenomena}, pages
  215--253. John Wiley, Chichester, 1986.

\bibitem{dangfifepeletier}
H.~Dang, P.~C. Fife, and L.~A. Peletier.
\newblock Saddle solutions of the bistable diffusion equation.
\newblock {\em Z. Angew. Math. Phys.}, 43(6):984--998, 1992.

\bibitem{gartlanddavis}
Timothy~A Davis and Eugene~C Gartland~Jr.
\newblock Finite element analysis of the landau--de gennes minimization problem
  for liquid crystals.
\newblock {\em SIAM J. Numer. Anal.}, 35(1):336--362, 1998.

\bibitem{dg}
P.~G. de~Gennes and J.~Prost.
\newblock {\em The Physics of Liquid Crystals}.
\newblock Clarendon Press, Oxford, 1974.

\bibitem{golovaty2017dimension}
Dmitry Golovaty, Jos{\'e}~Alberto Montero, and Peter Sternberg.
\newblock Dimension reduction for the landau-de gennes model on curved nematic
  thin films.
\newblock {\em Journal of Nonlinear Science}, 27(6):1905--1932, 2017.

\bibitem{Ignatetal}
R.~Ignat, L.~Nguyen, V.~Slastikov, and A.~Zarnescu.
\newblock Instability of point defects in a two-dimensional nematic liquid
  crystal model.
\newblock {\em Annales de l'Institut Henri Poincare (C) Non Linear Analysis},
  33(4):1131 -- 1152, 2016.

\bibitem{kralj2014order}
Samo Kralj and Apala Majumdar.
\newblock Order reconstruction patterns in nematic liquid crystal wells.
\newblock {\em Proc. R. Soc. A}, 470(2169):20140276, 2014.

\bibitem{kralj1993stability}
Samo Kralj and Slobodan {\v{Z}}umer.
\newblock The stability diagram of a nematic liquid crystal confined to a
  cylindrical cavity.
\newblock {\em Liquid Crystals}, 15(4):521--527, 1993.

\bibitem{kusumaatmaja2015free}
Halim Kusumaatmaja and Apala Majumdar.
\newblock Free energy pathways of a multistable liquid crystal device.
\newblock {\em Soft matter}, 11(24):4809--4817, 2015.

\bibitem{lagerwall2012new}
Jan~PF Lagerwall and Giusy Scalia.
\newblock A new era for liquid crystal research: applications of liquid
  crystals in soft matter nano-, bio-and microtechnology.
\newblock {\em Current Applied Physics}, 12(6):1387--1412, 2012.

\bibitem{lamy2014}
X.~Lamy.
\newblock Bifurcation analysis in a frustrated nematic cell.
\newblock {\em J. Nonlinear Sci.}, 24(6):1197--1230, 2014.

\bibitem{lewis2014colloidal}
Alexander~H Lewis, Ioana Garlea, Jos{\'e} Alvarado, Oliver~J Dammone, Peter~D
  Howell, Apala Majumdar, Bela~M Mulder, MP~Lettinga, Gijsje~H Koenderink, and
  Dirk~GAL Aarts.
\newblock Colloidal liquid crystals in rectangular confinement: theory and
  experiment.
\newblock {\em Soft Matter}, 10(39):7865--7873, 2014.

\bibitem{luo2012multistability}
Chong Luo, Apala Majumdar, and Radek Erban.
\newblock Multistability in planar liquid crystal wells.
\newblock {\em Phys. Rev. E}, 85(6):061702, 2012.

\bibitem{ejam2010}
A.~Majumdar.
\newblock Equilibrium order parameters of nematic liquid crystals in the
  {L}andau-de {G}ennes theory.
\newblock {\em European J. Appl. Math.}, 21(2):181--203, 2010.

\bibitem{majumdar2010landau}
Apala Majumdar and Arghir Zarnescu.
\newblock Landau--de gennes theory of nematic liquid crystals: the oseen--frank
  limit and beyond.
\newblock {\em Archive for rational mechanics and analysis}, 196(1):227--280,
  2010.

\bibitem{newtonmottram}
N.~J. Mottram and C.~Newton.
\newblock Introduction to {Q}-tensor theory.
\newblock Technical Report~10, Department of Mathematics, University of
  Strathclyde, 2004.

\bibitem{OsipovHess}
M.~A. Osipov and S.~Hess.
\newblock Density functional approach to the theory of interfacial properties
  of nematic liquid crystals.
\newblock {\em The Journal of Chemical Physics}, 99(5):4181--4190, 1993.

\bibitem{ravnik2009landau}
Miha Ravnik and Slobodan {\v{Z}}umer.
\newblock Landau--de gennes modelling of nematic liquid crystal colloids.
\newblock {\em Liquid Crystals}, 36(10-11):1201--1214, 2009.

\bibitem{robinson2017molecular}
Martin Robinson, Chong Luo, Patrick~E Farrell, Radek Erban, and Apala Majumdar.
\newblock From molecular to continuum modelling of bistable liquid crystal
  devices.
\newblock {\em Liquid Crystals}, 44(14-15):2267--2284, 2017.

\bibitem{SenSullivan}
A.~K. Sen and D.~E. Sullivan.
\newblock Landau-de {G}ennes theory of wetting and orientational transitions at
  a nematic-liquid--substrate interface.
\newblock {\em Phys. Rev. A}, 35:1391--1403, Feb 1987.

\bibitem{tsakonas2007multistable}
C~Tsakonas, A.~J. Davidson, C.~V. Brown, and N.~J. Mottram.
\newblock Multistable alignment states in nematic liquid crystal filled wells.
\newblock {\em Appl. Phys. Lett.}, 90(11):111913, 2007.

\bibitem{Walton2018}
J.~Walton, N.~J. Mottram, and G.~McKay.
\newblock Nematic liquid crystal director structures in rectangular regions.
\newblock {\em Phys. Rev. E}, 97:022702, Feb 2018.

\bibitem{wang2018order}
Yiwei Wang, Giacomo Canevari, and Apala Majumdar.
\newblock Order reconstruction for nematics on squares with isotropic
  inclusions: A landau-de gennes study.
\newblock {\em accepted, SIAM Journal on Applied Mathematics}, 2019.

\bibitem{wang2017topological}
Yiwei Wang, Pingwen Zhang, and Jeff~ZY Chen.
\newblock Topological defects in an unconfined nematic fluid induced by single
  and double spherical colloidal particles.
\newblock {\em Physical Review E}, 96(4):042702, 2017.

\bibitem{wright1999numerical}
Stephen Wright and Jorge Nocedal.
\newblock {\em Numerical Optimization}, volume~35.
\newblock Springer, 1999.

\end{thebibliography}






\end{document}